\newif\iffinal
\finaltrue	

\documentclass[11pt,a4paper,reqno, oneside]{amsart}
\usepackage[foot]{amsaddr}

\usepackage[utf8]{inputenc}
\usepackage{amsfonts, amsmath, amssymb, amsthm, color, enumerate, graphicx, mathtools, tikz, hyperref, relsize}
\usepackage[numbers,sort]{natbib}

\usepackage[sc]{mathpazo}
\usepackage[T1]{fontenc}
\usepackage{eulervm}
\usepackage[sc]{mathpazo}
\usepackage[a4paper,includeheadfoot, total={6.5in,10in}]{geometry}

\usepackage{environ}
\NewEnviron{eq}{%
\begin{equation}\begin{split}
  \BODY
\end{split}\end{equation}
}

\usepackage{subcaption} 

\usepackage{tikz-3dplot}
\usetikzlibrary{decorations.text, decorations.markings, mindmap,trees,shadows, decorations.pathreplacing}
\tikzstyle{doublearr}=[latex-latex,red, line width=0.5pt]
\tikzstyle{doublearr2}=[latex-latex,green!80!black, line width=0.5pt]
\usepackage{xcolor}

\makeatletter
\renewcommand{\tocsection}[3]{%
  \indentlabel{\@ifnotempty{#2}{\bfseries\ignorespaces#1 #2\quad}}\bfseries#3}
\def\l@subsection{\@tocline{2}{0pt}{2.5pc}{5pc}{}}
\renewcommand\tocchapter[3]{%
  \indentlabel{\@ifnotempty{#2}{\ignorespaces#2.\quad}}#3%
}
\newcommand\@dotsep{4.5}
\def\@tocline#1#2#3#4#5#6#7{\relax
  \ifnum #1>\c@tocdepth 
  \else
    \par \addpenalty\@secpenalty\addvspace{#2}%
    \begingroup \hyphenpenalty\@M
    \@ifempty{#4}{%
      \@tempdima\csname r@tocindent\number#1\endcsname\relax
    }{%
      \@tempdima#4\relax
    }%
    \parindent\z@ \leftskip#3\relax \advance\leftskip\@tempdima\relax
    \rightskip\@pnumwidth plus1em \parfillskip-\@pnumwidth
    #5\leavevmode\hskip-\@tempdima{#6}\nobreak
    \leaders\hbox{$\m@th\mkern \@dotsep mu\hbox{.}\mkern \@dotsep mu$}\hfill
    \nobreak
    \hbox to\@pnumwidth{\@tocpagenum{#7}}\par
    \nobreak
    \endgroup
  \fi}
\makeatother
\AtBeginDocument{%
\makeatletter
\expandafter\renewcommand\csname r@tocindent0\endcsname{0pt}
\makeatother
}
\def\l@subsection{\@tocline{2}{0pt}{2.5pc}{5pc}{}}

\newcommand{\QQ}{\mathbf{Q}}
\newcommand{\bQ}{\bar{Q}}

\newcommand{\N}{\mathbb{N}}
\newcommand{\R}{\mathbb{R}}

\newcommand{\dif}{\ensuremath{\mbox{d}}}  
\newcommand{\e}{\mathrm{e}}

\newcommand{\pto}{\ensuremath{\xrightarrow{\mathbb{P}}}}  
\newcommand{\dto}{\ensuremath{\xrightarrow{\mathrm{d}}}}  

\newcommand{\expt}{\mathbb{E}}
\newcommand{\prob}{\mathbb{P}}

\newtheorem{theorem}{Theorem}
\newtheorem{corollary}[theorem]{Corollary}
\newtheorem{lemma}[theorem]{Lemma}
\newtheorem{proposition}[theorem]{Proposition}
\newtheorem{remark}{Remark}
\newtheorem{claim}{Claim}

\let\plainqed\qedsymbol
\newcommand{\claimqed}{$\lrcorner$}
\newenvironment{claimproof}{\begin{proof}\renewcommand{\qedsymbol}{\claimqed}}{\end{proof}\renewcommand{\qedsymbol}{\plainqed}}

\hypersetup{
 colorlinks=true,
 linkcolor=blue,          
 citecolor=red,       
 filecolor=red,   
 urlcolor=red, 
 pdftitle={},
 pdfauthor={},
 pdfcreator={Sayan Banerjee, Debankur Mukherjee},
 pdfsubject={},
 pdfkeywords={},
 linktocpage=true
}

\usepackage{fancyhdr}
 
\tikzstyle{mybox} = [draw=red, fill=yellow!20, thick, minimum height=.4cm,
    rectangle, rounded corners]
\tikzstyle{fancytitle} =[fill=blue, text=white]
\usetikzlibrary{mindmap,trees,shadows}

\numberwithin{equation}{section}
\numberwithin{theorem}{section}

\tikzstyle{mybox} = [draw=black, thick, minimum height=.6cm,
    rectangle,text centered]
\tikzstyle{fancytitle} =[fill=blue, text=white]

\title[Join-the-Shortest Queue Diffusion Limit in Halfin-Whitt Regime]{Join-the-Shortest Queue Diffusion Limit in Halfin-Whitt Regime: Sensitivity on the Heavy-traffic Parameter}

\author{Sayan Banerjee}
\address{University of North Carolina, Chapel Hill}
\email{sayan@email.unc.edu}
\author{Debankur Mukherjee}
\address{Brown University}
\email{debankur\_mukherjee@brown.edu}

\date{\today}

\begin{document}

\begin{abstract}
   Consider a system of $N$~parallel single-server queues with unit-exponential service time distribution and a single dispatcher where tasks arrive as a Poisson process of rate $\lambda(N)$. When a task arrives, the dispatcher assigns it to one of the servers according to the Join-the-Shortest Queue (JSQ) policy. Eschenfeldt and Gamarnik (2018)~\cite{EG15} identified a novel limiting diffusion process that arises as the weak-limit of the appropriately scaled occupancy measure of the system under the JSQ policy in the Halfin-Whitt regime, where $(N - \lambda(N)) / \sqrt{N} \to \beta > 0$ as $N \to \infty$. The analysis of this diffusion goes beyond the state of the art techniques, and even proving its ergodicity is non-trivial, and was left as an open question. Recently, exploiting a generator expansion framework via the Stein's method, Braverman (2018)~\cite{Braverman18} established its exponential ergodicity, and adapting a regenerative approach, Banerjee and Mukherjee (2018)~\cite{BM18} analyzed the tail properties of the stationary distribution and path fluctuations of the diffusion.
    
    However, the analysis of the bulk behavior of the stationary distribution, viz., the moments, remained intractable until this work. In this paper, we perform a thorough analysis of the bulk behavior of the stationary distribution of the diffusion process, and discover that it exhibits different qualitative behavior, depending on the value of the heavy-traffic parameter $\beta$. Moreover, we obtain precise asymptotic laws of the centered and scaled steady state distribution, as $\beta$ tends to 0 and $\infty$.
Of particular interest, we also establish a certain intermittency phenomena in the $\beta\to \infty$ regime and a surprising distributional convergence result in the $\beta\to 0$ regime.
    \\


\end{abstract}

\maketitle 

 \tableofcontents

\section{Introduction}
\subsection{Background and motivation.}
For any $\beta>0$, consider the following diffusion process with state space $(-\infty, 0] \times (0, \infty)$
\begin{equation}\label{eq:diffusionjsq}
\begin{split}
Q_1(t) &= Q_1(0) + \sqrt{2} W(t) - \beta t +
\int_0^t (- Q_1(s) + Q_2(s)) \dif s - L(t), \\
Q_2(t) &= Q_2(0) + L(t) - \int_0^t Q_2(s)\dif s
\end{split}
\end{equation}
for $t \geq 0$, where $W$ is the standard Brownian motion, $L$ is
the unique nondecreasing nonnegative process in~$D_\R[0,\infty)$ satisfying
$\int_0^\infty \mathbf{1}_{[Q_1(t) < 0]} \dif L(t) = 0$,
and $(Q_1(0), Q_2(0)) \in (-\infty, 0] \times [0, \infty)$.
In this paper, we consider the stationary distribution of the above diffusion process.
In particular, we analyze the bulk behavior of the steady state for all fixed $\beta$ sufficiently large and small, and identify its scaling behavior as $\beta\to 0$ and $\beta\to\infty$.

In the context of task allocation in many-server systems, the diffusion process in~\eqref{eq:diffusionjsq} arises as the weak limit of the sequence of scaled occupancy measures of systems under the classical Join-the-Shortest Queue (JSQ) policy, as the system size (number of servers in the system) becomes large. 
Specifically, consider a system with $N$~parallel identical single-server queues and a single dispatcher.
Tasks with unit-mean exponential service requirements arrive at the
dispatcher as a Poisson process of rate $\lambda(N)$,
and are instantaneously forwarded to one of the servers with the shortest queue length (ties are broken arbitrarily).
For $t\geq 0$, let
$$\QQ^N(t): =
\left(Q_1^N(t), Q_2^N(t), \dots\right)$$ 
denote the
system occupancy measure, where $Q_i^N(t)$ is the number of servers
under the JSQ policy with a queue length of~$i$ or larger,
at time~$t$, including the possible task in service, $i = 1, 2,\dots$.
Note that due to exchangeability of the servers and the Markovian service requirements, $\QQ^N(\cdot)$ is a Markov process.
In fact, it can also be seen that if $\lambda(N)<N$ (i.e., load per server $\lambda(N)/N$ is less than 1), then $\QQ^N$ is positive recurrent and has a unique stationary distribution.
Now consider an asymptotic regime where the number of servers grows large, and additionally assume that
$$\frac{N - \lambda(N)}{ \sqrt{N}} \to \beta\quad\text{as}\quad N \to \infty$$
for some positive coefficient $\beta > 0$, i.e., the load per server approaches unity as $1 - \beta / \sqrt{N}$. 
In terms of the aggregate traffic load and total service capacity, this scaling corresponds to the so-called Halfin-Whitt heavy-traffic regime which was introduced in the seminal paper~\cite{HW81} and has been extensively studied since. 
The set-up in~\cite{HW81}, as well as the numerous model extensions in the literature (see~\cite{GG13a, GG13b, HW81, LK11, LK12, FKL14, LMZ17}, and the references therein), primarily considered a single centralized queue and server pool (M/M/N), rather than a scenario with parallel queues.
Eschenfeldt and Gamarnik~\cite{EG15} initiated the study of the scaling behavior for parallel-server systems in the Halfin-Whitt heavy-traffic regime.
Define the centered and scaled system occupancy measures as $\bar{\QQ}^N(t) =
\big(\bar{Q}_1^N(t), \bar{Q}_2^N(t), \dots\big)$, with 
$$\bar{Q}_1^N(t) = - \frac{N-Q_1^N(t)}{ \sqrt{N}}\leq 0,\qquad \bar{Q}_i^N(t) =\frac{ Q_i^N(t)}{\sqrt{N}}\geq 0,\quad i = 2, 3\dots.$$
The reason why $Q_1^N(t)$ is centered around~$N$ while $Q_i^N(t)$,
$i = 2, \dots$, are not, is because the fraction of servers at time $t$ with a queue
length of exactly one tends to $1$, whereas the fraction of servers with a queue length of two or more tends to zero as $N\to\infty$. 
For each fixed~$N$, $\bar{\QQ}^N$ is a positive recurrent continuous time Markov chain, and has a stationary distribution as $t \rightarrow \infty$. 
Denote by $\bar{\QQ}^N(\infty)$ a random variable distributed as the steady state of the process $\bar{\QQ}^N(t)$.
Assuming $(\bQ_i^N(0))_{i\geq 1} \to (Q_i(0)))_{i\geq 1}$ with  $Q_i(0)= 0$ for $i \ge 3$, it was shown by Eschenfeldt and Gamarnik~\cite{EG15} that on any finite time interval $[0,T]$,
the sequence of processes $\big\{(\bar{Q}_1^N(t), \bar{Q}_2^N(t), \ldots)\big\}_{0 \le t \le T}$ 
converges weakly to the limit
 $\big\{(Q_1(t), Q_2(t),\ldots)\big\}_{0 \le t \le T}$, where $(Q_1, Q_2)$ is given by~\eqref{eq:diffusionjsq} and $Q_i(\cdot)\equiv 0$ for~$i\geq 3$. 
 Subsequently, a broad class of other schemes were shown to exhibit the same scaling behavior in this regime~\cite{MBLW16-3, MBLW16-1, MBL17}.
 See~\cite{BBLM18} for a recent survey.

In all the above works, the convergence of the scaled occupancy measure was established in the transient regime on any finite time interval.
Long time asymptotic properties of the new diffusion process in~\eqref{eq:diffusionjsq} thus discovered in~\cite{EG15} is technically hard to analyze.
In fact, even establishing its ergodicity is non-trivial and was left as an open question in~\cite{EG15}.
The tightness of the diffusion-scaled occupancy measure under the JSQ policy, exponential ergodicity of the diffusion process, and the interchange of limits were established by Braverman~\cite{Braverman18} via a sophisticated generator expansion framework using the Stein's method.
There, it was shown that the steady state of the $N$-server system $\bar{\QQ}^N(\infty)$ converges weakly to $(Q_1(\infty), Q_2(\infty), 0, 0,\ldots)$ as $N\to\infty$, where $(Q_1(\infty), Q_2(\infty))$ is distributed as the steady state of the diffusion process $(Q_1, Q_2)$.
Thus, the steady state of the diffusion process in~\eqref{eq:diffusionjsq} captures the asymptotic behavior of large-scale systems under the JSQ policy.
Recently, Banerjee and Mukherjee~\cite{BM18} considered the tail asymptotics of $(Q_1(\infty), Q_2(\infty))$, and established that for each fixed $\beta>0$, $Q_1(\infty)$ has a Gaussian tail and $Q_2(\infty)$ has an exponential tail.
A high-level heuristic for such tail behavior is that for any fixed $\beta>0$, when $-Q_1$ is large enough, it behaves as an Ornstein-Uhlenbeck (OU) process (giving rise to the Gaussian tail for $Q_1$), and when $Q_2$ is large it behaves as a Brownian motion with a negative drift (giving rise to the exponential tail for $Q_2$).
However, in order to characterize the bulk behavior of the stationary distribution, such as its mean, one needs precise control over the diffusion paths not only when $-Q_1$ or $Q_2$ is large, but also near the origin. 


\subsection{Key contributions and our approach}
In this paper, we perform a thorough analysis of the bulk behavior of the stationary distribution
and, quite surprisingly, find that its qualitative behavior is sensitive to the heavy-traffic parameter $\beta$.
In particular, we show that $$\e^{-C_1\beta^2} \le \mathbb{E}_{\pi}\left(Q_2(\infty)\right) \le \e^{-C_2\beta^2}$$ for all large enough $\beta$ and $$C_1\beta^{-1} \le \mathbb{E}_{\pi}\left(Q_2(\infty)\right) \le C_2\beta^{-1}$$ for all small enough $\beta$, where $C_1, C_2$ are positive constants that do not depend on $\beta$.
Moreover, $Q_2$ exhibits an {\em intermittency phenomenon} for large $\beta$ in the sense that most of the steady-state mass of $Q_2$ is concentrated in the region $(0, \e^{-\e^{\mathcal{C}^*\beta^2}})$, i.e., 
$$\prob(Q_2(\infty) \ge \e^{-\e^{\mathcal{C}^*\beta^2}}) \le \e^{-D\beta^2},$$ 
for positive constants $\mathcal{C}^*, D$ (that do not depend on $\beta$). However, as we just saw, the expected value decays only exponentially in $\beta^2$.
This indicates that in the steady-state dynamics, $Q_2$ usually remains very close to zero, but in the rare events when it becomes large, it takes a long time to become small again. 
A more detailed discussion on this behavior is given in Remark~\ref{rem:intermit}.
We also show that $Q_1(\infty) + \beta$ converges weakly to a standard normal distribution and $Q_2(\infty)$ converges to zero in $L^p$ for any $p > 0$ as $\beta \rightarrow \infty$. Furthermore, as $\beta \rightarrow 0$, the random variable $\beta Q_2(\infty)$ converges weakly to a $\operatorname{Gamma}(2)$ distribution (i.e., sum of two independent unit-mean exponential random variables) and $Q_1(\infty)$ converges to zero in $L^p$ for any $p > 0$.
The distributional convergence result for $\beta Q_2(\infty)$ is quite surprising, and reveals an important feature about large-scale parallel-server systems, namely, although JSQ achieves economies of scale in the Halfin-Whitt regime, it is a factor 2 worse compared to the completely pooled (or centralized queueing) system. 
This is further discussed in Remark~\ref{rem:mmncomp}.

Understanding bulk behavior of stationary distributions of diffusion processes has always been a challenging problem. 
State-of-the-art probabilistic tools to analyze stationary distributions \cite{ABD01, BL07,DW94, HM09} identify a large enough `small set' in the state space along with a Lyapunov type drift criterion which gives good control on the exponential moments of return times to the small set \cite{MT93}. 
These exponential moment bounds translate to exponential ergodicity as well as exponential tail bounds for the stationary measure. 
However, this approach sheds little light on the behavior of the diffusion paths inside the small set, which essentially determines the bulk behavior of the stationary distribution. 
In this article, we achieve control inside the small set by
exploiting an idea of using the theory of regenerative processes, which was introduced in \cite{BM18} (see Chapter 10 of \cite{Thorisson}, also~\cite{BBD15}, for its usage in a somewhat related scenario).
In this approach, we identify regeneration times in the diffusion path (random times when the diffusion starts afresh, see Section~\ref{sec:reg} for further details) and performing a detailed analysis of the excursions between two successive regeneration times. 
A key idea used in the analysis and control of these excursions is to define
various stopping times and bound them by the hitting times of some (reflected) Brownian motion with appropriate drift or (reflected) OU process, which are analytically more tractable.
The construction of these bounding processes depends on understanding the specific dynamics of the process in different parts of the state space, and in particular, on whether the heavy-traffic parameter $\beta$ is large or small.
The hitting time estimates provide key insights into how the behavior of the process changes depending on the value of $\beta$.
Consequently, we uncover the sensitivity of the stationary distribution on $\beta$.

\subsection{Organization and Notation}
Rest of the article is arranged as follows.
In Section~\ref{sec:main} we state the main results.
Section~\ref{sec:reg} contains a brief overview of the regenerative approach as introduced in~\cite{BM18}.
The proofs of the main results in the large-$\beta$ regime is presented in Section~\ref{sec:analysis-large}, while proofs of many intermediate lemmas in this regime are deferred till Appendices~\ref{app:large-aux} and~\ref{app:lemma4.8}. 
The proofs of the main results in the small-$\beta$ regime is presented in Section~\ref{sec:small-beta}, while proofs of many intermediate lemmas in this regime are deferred till Appendices~\ref{app:small-aux} and~\ref{app:lem5.6}. 

For any two real numbers $x, y$, we denote by $x\vee y$ and $x\wedge y$, $\max\{x,y\}$ and $\min\{x,y\}$, respectively.
We adopt the usual notations to describe asymptotic comparisons: For two functions $f, g:\N\to\R$,  we say $f(n) = O(g(n)), \Omega(g(n)), \Theta(g(n)), o(g(n))$, and $\omega(g(n))$ if for some fixed positive constants $c_1$ and $c_2$,
$f(n) \leq c_1 g(n)$, $f(n) \geq c_2 g(n)$, $c_2g(n) \leq f(n) \leq c_1g(n)$, $f(n)/g(n) \to 0$ as $n\to\infty$, and $f(n)/g(n) \to \infty$ as $n\to\infty$, respectively.
Convergence in distribution and in probability are denoted by `$\dto$' and `$\pto$', respectively.

\section{Main results}\label{sec:main}

In this section we will state the main results and discuss their ramifications. 
Recall the diffusion process $\{(Q_1(t), Q_2(t))\}_{t \ge 0}$ defined by Equation~\eqref{eq:diffusionjsq}.
As mentioned in the introduction, it is known~\cite{Braverman18} that for any $\beta>0$, $(Q_1, Q_2)$ is an ergodic continuous-time Markov process.
Let $(Q_1(\infty), Q_2(\infty))$ denote a random variable distributed as the unique stationary distribution~$\pi$ of the process.
In Subsections~\ref{ssec:large} and~\ref{ssec:small} we will consider $\pi$ for all $\beta$ large and small enough, respectively.

\subsection{Large-$\beta$ regime and asymptotics}\label{ssec:large}

All the results stated in this subsection are proved in Subsection~\ref{ssec:proof-large}.
Our first main result concerns the steady state of $Q_2$.
As we will see, although the tail of the steady state of $Q_2$ decays exponentially and the exponent is linear in $\beta$, the prefactor decays exponentially in $\beta^2$. 
\begin{theorem}\label{largestat}
There exist $\beta_0 \ge 1$ and positive constants $C_1^+, C_2^+, C_1^-, C_2^-, \mathcal{C}_1^-, \mathcal{C}_2^-$, such that for all $\beta \ge \beta_0$ and $y \ge 4\beta \e^{-\mathcal{C}_1^-\e^{\mathcal{C}_2^-\beta^2}}$, 
\begin{align*}
\pi(Q_2(\infty) \ge y) &\le C_1^+\e^{-C_2^+\beta^2}\left(1+\log\left(\frac{1}{\beta y}\right) \mathbf{1}_{[y \le \beta^{-1}]}\right)\e^{-C_2^+ \beta y},\\
\pi(Q_2(\infty) \ge y) &\ge C_1^-\e^{-C_2^-\beta^2}\left(1+\log\left(\frac{1}{\beta y}\right) \mathbf{1}_{[y \le \beta^{-1}]}\right)\e^{-C_2^- \beta y}.
\end{align*}
In particular, for any $p>0$, $Q_2(\infty)$ converges in $L^p$ to zero as $\beta \rightarrow \infty$.
\end{theorem}
Theorem \ref{largestat} gives detailed characterization of the shape of the stationary distribution of $Q_2$. 
It not only captures the tail behavior, but also characterizes the distribution near zero. 
It is worthwhile to point out that Theorem \ref{largestat} provides several key insights that cannot be captured by only the tail asymptotics. 
Consequently, the value of $y$ for which the tail behavior kicks in and the precise form of the prefactor in the tail probabilities become crucial in understanding the bulk behavior of the steady state of $Q_2$.
This is elaborated in Remark~\ref{rem:intermit} below.

\begin{remark}[{Condensation of steady state and intermittency}]\label{rem:intermit}
\normalfont
Observe that Theorem~\ref{largestat} can be used to obtain sharp bounds on the steady-state mean of $Q_2$.
In fact, it shows that despite having an exponentially decaying tail, $Q_2(\infty)$ exhibits a \textit{condensation of steady-state mass}, namely, most of the steady-state mass of $Q_2$ is concentrated in the region $(0, \e^{-\e^{C\beta^2}})$ although the mean is of the order of $\e^{-C'\beta^2}$ (where $C, C'$ are positive constants not depending on $\beta$). 
This indicates an \textit{intermittency} phenomenon, i.e., at most times, $Q_2$ is very close to zero, but in the rare occasions when $Q_2$ gets to an appreciable positive level, it takes a while to get back to near zero. 
From a high level, this can be understood as follows.
First note that for any $\beta>0$, $\expt(Q_1(\infty)) = -\beta$, i.e., $Q_1$ fluctuates around $-\beta$.
Also, when $Q_2$ is small, $Q_1$ behaves as an OU process with mean reverting towards $-\beta$.
Now, when $\beta$ is large, usually $Q_2$ is very small (of order $\e^{-\e^{C\beta^2}}$), and the rare occasions when $Q_2$ gets to an appreciable positive level are precisely the times when $Q_1$ hits 0 and gathers some local time.
In turn, this can be thought of as hitting times of an OU process to level $\beta$, which is exponential in $\beta^2$.
Further, since the rate of decrease of $Q_2$ is proportional to itself, whenever $Q_2$ becomes much higher than usual it takes exponentially long time to return to the level $\e^{-\e^{C\beta^2}}$. 
This explains the condensation and intermittency of $Q_2$.
\end{remark}
The observations in Remark~\ref{rem:intermit} are formalized in the following corollary.
We will use $\mathbb{E}_{\pi}$ to denote the expectation with respect to the stationary distribution $\pi$.
\begin{corollary}\label{intermit}
There exist $\beta_0 \ge 1$ and positive constants $C_1, C_2, \mathcal{C}^*, D$ such that for all $\beta \ge \beta_0$,
\begin{align*}
\e^{-C_1\beta^2} \le \mathbb{E}_{\pi}\left(Q_2(\infty)\right) \le \e^{-C_2\beta^2},\\
\pi\Big(Q_2(\infty) \ge \e^{-\e^{\mathcal{C}^*\beta^2}}\Big) \le \e^{-D\beta^2}.
\end{align*}
\end{corollary}
From Theorem \ref{largestat} and Corollary \ref{intermit}, it is clear that in the large-$\beta$ regime, $Q_2$ spends most times near zero and \eqref{eq:diffusionjsq} indicates that during these times, $Q_1+ \beta$ behaves like an Ornstein-Uhlenbeck process reflected downwards at $\beta$. 
So, we would expect $Q_1 + \beta$ to have a steady-state distribution close to a standard Gaussian for large $\beta$. 
We formalize this notion by proving in Theorem \ref{betainfty} that $Q_1(\infty) + \beta$ converges in distribution to the standard normal distribution as $\beta \rightarrow \infty$.
Proposition~\ref{Q1fin} provides moment bounds for $Q_1(\infty)$, which is of independent interest.
Proposition~\ref{Q1fin} will be used to prove Theorem \ref{betainfty}.

\begin{proposition}\label{Q1fin}
For any $n \ge 1$,
$$
\limsup_{\beta \rightarrow \infty}\mathbb{E}_{\pi}((Q_1(\infty) + \beta)^{2n}) < \infty.
$$
\end{proposition}

\begin{theorem}\label{betainfty}
As $\beta \rightarrow \infty$, $Q_1(\infty) + \beta$ converges weakly to the standard normal distribution.
\end{theorem}

\subsection{Small-$\beta$ regime and asymptotics}\label{ssec:small}

Note that when $Q_2$ is small, $Q_1$ behaves like a reflected OU process but when $Q_2$ is large, it increases the drift of $Q_1$ towards zero and hence $Q_1$ behaves roughly like a reflected Brownian motion with a large drift. 
As indicated by Theorem \ref{smallbetathm} below, $Q_2$ has steady state mean of the order of $\beta^{-1}$ and hence, it spends considerable time taking large values, which draws $Q_1$ towards zero. 
This is also reflected by the fact that the steady-state mean of $Q_1$ is $-\beta$. The technical challenge that arises is to patch up the different behaviors of $Q_1$ for small and large $Q_2$ and to produce an estimate that unifies these effects. 
We achieve this in Theorem \ref{Q1statsmall} where we provide an upper bound on the steady-state upper tails of $-Q_1$ as a mixture of a Gaussian (from the OU behavior for small $Q_2$) and an exponential with mean $\sqrt{\beta}$ (from the reflected `Brownian motion with drift' behavior for large $Q_2$). 
Theorems \ref{smallbetathm} and \ref{Q1statsmall} will be crucial in proving the distributional convergence result for $Q_2(\infty)$ in Theorem \ref{betazero}, namely $\beta Q_2(\infty)$ converges in distribution to $\operatorname{Gamma}(2)$ as $\beta \rightarrow 0$. 
All the results stated in this subsection are proved in Subsection~\ref{ssec:proof-small}.

\begin{theorem}\label{smallbetathm}
There exist positive constants $M_0, C_s^{1+}, C_s^{2+}, C_s^{1-}, C_s^{2-}, \beta_s^*$ such that for all $\beta \le \beta_s^*$ and all $y \ge 8M_0\beta^{-1}$,
$$
C_s^{1-}\e^{-C_s^{2-}\beta y} \le \pi(Q_2(\infty) \ge y) \le C_s^{1+}\e^{-C_s^{2+}\beta y}.
$$
In particular,
$$
\bigg(\frac{C_s^{1-}\e^{-8C_s^{2-}M_0}}{C_s^{2-}}\bigg)\frac{1}{\beta} \le \mathbb{E}_{\pi}(Q_2(\infty)) \le \left(8M_0 + \frac{C_s^{1+}}{C_s^{2+}}\right)\frac{1}{\beta}.
$$
\end{theorem}
Theorem~\ref{smallbetathm} should be contrasted with Corollary~\ref{intermit} in terms of the dependence on $\beta$ in the small-$\beta$ regime.
Note how the dependence of the prefactor on $\beta$ crucially governs the bulk behavior.
Unlike the large-$\beta$ regime, the steady-state expectation of $Q_2$ tends to depend inversely on $\beta$.

The next theorem bounds the lower tail of $Q_1(\infty)$.
As mentioned earlier, it captures the two effects Gaussian and exponential, rising from the dynamics when $Q_2$ is small and large, respectively.

\begin{theorem}\label{Q1statsmall}
There exist $\beta'_0 \in (0,1)$, and positive constants $R,C,C', C"$ such that for all $\beta \le \beta'_0$ and all $x \ge 2\beta^{1/4}$,
\begin{equation*}
\pi(Q_1(\infty) \le -x) \le C\Big(\beta^{1/4}\e^{-(x-2\beta)^2/8} + \beta^2\e^{-C'\frac{x}{\sqrt{\beta}}}\Big)\mathbf{1}_{\left[2\beta^{1/4}\leq x \le \frac{R}{\beta}\log \left(\frac{1}{\beta}\right)\right]} + \e^{-C"x^2}\mathbf{1}_{\left[x > \frac{R}{\beta}\log \left(\frac{1}{\beta}\right)\right]}.
\end{equation*}
In particular, for every $n > 0$, $\mathbb{E}_{\pi}(|Q_1(\infty)|^n) \rightarrow 0$ as $\beta \rightarrow 0$.
\end{theorem}
The moment bounds obtained from Theorems~\ref{smallbetathm} and \ref{Q1statsmall} along with the application of Ito's formula leads to a somewhat surprising distributional convergence result as stated in the next theorem.
\begin{theorem}\label{betazero}
As $\beta \rightarrow 0$, the law of $\beta Q_2(\infty)$ converges weakly to a $\operatorname{Gamma}(2)$ distribution {\normalfont (}whose density is given by $f(x) = x\e^{-x}, \ x \ge 0${\normalfont)}.
\end{theorem}

\begin{remark}[{$\beta$ thresholds for small and large $\beta$ regimes}]
\normalfont
It is a natural and important question to ask how large $\beta$ needs to be for the `large $\beta$ regime' to actually manifest itself, and similarly for the `small $\beta$ regime'. Although we can obtain some thresholds for $\beta$ from the methods in the current article by making explicit choices of the bounds on $\beta$ required in our calculations and keeping track of these bounds, we believe that the obtained values will not be optimal. 
Figure~\ref{fig:num} gives numerical simulations for $-Q_1(\infty)$ and $Q_2(\infty)$ for different values of $\beta$ and visually depicts how the steady state behavior changes as $\beta$ varies.
All the figures in Figure~\ref{fig:num} are obtained by simulating sample paths of $(Q_1(t),Q_2(t))$ and plotting the histogram of occupancy measures of $Q_1$ and $Q_2$ over a time interval of length 1.5$\times 10^4$.
As can be observed, a `transition' occurs from one regime to the other as we vary $\beta$ from $0.1$ to $3$. 
Mathematically characterizing these thresholds and studying the `intermediate' regime is a challenging problem and we leave it as an open question.
\end{remark}

\begin{remark}[{Comparison with M/M/N}]\label{rem:mmncomp}
\normalfont
Theorem~\ref{betazero} should be contrasted with the corresponding result for the centralized queueing system.
Let $\bar{S}^N(t)$ denote the total number of tasks in an M/M/N system at time $t$. 
In that case, note that the total number of idle servers $\max\{N-\bar{S}^N, 0\}$ and the total number of waiting tasks $\max\{\bar{S}^N-N, 0\}$ are comparable to $-Q_1^N$ and $Q_2^N$ for the systems under the JSQ policy, respectively.
It is known that in case of M/M/N systems if the arrival rate $\lambda(N)$ scales as in the Halfin-Whitt regime~\cite[Theorem 2]{HW81}, then the centered and scaled total number of tasks in the system $(\bar{S}^N(t)-N)/\sqrt{N}$ converges weakly to a suitable diffusion process $\{\bar{S}(t)\}_{t\geq 0}$, and $\bar{S}^N(\infty)\dto \bar{S}(\infty)$,
where $\bar{S}(\infty)$ is the steady state of $\bar{S}$.
As $\beta\to 0$, \cite[Proposition 2]{HW81} implies that $\beta \bar{S}(\infty)$ for the M/M/$N$ queue converges weakly to a unit-mean exponential distribution. 
In contrast, Theorems~\ref{Q1statsmall} and \ref{betazero} shows that $\beta (Q_1(\infty) + Q_2(\infty))$ converges weakly to a Gamma$(2)$ random variable.
This indicates that in the Halfin-Whitt regime, 
although systems under the JSQ policy and the M/M/$N$ system have similar order of performance (in the sense that in both cases the total number of waiting tasks and  idle servers scale with $\sqrt{N}$), due to the distributed operation, in terms of the number of waiting tasks JSQ is a factor 2 worse in expectation than the corresponding centralized system.
\end{remark}

\begin{figure}
\begin{center}$
\begin{array}{ccc}
\includegraphics[width=80mm]{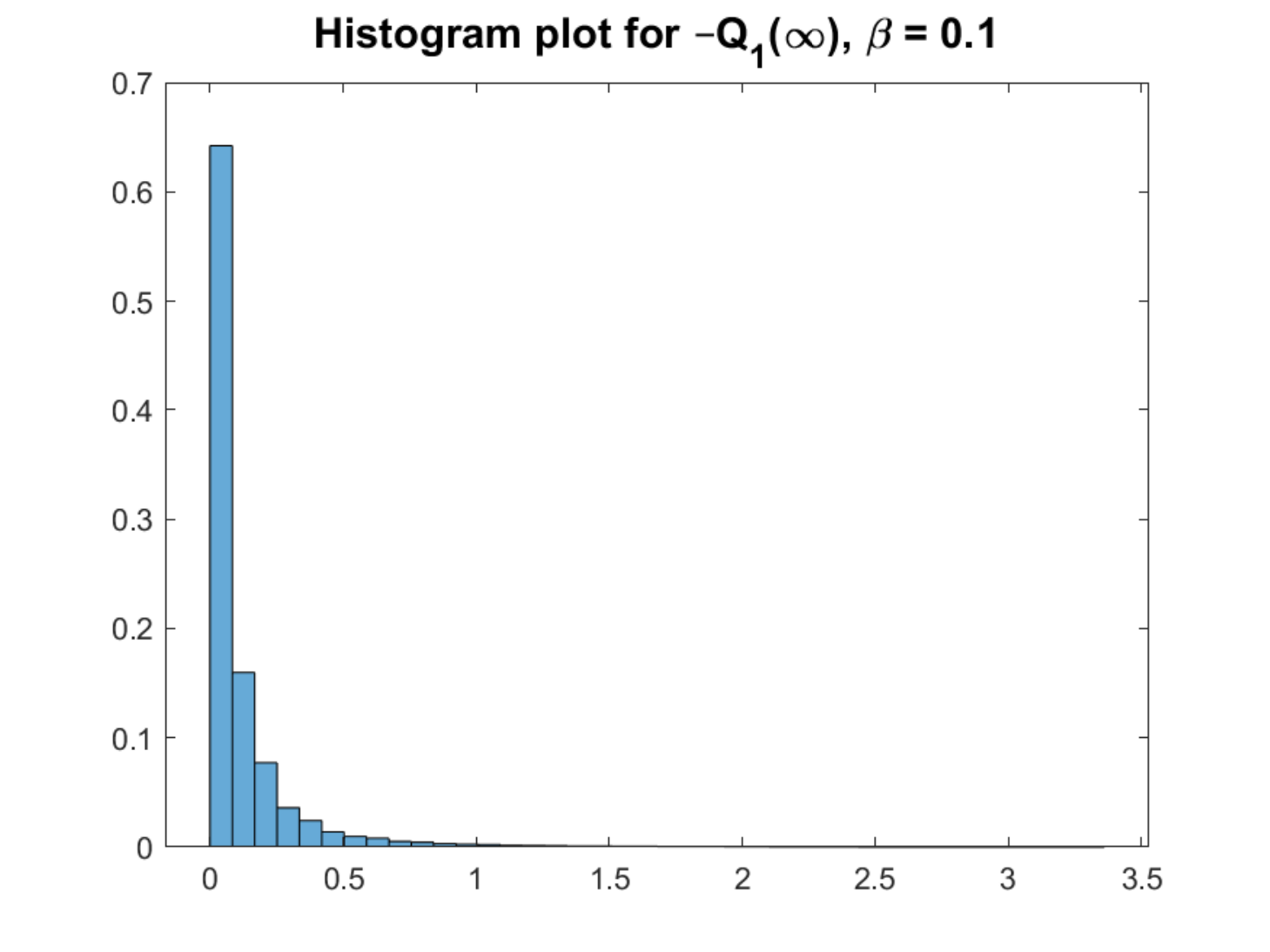}&
\includegraphics[width=80mm]{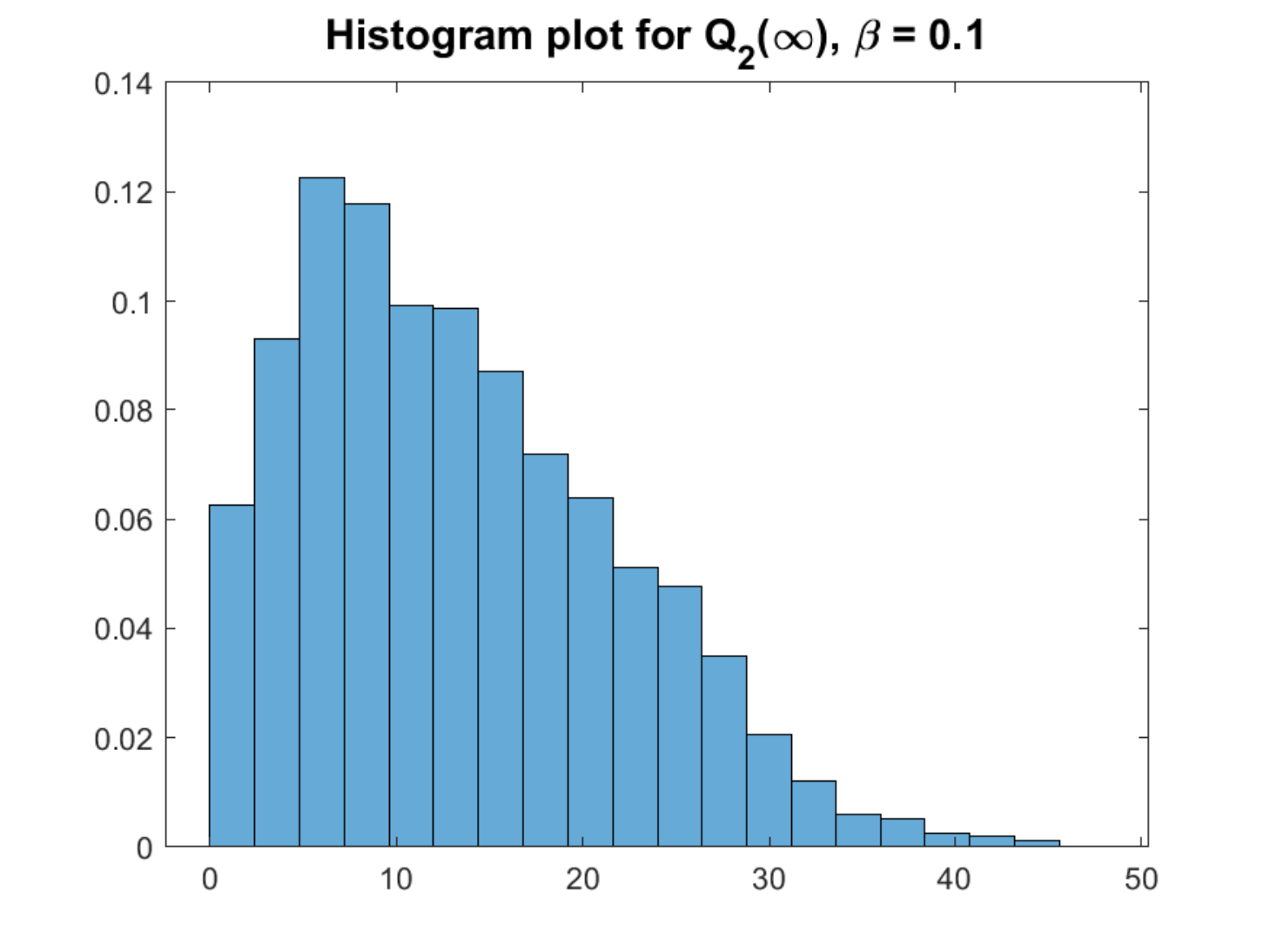}\\
\includegraphics[width=80mm]{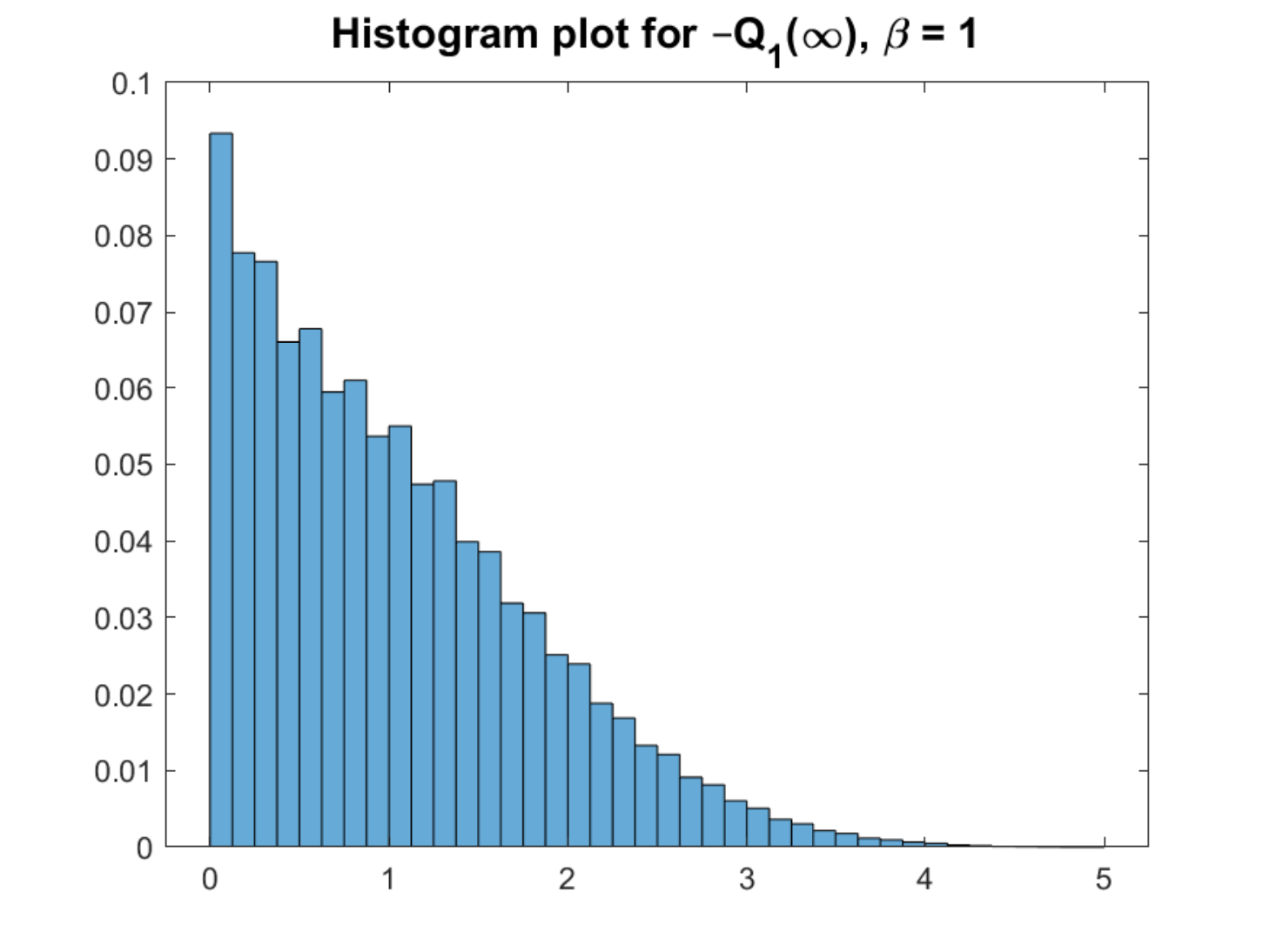}&
\includegraphics[width=80mm]{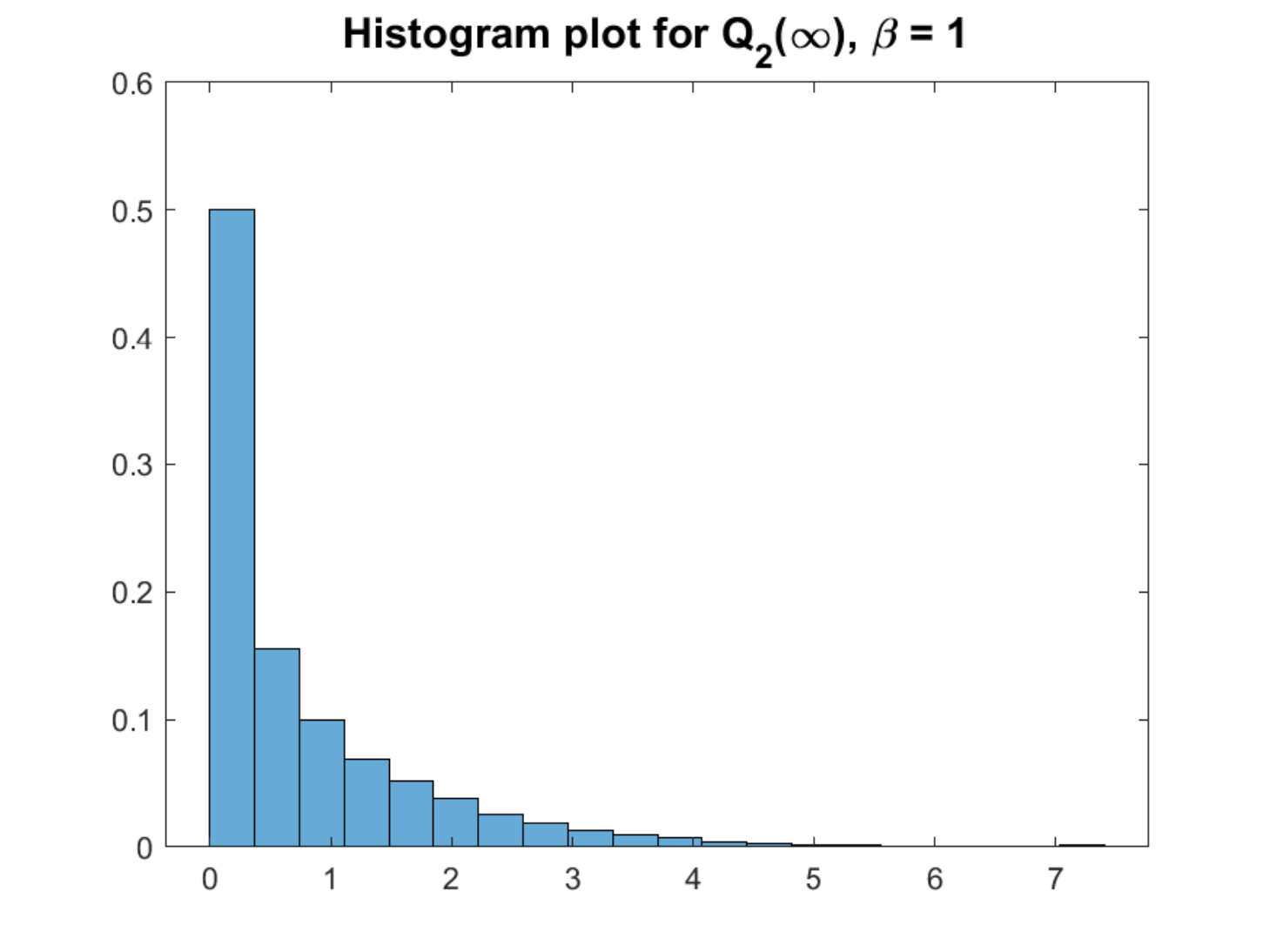}\\
\includegraphics[width=80mm]{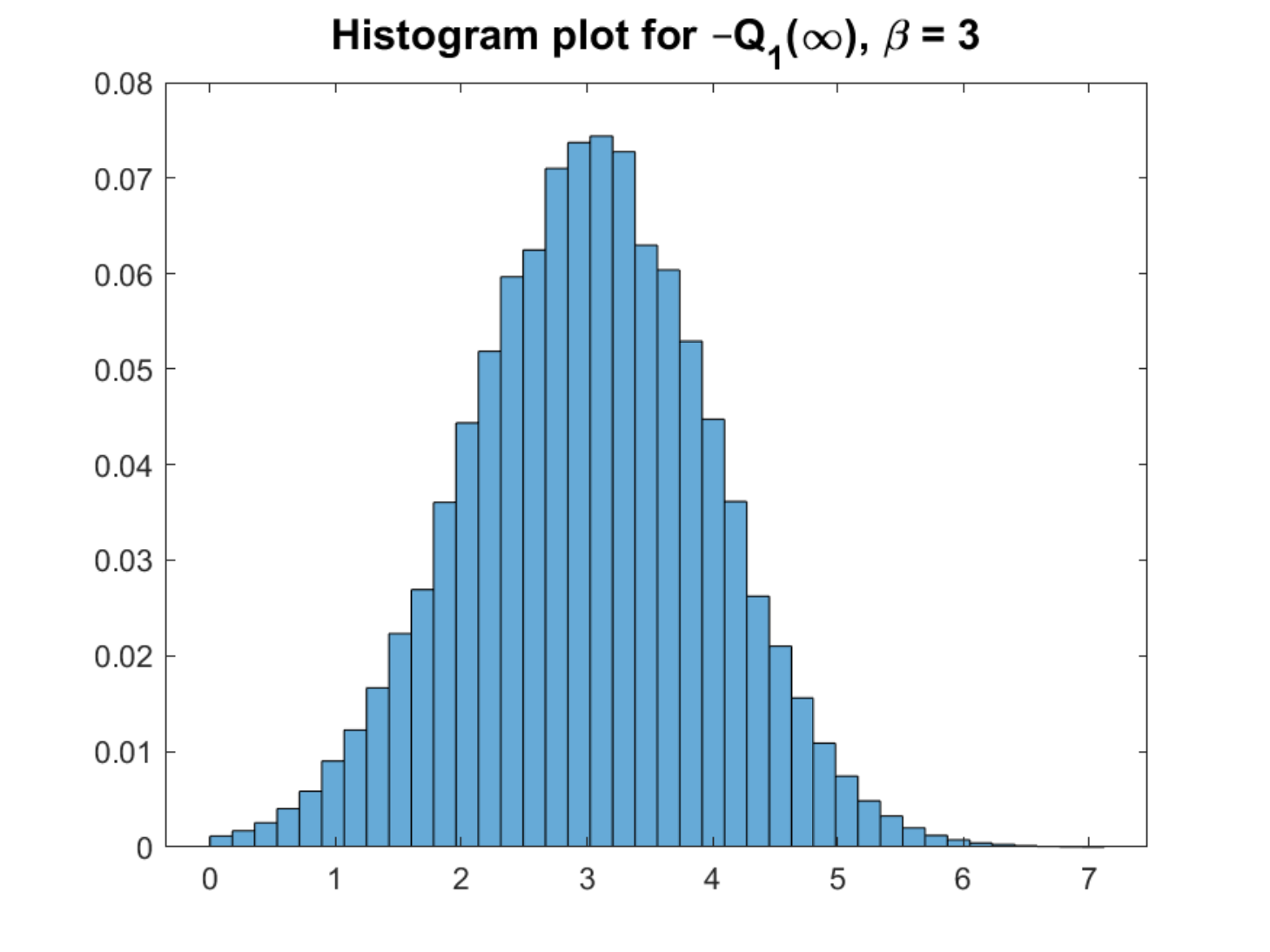}&
\includegraphics[width=80mm]{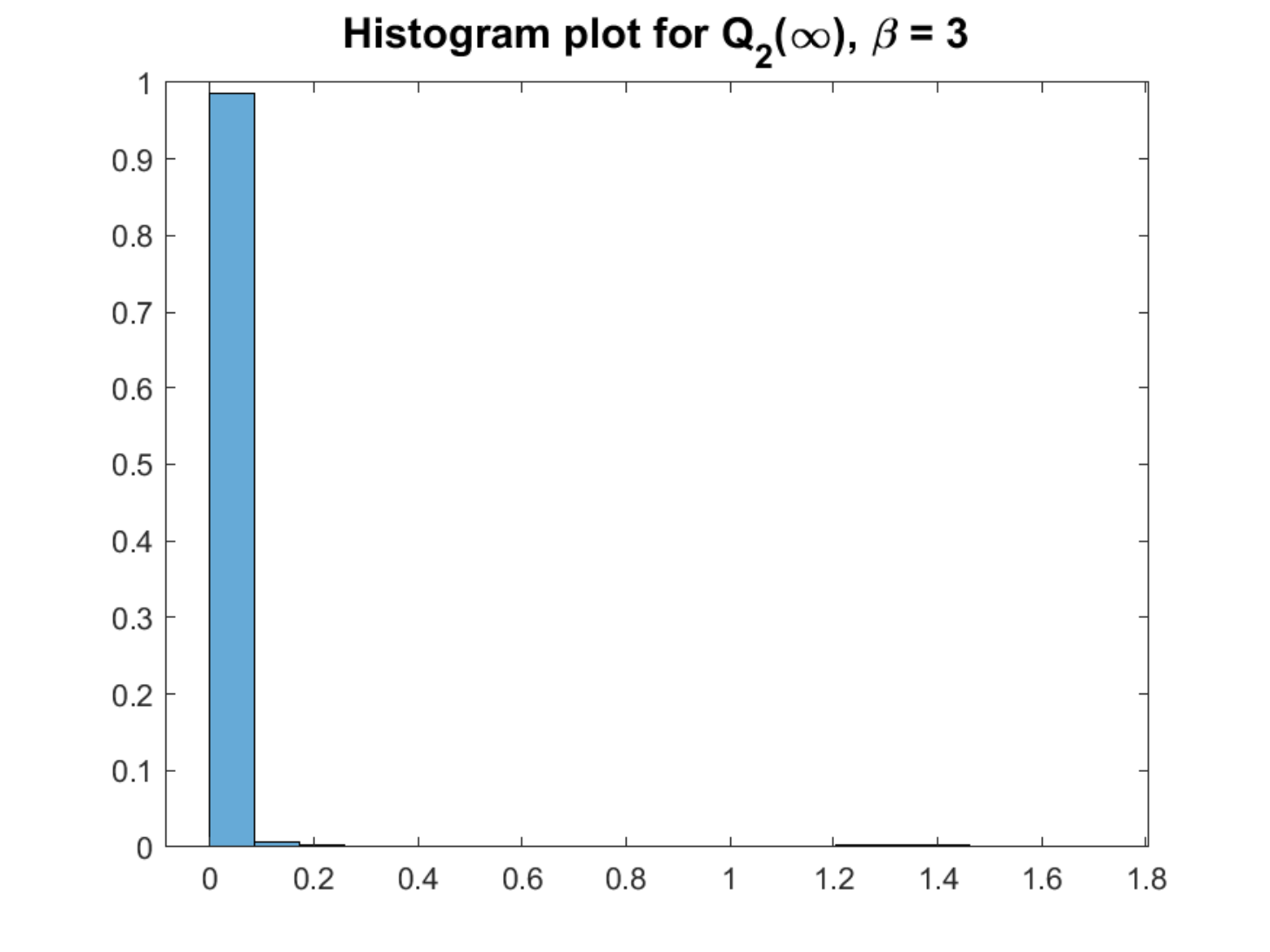}
\end{array}$
\end{center}
\caption{Histogram plots for $-Q_1(\infty)$ (left) and $Q_2(\infty)$ (right) when the value of $\beta$ is small (top), intermediate (middle), and large (bottom).
Observe the condensation of mass phenomenon for the distribution of $Q_2(\infty)$ when $\beta=3$.}
\label{fig:num}
\end{figure}

\section{Brief overview of the regenerative approach}
\label{sec:reg}

In this section we recall the regenerative approach introduced in~\cite{BM18}, which provides a tractable 
expression for the stationary distribution.
A stochastic process is called \emph{classical regenerative} if it starts anew at random times (called \emph{regeneration times}), independently of the past. See~\cite[Chapter 10]{Thorisson} for a rigorous treatment of regenerative processes.
The regeneration times split the process into renewal cycles that are independent and identically distributed, except possibly the first cycle.
Consequently, the behavior inside a specific renewal cycle characterizes the steady-state behavior.
In order to define the regeneration times, we introduce a few notations.
\begin{align*}
\tau_i(z)&:=\inf\{t \ge 0: Q_i(t)=z\}, \ \ i=1,2.\qquad\text{and}\qquad
\sigma(t):= \inf\{s \ge t: Q_1(s)=0\}.
\end{align*}
Now fix any $B>0$. For $k\geq 0$, define the stopping times
\begin{align}\label{rendef}
\alpha_{2k+1} &:= \inf \Big\{t\geq \alpha_{2k} : Q_2(t) = B\Big\},\quad
\alpha_{2k+2} := \inf \left\{t>\alpha_{2k+1}:Q_2(t) = 2B\right\},\quad
\Xi_k := \alpha_{2k+2},
\end{align}
with the convention that $\alpha_0=0$ and $\Xi_{-1}=0$.
The dependence of $B$ in the above stopping times is suppressed for convenience in notation, since the value of $B$ will be clear from the context. 
For any fixed $B>0$, \cite[Lemma 3.1]{BM18} states that the process $\{Q_1(t), Q_2(t)\}_{t\geq 0}$ is a classical regenerative process with regeneration times given by $\{\Xi_k\}_{k\geq 0}$.
Thus, invoking the theory of regenerative processes, it can be concluded that 
the process described by \eqref{eq:diffusionjsq} has a unique stationary distribution $\pi$, which can be represented as
\begin{equation}\label{eq:pi}
\pi(A) = \frac{\mathbb{E}_{(0, 2B)}\left(\int_{0}^{\Xi_0}\mathbf{1}_{[(Q_1(s),Q_2(s)) \in A]}ds\right)}{\mathbb{E}_{(0, 2B)}\left(\Xi_0\right)}
\end{equation}
for any measurable set $A \subseteq (-\infty, 0] \times (0, \infty)$. 
For convenience, rigorous statement of the above, along with some other useful results, are included in Appendix~\ref{app:recall}.

\section{Analysis in the large-$\beta$ regime}\label{sec:analysis-large}

In this section, we will investigate the behavior of the stationary distribution in the regime $\beta \ge \beta_0$ for sufficiently large $\beta_0$, and take $B=\beta^{-1}$ in \eqref{rendef}.
First, in Subsection~\ref{ssec:hitting-large}, we obtain estimates on the expectations of several carefully chosen hitting times. 
In Subsection~\ref{ssec:inter-large} we will provide upper and lower bounds on the expected inter-regeneration time $\mathbb{E}_{(0, 2\beta^{-1})}\left(\Xi_0\right)$.
Further, the hitting-time results of Subsection~\ref{ssec:hitting-large} will be used to obtain sharp bounds on the numerator on the right-side of~\eqref{eq:pi}, i.e., the amount of time the process spends on various regions within one renewal cycle.
Combining the results of Subsections~\ref{ssec:hitting-large} and \ref{ssec:inter-large} we prove in Subsection~\ref{ssec:proof-large}  the main results for the large-$\beta$ regime. 

To avoid cumbersome notation, we will use $\beta_0$ to denote the lower bound on $\beta$ for the assertion of each of the following lemmas to hold (the lower bounds change between lemmas and a common lower bound is obtained by taking the maximum of these bounds). Also, in the proofs, we will denote by $C, C'$ generic positive constants which do not depend on $\beta$ and whose values might change from line to line and between steps of calculations.

\subsection{Hitting time estimates}\label{ssec:hitting-large}
We start with a hitting-time estimate for $Q_2$ to hit some $\Theta(\beta)$ level starting from a larger level $y$.
When $Q_2$ is large, there is a deep interplay between the rate of decrease of $Q_2$ and the local time accumulated by $Q_1$.
Recall that, the rate of decay for $Q_2$ is proportional to itself.
However, observe that when $Q_2 \gg \beta$, $Q_1$ has a drift towards zero, and thus, spends most of the time around 0.
This increases the local time process $L$, which adds to $Q_2$.
Due to these two effects, it can be shown that $Q_2$ roughly behaves as a Brownian motion with drift $-\beta$, and so the expected time taken to hit some $\Theta(\beta)$ level starting from a larger level $y$ is approximately $\frac{y}{\beta}$. 
The next lemma formalizes the above heuristic.
\begin{lemma}\label{linfall}
There exists $\beta_0 \ge 1, \mathcal{C}>0$, such that for all $\beta \ge \beta_0$ and all $y \ge \beta/4$,
$$
\mathbb{E}_{(0,y)}\left(\tau_2(\beta/4)\right) \le \mathcal{C}\frac{y}{\beta}.
$$
\end{lemma}
Now we provide a useful estimate on the hitting time of zero by $Q_1$ when $Q_1(0)<0$ and $Q_2(0)$ is small.
From a high level, observe that if $\beta$ is large and $Q_2$ is considerably small (less than $\beta/2$, say), then $Q_1$ experiences a drift towards $-\beta$ whenever it is in the region $[-\beta+Q_2, 0]$.
Also, the drift is given by $-(\beta-Q_2) + (-Q_1) = \Theta(\beta)$.
Therefore, hitting time of zero by $Q_1$, in this case, can be thought of as the hitting time of a Brownian motion with a negative $\Theta(\beta)$ drift to hit level $\beta$.
As a result, for large enough $\beta$, the expected hitting time to 0 can be shown to increase exponentially with $\beta^2$.
The next lemma formalizes the above intuition.
\begin{lemma}\label{OUhit}
There exists $\beta_0 \ge 1$ and positive constants $\mathcal{C}_1^+,\mathcal{C}_2^+, \mathcal{C}_1^-,\mathcal{C}_2^-$ that do not depend on $\beta$ such that for all $\beta \ge \beta_0$,
\begin{align}\label{expexc}
\sup_{y \in (0,\beta]}\mathbb{E}_{(-\beta,y)}\left(\tau_1(0)\right) \le \mathcal{C}_1^+ \e^{\mathcal{C}_2^+\beta^2}, \ \ \inf_{y \in (0,\beta/2]}\mathbb{E}_{(-\beta/4,y)}\left(\tau_1(0)\right) \ge \mathcal{C}_1^-\e^{\mathcal{C}_2^-\beta^2}.
\end{align}
Moreover, $\mathcal{C}_1^-, \mathcal{C}_2^-$ above can be chosen so that for all $\beta \ge \beta_0$,
\begin{equation}\label{problb}
\sup_{y \in (0,\beta/2]}\prob_{(-\beta/4,y)}\left(\tau_1(0) \le \mathcal{C}_1^- \e^{\mathcal{C}_2^-\beta^2}\right) \le \mathcal{D}_1 \e^{-\mathcal{D}_2\beta^2},
\end{equation}
where $\mathcal{D}_1, \mathcal{D}_2$ are positive constants that do not depend on $\beta$.
\end{lemma}
From the heuristics given before Lemma~\ref{OUhit}, note that when $Q_2 < \beta$, $Q_1$ mostly stays away from zero, and hence, the local time process $L$ does not increase appreciably, resulting in a roughly exponential decay of $Q_2$. 
Consequently, starting from a suitable $\Theta(\beta)$ level $Q_2(0)$, the expected time take by $Q_2$ to hit a level $y \le Q_2(0)$ is $O(\log(\beta/y))$. 
The next lemma formalizes this.
\begin{lemma}\label{middletosmall}
There exist positive constants $C$ and $\beta_0$, such that for all fixed $\beta \ge \beta_0$ the following holds: For all $y \in [\beta \e^{-\mathcal{C}_1^- \e^{\mathcal{C}_2^-\beta^2}}, \beta/4]$ (where $\mathcal{C}_1^-, \mathcal{C}_2^-$ are the constants appearing in Lemma \ref{OUhit}),
$$
\sup_{z \in [y,\beta/4]}\mathbb{E}_{(0, z)}\left(\tau_2(y)\right) \le C\log\left(\frac{\beta}{y}\right).
$$
\end{lemma}
It should be noted that in Lemma~\ref{middletosmall}, although $z$, the initial value of $Q_2$, can be anything in the region $[y, \beta/4]$, the upper bound of the expected hitting time to $y$  does not depend on $z$. 
By imposing some restriction on the value of $z$, this can be further improved.
This is achieved in the next lemma.
\begin{lemma}\label{middle}
There exist positive constants $C$ and $\beta_0$, such that for all fixed $\beta\geq \beta_0$ the following holds:
For any $z \in [\beta^{-1}, \beta/4]$ and any $y \in [\beta \e^{-\mathcal{C}_1^- \e^{\mathcal{C}_2^-\beta^2}}, \beta/8]$ with $z \ge 2y$,
$$
\mathbb{E}_{(0,z)}\left(\tau_2(y)\right) \le C\log\left(\frac{z}{y}\right).
$$
\end{lemma}
There is a subtlety in the choice of the value of $z$ in the statement of Lemma \ref{middle}.
Note that it is crucial to have $z \ge \beta^{-1}$. This is because if $z \ll\beta^{-1}$, $Q_2$ can jump up to a $\Theta(\beta^{-1})$ level first (as can be seen by lower bounding the sum $Q_1(t)+Q_2(t)$ by a Brownian motion with drift $-\beta$) and then decrease (roughly exponentially) to hit $y$.
This produces an estimate of approximately $\log\left(1/(\beta y)\right)$. 
An estimate in such a scenario is obtained in the following lemma, where we obtain an upper bound on the expected amount of time spent by $Q_2$ above a suitable level $y\leq 2/\beta$ before hitting the level $2/\beta$. 
\begin{lemma}\label{smallval}
There exist positive constants $C$ and $\beta_0$,  such that for all fixed $\beta \ge \beta_0$ the following holds:
For all $y \in [2\beta \e^{-\mathcal{C}_1^-\e^{\mathcal{C}_2^-\beta^2}},2\beta^{-1}]$ (where $\mathcal{C}_1^-, \mathcal{C}_2^-$ are the constants appearing in Lemma \ref{OUhit}),
$$
\mathbb{E}_{(0,y/2)}\left(\int_0^{\tau_2(2\beta^{-1})}\mathbf{1}_{[Q_2(s) \ge y]}ds\right) \le C\log\left(\frac{4}{\beta y}\right).
$$
\end{lemma}
The next lemma essentially provides an upper bound for the steady-state tail probabilities for $Q_2$ in the region $(\beta^{-1}, \infty)$.
It will also be used in bounding the expected inter-regeneration times $\mathbb{E}_{(0, 2\beta^{-1})}\left(\Xi_0\right)$, as stated in Lemma~\ref{renexp}.
\begin{lemma}\label{largebeta}
There exist positive constants $C_L, C'_L$ and $\beta_0 \ge 1$ such that for any fixed $\beta \ge \beta_0$ the following holds:
For all $y \in [4\beta^{-1}, \infty)$
\begin{equation*}
\prob_{(0,2\beta^{-1})}\left(\tau_2(y) < \tau_2(\beta^{-1})\right) \le C_L\e^{-C'_L \beta y}.
\end{equation*}
\end{lemma}
The proof of Lemma~\ref{largebeta} relies on some very intricate understanding of the qualitative behavior of the diffusion process, and follows using several intermediate steps, as further explained below in Remark~\ref{rem:lem4.6}.

\begin{remark}\normalfont
\label{rem:lem4.6}
Observe that from~\cite{BM18} we already know that for any fixed $\beta>0$, the steady-state tail probability $\pi(Q_2>y)$ is upper bounded by $C_1^*\e^{-C_2^*\beta y}$ for all $y\geq \beta + R^+/\beta$, where $C_1, C_2$, and $R^+$ are positive constants independent of $\beta$.
However, it requires a significant effort to get the tail estimate when $y$ is in the region $[\beta^{-1},\infty)$.
This is a crucial step, since for large values of $\beta$, this produces huge improvement in understanding the bulk behavior of $Q_2$ (viz., to obtain sharper bounds on the steady-state  expectation).

The key challenge stems from the fact that
the diffusion process exhibits a different qualitative behavior in the region $\{-Q_1 + Q_2 < \beta\}$ than in the region $\{-Q_1 + Q_2 > \beta\}$. 
This is because in the latter region, when $-Q_1 + Q_2$ is large enough, the local time and the drift acting on $Q_2$ result in a net negative linear drift of approximately $-\beta$. 
Lemma \ref{Q2gebeta2} exploits this linear drift to produce the exponential tail estimate on $Q_2$ in the latter region.
However, in the former region, $Q_1$ has a negative drift, and consequently, it does not hit the origin as often as in the latter region. 
Thus, with very little increment in the local time, $Q_2$ decays almost exponentially. 
Hence, a careful analysis is needed to combine the different behaviors in different regions to obtain a unified tail estimate.
Details of the above approach are given in Appendix~\ref{app:lemma4.8}.
\end{remark}
Lemmas \ref{linfall} -- \ref{smallval} are proved in Appendix~\ref{app:large-aux}, and Lemma~\ref{largebeta} is proved in Appendix~\ref{app:lemma4.8}.

\subsection{Bounds on the inter-regeneration times}\label{ssec:inter-large}
In this section we state and prove upper and lower bounds on the expected inter-regeneration times $\mathbb{E}_{(0, 2\beta^{-1})}\left(\Xi_0\right)$, which will be used in Subsection~\ref{ssec:proof-large} to prove the main results.
\begin{lemma}\label{renexp}
There exist $\beta_0>0$ and positive constants $C_1, C_2, C'_1, C'_2$ (not depending on $\beta$) such that for all $\beta \ge \beta_0$,
$$
C_1\e^{C_2\beta^2} \le \mathbb{E}_{(0, 2\beta^{-1})}\left(\Xi_0\right) \le C'_1\e^{C'_2\beta^2}.
$$
\end{lemma}
Rest of this section is devoted in the proof of Lemma~\ref{renexp}.
\subsubsection{Proof of the upper bound}
Recall that $\Xi_0= \alpha_2$ where $\alpha_1$ and $\alpha_2$ are as defined in \eqref{rendef} with $B=\beta^{-1}$.
Write $\alpha_{1,-\beta} = \inf\{t \ge \alpha_1: Q_1(t) = -\beta\}$. Then, using the strong Markov property,
\begin{eq}\label{ren21}
&\mathbb{E}_{(0,2\beta^{-1})}\left(\alpha_2\right) = \mathbb{E}_{(0,2\beta^{-1})}\left(\alpha_2\mathbf{1}_{[\alpha_{1,-\beta} \le \alpha_2]}\right) + \mathbb{E}_{(0,2\beta^{-1})}\left(\alpha_2\mathbf{1}_{[\alpha_{1,-\beta} > \alpha_2]}\right)\\
&\quad\le \sup_{y \in (0, 2\beta^{-1})}\mathbb{E}_{(-\beta,y)}\left(\tau_2(2/\beta)\right) + \mathbb{E}_{(0,2\beta^{-1})}\left(\alpha_{1,-\beta}\mathbf{1}_{[\alpha_{1,-\beta} \le \alpha_2]}\right) + \mathbb{E}_{(0,2\beta^{-1})}\left(\alpha_2\mathbf{1}_{[\alpha_{1,-\beta} > \alpha_2]}\right)\\
&\quad= \sup_{y \in (0, 2\beta^{-1})}\mathbb{E}_{(-\beta,y)}\left(\tau_2(2/\beta)\right) + \mathbb{E}_{(0,2\beta^{-1})}\left(\alpha_{1,-\beta} \wedge \alpha_2\right)\\
&\quad\le \sup_{y \in (0, 2\beta^{-1})}\mathbb{E}_{(-\beta,y)}\left(\tau_2(2/\beta)\right) + \mathbb{E}_{(0,2\beta^{-1})}\left(\alpha_1\right) + \mathbb{E}_{(0,2\beta^{-1})}\left(\mathbb{E}_{(Q_1(\alpha_1), \beta^{-1})}\left(\tau_1(-\beta) \wedge \tau_2(2/\beta)\right)\right).
\end{eq}
In the rest of the proof of the upper bound, we will prove bounds on the three terms on the right side of~\eqref{ren21}.
From Lemma \ref{middle}, note that for $\beta \ge \beta_0$ sufficiently large,
\begin{equation}\label{ren1exp}
\mathbb{E}_{(0,2\beta^{-1})}\left(\alpha_1\right) \le C.
\end{equation}
Equations~\eqref{e1} and \eqref{e4}  provide upper bounds for $\mathbb{E}_{(0,2\beta^{-1})}\left(\mathbb{E}_{(Q_1(\alpha_1), \beta^{-1})}\left(\tau_1(-\beta)\wedge \tau_2(2/\beta)\right)\right)$ and $\sup_{y \in (0, 2\beta^{-1})}\mathbb{E}_{(-\beta,y)}\left(\tau_2(2/\beta)\right)$, respectively.
Combining \eqref{ren1exp}, \eqref{e1}, and \eqref{e4} will complete the proof of the upper bound.

First we claim the following.
\begin{claim}\label{cl:ren22}
For all $\beta>0$ the following holds:
\begin{equation}\label{ren22}
\sup_{x \in [-\beta,0], y \in (0,2\beta^{-1}]}\mathbb{E}_{(x,y)}\left(\tau_1(-\beta)\wedge \tau_2(2/\beta)\right) \le C\beta^2.
\end{equation}
\end{claim}
\begin{claimproof}
Note that for $s< t \le \tau_1(-\beta)\wedge \tau_2(2/\beta)$,
$$
Q_1(t) \le Q_1(s) + \sqrt{2}(W(t) - W(s)) + 2\beta^{-1}(t-s).
$$
This gives us
\begin{multline*}
\inf_{x \in [-\beta,0], y \in (0,2\beta^{-1}]}\prob_{(x,y)}\left(\tau_1(-\beta)\wedge \tau_2(2/\beta) < \beta^2\right) \ge \prob\left(\sqrt{2}W(\beta^2) < -3\beta\right) \ge p>0,
\end{multline*}
where $p$ does not depend on $\beta, x,y$. Thus, for $n \ge 1$, by the Markov property applied at time $(n-1)\beta^2$,
\begin{align*}
&\sup_{x \in [-\beta,0], y \in (0,2\beta^{-1}]}\prob_{(x,y)}\left(\tau_1(-\beta)\wedge \tau_2(2/\beta) \ge n\beta^2\right)\\
&\le\sup_{x \in [-\beta,0], y \in (0,2\beta^{-1}]}\prob_{(x,y)}\left(\tau_1(-\beta)\wedge \tau_2(2/\beta) \ge (n-1)\beta^2\right)\\
&\hspace{5cm}\times \sup_{x \in [-\beta,0], y \in (0,2\beta^{-1}]}\prob_{(x,y)}\left(\tau_1(-\beta)\wedge \tau_2(2/\beta) \ge \beta^2\right)
\le (1-p)^n,
\end{align*}
which implies \eqref{ren22}.
\end{claimproof}
\noindent
Next, we will bound $\mathbb{E}_{(0,2\beta^{-1})}\left(-Q_1(\alpha_1)\right)$. 

\begin{claim}\label{cl:qalpha}
There exists $\beta_0 \ge 1$ such that for all $\beta \ge \beta_0$
\begin{equation}\label{qalpha}
\mathbb{E}_{(0,2\beta^{-1})}\left(-Q_1(\alpha_1)\right) \le C\beta^4.
\end{equation}
\end{claim}
\begin{claimproof}
Take $(Q_1(0), Q_2(0))= (0,2\beta^{-1})$. For $x \ge (2\beta)^4$,
\begin{eq}\label{threeest}
\prob_{(0,2\beta^{-1})}&\left(Q_1(\alpha_1) \le -x\right) = \mathbb{E}_{(0,2\beta^{-1})}\left(\mathbf{1}_{[\tau_1(-x/2) < \alpha_1]} Q_1(\alpha_1) \le -x\right)\\
&\le \prob_{(0,2\beta^{-1})}\left(\tau_2(x^{1/4}) < \alpha_1\right) + \sup_{y \le x^{1/4}}\prob_{(-x/2, y)}\left(\tau_1(-x) \wedge \tau_1(0) \le \log(\beta x^{1/4})\right)\\
&\le \prob_{(0,2\beta^{-1})}\left(\tau_2(x^{1/4}) < \alpha_1\right) + \sup_{y \le x^{1/4}}\prob_{(-x/2, y)}\left(\tau_1(0) \le \log(\beta x^{1/4})\right)\\
&\hspace{3.5cm}+ \sup_{y \le x^{1/4}}\prob_{(-x/2, y)}\left(\tau_1(-x) \le \log(\beta x^{1/4}), \tau_1(-x) \le \tau_1(0)\right)
\end{eq}
where we used the fact that $Q_2(t)$ decreases exponentially for $t \le \tau_1(0)$. By Lemma~\ref{largebeta}, there is $\beta_0 \ge 1$ such that for all $\beta \ge \beta_0$ and all $x \ge (2\beta)^4$,
\begin{equation}\label{qalpha1}
\prob_{(0,2\beta^{-1})}\left(\tau_2(x^{1/4}) < \alpha_1\right) \le C\e^{-C'\beta x^{1/4}} \le C\e^{-C'x^{1/4}}.
\end{equation}
Now, we estimate $\sup_{y \le x^{1/4}}\prob_{(-x/2, y)}\left(\tau_1(0) \le \log(\beta x^{1/4})\right)$. Take $(Q_1(0), Q_2(0)) = (-x/2, y)$ for $y \le x^{1/4}$. For $t \le \tau_1(0)$,
$$
Q_1(t) = -\frac{x}{2} + \sqrt{2}W(t) - \beta t + \int_0^t \left(-Q_1(s) + y\e^{-s}\right)ds.
$$
By Proposition 2.18 of \cite{Karatzas}, for $t \le \tau_1(0)$, $Q_1(t)$ is stochastically bounded above by the Ornstein-Uhlenbeck process
$$
X(t) = -\frac{x}{2} + \sqrt{2}W(t) + \int_0^t \left(x^{1/4} -\beta-X(s)\right)ds.
$$
By the Doob representation of Ornstein-Uhlenbeck processes, we can write
$$
X(t) = -\frac{x}{2}\e^{-t} + (x^{1/4} -\beta)(1-\e^{-t}) + \e^{-t} \widetilde{W}\left(\e^{2t}-1\right)
$$
for some Brownian motion $\widetilde{W}$. Therefore, for $x \ge (2\beta)^4$ where $\beta \ge \beta_0$ for sufficiently large $\beta_0$,
\begin{equation}\label{qalpha2}
\sup_{y \le x^{1/4}}\prob_{(-x/2, y)}\left(\tau_1(0) \le \log(\beta x^{1/4})\right) \le \prob\left(\sup_{t \le x} \widetilde{W}(t) \ge \frac{x}{4}\right) \le C\e^{-C'x}.
\end{equation}
Recall that for $t \le \tau_1(0)$, $Q_1(t) \ge Q_1(0) + \sqrt{2}W(t) - \beta t$. Thus, for $x \ge (2\beta)^4$,
\begin{eq}\label{qalpha3}
&\sup_{y \le x^{1/4}}\prob_{(-x/2, y)}\Big(\tau_1(-x) \le \log(\beta x^{1/4}), \tau_1(-x) \le \tau_1(0)\Big) \\
&\le \prob\Big(\inf_{t \le C \log x} \sqrt{2}W(t) \le -\frac{x}{2} + C\beta \log x\Big)\le \prob\Big(\inf_{t \le C \log x} \sqrt{2}W(t) \le -\frac{x}{2} + Cx^{1/4} \log x\Big) \le C\e^{-C'x}.
\end{eq}
Thus, combining \eqref{qalpha1}, \eqref{qalpha2} and \eqref{qalpha3}, we get for $\beta_0 \ge 1$ such that for all $\beta \ge \beta_0$ and all $x \ge (2\beta)^4$,
\begin{equation*}
\prob_{(0,2\beta^{-1})}\left(Q_1(\alpha_1) \le -x\right) \le C\e^{-C'x^{1/4}}.
\end{equation*}
Consequently, Claim~\ref{cl:qalpha} follows.
\end{claimproof}
\noindent
Note that
\begin{align*}
&\mathbb{E}_{(0,2\beta^{-1})}\left(\mathbb{E}_{(Q_1(\alpha_1), \beta^{-1})}\left(\tau_1(-\beta)\wedge \tau_2(2/\beta)\right)\right)\\
&= \mathbb{E}_{(0,2\beta^{-1})}\left(\mathbb{E}_{(Q_1(\alpha_1), \beta^{-1})}\left(\tau_1(-\beta)\wedge \tau_2(2/\beta)\right) \mathbf{1}_{[Q_1(\alpha_1) \ge -\beta]}\right)\\
&\hspace{4cm}+ \mathbb{E}_{(0,2\beta^{-1})}\left(\mathbb{E}_{(Q_1(\alpha_1), \beta^{-1})}\left(\tau_1(-\beta)\wedge \tau_2(2/\beta)\right) \mathbf{1}_{[Q_1(\alpha_1) < -\beta]}\right)\\
&\le \sup_{x \in [-\beta,0], y \in (0,2\beta^{-1}]}\mathbb{E}_{(x,y)}\left(\tau_1(-\beta)\wedge \tau_2(2/\beta)\right)\\
&\hspace{4cm} + \mathbb{E}_{(0,2\beta^{-1})}\left(\mathbb{E}_{(Q_1(\alpha_1), \beta^{-1})}\left(\tau_1(-\beta)\wedge \tau_2(2/\beta)\right) \mathbf{1}_{[Q_1(\alpha_1) < -\beta]}\right)\\
&\le \sup_{x \in [-\beta,0], y \in (0,2\beta^{-1}]}\mathbb{E}_{(x,y)}\left(\tau_1(-\beta)\wedge \tau_2(2/\beta)\right)\\
 &\hspace{4cm}+ \mathbb{E}_{(0,2\beta^{-1})}\left(\mathbf{1}_{[Q_1(\alpha_1) < -\beta]}\mathbb{E}_{(Q_1(\alpha_1), \beta^{-1})}\left(\tau_1(-\beta)\right)\right).
\end{align*}
By \eqref{ren22}, $\sup_{x \in [-\beta,0], y \in (0,2\beta^{-1}]}\mathbb{E}_{(x,y)}\left(\tau_1(-\beta)\wedge \tau_2(2/\beta)\right) \le C\beta^2$. 
Further, to estimate the second term in the right hand side above, 
we will make use of the following simple claim.
\begin{claim}\label{cl:threeuse}
Fix any $\beta>0$.
For any $x<-\beta, y>0$,
\begin{equation}\label{threeuse}
\mathbb{E}_{(x,y)}\left(\tau_1(-\beta)\right) \le C \log\left(2 + |x + \beta|\right).
\end{equation}
\end{claim}
\begin{claimproof}
Note that if $(Q_1(0),Q_2(0)) = (x,y)$ where $x < -\beta$, then for $t \le \tau_1(-\beta)$, $Q_1^*(t) = Q_1(t) + \beta$ is stochastically bounded below by an Ornstein-Uhlenbeck process
$$
X^*(t) = x + \beta + \sqrt{2}W(t)  - \int_0^t X^*(s)ds.
$$
From the Doob representation $X^*(t)  =\left(x + \beta\right)\e^{-t} + \e^{-t}W^*\left(\e^{2t}-1\right)$ (where $W^*$ is a standard Brownian motion), for $t \ge \log \left(2 + |x + \beta|\right)$,
\begin{align*}
\prob_{(x,y)}\left(\tau_1(-\beta) >t\right) &\le \prob\left(\left(x + \beta\right) + W^*\left(\cdot\right) \text{ hits zero after time } \e^{2t}-1\right)\\
&= \int_{\e^{2t}-1}^{\infty}\frac{|x+\beta|}{\sqrt{2\pi s^3}}\e^{-(x +\beta)^2/(2s)}ds \le \frac{|x+\beta|}{(\e^{2t}-1)^{1/2}} \le C|x+\beta| \e^{-t}.
\end{align*}
This completes the proof of the claim.
\end{claimproof}
\noindent
Note that the statement of Claim~\ref{cl:threeuse} is for {\em all} $\beta>0$ and it will be used subsequently in the small $\beta$ regime in Section \ref{sec:small-beta}.
Now, using Jensen's inequality, Claim~\ref{cl:threeuse}, and \eqref{qalpha},
\begin{eq}\label{eq:aftercl1}
\mathbb{E}_{(0,2\beta^{-1})}\left(\mathbf{1}_{[Q_1(\alpha_1) < -\beta]}\mathbb{E}_{(Q_1(\alpha_1), \beta^{-1})}\left(\tau_1(-\beta)\right)\right) \le \mathbb{E}_{(0,2\beta^{-1})}\left(\log(2+ |Q_1(\alpha_1) + \beta|)\right)\\
\le \log\left(2+ \mathbb{E}_{(0,2\beta^{-1})}\left(-Q_1(\alpha_1)\right) + \beta\right) \le C \log(2 + C'\beta^4) \le C\log \beta.
\end{eq}
Thus, we get from~\eqref{ren22} and \eqref{eq:aftercl1}
\begin{equation}\label{e1}
\mathbb{E}_{(0,2\beta^{-1})}\left(\mathbb{E}_{(Q_1(\alpha_1), \beta^{-1})}\left(\tau_1(-\beta)\wedge \tau_2(2/\beta)\right)\right) \le C\beta^2 + C\log\beta \le C'\beta^2.
\end{equation}
\ 

Next, we will estimate $\sup_{y \in (0, 2\beta^{-1})}\mathbb{E}_{(-\beta,y)}\left(\tau_2(2/\beta)\right)$. 
Take $(Q_1(0), Q_2(0)) = (-\beta, y)$, where $y < 2\beta^{-1}$. Define $\mathbf{e}_0 = 0$ and for $k \ge 0$,
\begin{align*}
\mathbf{e}_{2k+1} &= \inf\{t \ge \mathbf{e}_{2k}: Q_1(t)=0 \text{ or } Q_2(t) = 2\beta^{-1}\},\\
\mathbf{e}_{2k+2} &= \inf\{t \ge \mathbf{e}_{2k+1}: Q_1(t)=-\beta \text{ or } Q_2(t) = 2\beta^{-1}\}.
\end{align*}
Let $\mathcal{N}_{\mathbf{e}} = \inf\{k \ge 1: Q_2(\mathbf{e}_{2k})=2\beta^{-1}\}$. Note that by the expectation upper bound given in Lemma \ref{OUhit} and \eqref{ren22},
\begin{equation}\label{e2}
\sup_{y \in (0, 2\beta^{-1})}\mathbb{E}_{(-\beta,y)}\left(\mathbf{e}_{2}\right) \le C\e^{C'\beta^2}.
\end{equation}
Define $S(t): = Q_1(t) + Q_2(t)$.
If $(Q_1(0), Q_2(0))=(0,y)$ for any $y \le 2\beta^{-1}$, note that $S(t) \le Q_1(t) + 2\beta^{-1}$ for $t \le \tau_2(2\beta^{-1})$ and $Q_2(t) \ge S(t) \ge \sqrt{2}W(t) - \beta t$ for all $t$, and hence we get for $\beta \ge 2$,
\begin{multline*}
\inf_{y \le 2\beta^{-1}}\prob_{(0,y)}\left(\tau_2(2\beta^{-1}) \le \tau_1(-\beta)\right) \ge \prob\left(\sqrt{2}W(t) - \beta t \text{ hits } 2\beta^{-1} \text{ before } -\beta/2\right) = \frac{1-\e^{-\beta^2/2}}{\e^2 - \e^{-\beta^2/2}}\\
\ge (1-\e^{-2})\e^{-2} = p_{\mathbf{e}}>0,
\end{multline*}
where $p_{\mathbf{e}}$ does not depend on $\beta$. This immediately gives us for $k \ge 1$,
\begin{equation}\label{e3}
\sup_{y \in (0, 2\beta^{-1})}\prob_{(-\beta,y)}\left(\mathcal{N}_{\mathbf{e}} \ge k\right) \le (1-p_{\mathbf{e}})^k.
\end{equation}
Thus, by \eqref{e2} and \eqref{e3},
\begin{eq}\label{e4}
\sup_{y \le 2\beta^{-1}}\mathbb{E}_{(-\beta,y)}(\tau_2(2/\beta)) &= \sup_{y \le 2\beta^{-1}}\mathbb{E}_{(-\beta,y)}\Big(\sum_{k=1}^{\infty}\mathbf{1}_{[\mathcal{N}_{\mathbf{e}} = k]}\mathbf{e}_{2k}\Big)\\
&= \sup_{y \le 2\beta^{-1}}\mathbb{E}_{(-\beta,y)}\Big(\sum_{k=1}^{\infty}\left(\mathbf{e}_{2k}-\mathbf{e}_{2k-2}\right)\mathbf{1}_{[\mathcal{N}_{\mathbf{e}} > k-1]}\Big)\\
&\le \sum_{k=1}^{\infty}\sup_{y \in (0, 2\beta^{-1})}\mathbb{E}_{(-\beta,y)}\left(\mathbf{e}_{2}\right)\sup_{y \in (0, 2\beta^{-1})}\prob_{(-\beta,y)}\left(\mathcal{N}_{\mathbf{e}} \ge k\right)\\
&\le C\e^{C'\beta^2}\sum_{k=1}^{\infty}(1-p_{\mathbf{e}})^k = C\e^{C'\beta^2}.
\end{eq}
Finally, using \eqref{ren1exp}, \eqref{e1}, and \eqref{e4} in \eqref{ren21}, we obtain
\begin{equation*}
\mathbb{E}_{(0,2\beta^{-1})}\left(\alpha_2\right) \le C'_1\e^{C'_2\beta^2},
\end{equation*}
which yields the upper bound claimed in the lemma.

\subsubsection{Proof of the lower bound} 
By the strong Markov property applied at time $\inf\{t \ge \alpha_1: Q_1(t) =0\}$,
\begin{align*}
\mathbb{E}_{(0,2\beta^{-1})}\left(\alpha_2\right) &\ge \inf_{y \in (0,\beta^{-1}]}\mathbb{E}_{(0,y)}\left(\tau_2(2\beta^{-1})\right)\\
&\ge \inf_{y \in (0,\beta^{-1}]}\prob_{(0,y)}\left(\tau_1(-\beta/4) <\tau_2(2\beta^{-1})\right)\inf_{y \in (0,2\beta^{-1}]}\mathbb{E}_{(-\beta/4,y)}\left(\tau_1(0)\right).
\end{align*}
Recall that if $(Q_1(0),Q_2(0))=(0,y)$ for any $y \in (0,\beta^{-1}]$, $Q_2(t)=2\beta^{-1}$ for some $t$ if and only if $S(t) = 2\beta^{-1}$. Moreover, $y + \sqrt{2}W(t) -3\beta t/4 \ge S(t) \ge Q_1(t)$ for all $t \le \tau_1(-\beta/4)$. Thus, for $\beta \ge 2$,
\begin{align*}
\sup_{y \in (0,\beta^{-1}]}&\prob_{(0,y)}\left(\tau_1(-\beta/4) > \tau_2(2\beta^{-1})\right)\\
&\le \sup_{y \in (0,\beta^{-1}]}\prob_{(0,y)}\left(y + \sqrt{2}W(t) -3\beta t/4 \text{ hits } 2\beta^{-1} \text{ before } -\beta/4\right)\\
&\le \prob\left(\sqrt{2}W(t) -3\beta t/4 \text{ hits } \beta^{-1} \text{ before } -\beta/2\right) = \frac{1 - \e^{-3\beta^2/8}}{\e^{3/4} - \e^{-3\beta^2/8}} \le \e^{-3/4}.
\end{align*}
Combining this with the expectation lower bound in Lemma \ref{OUhit}, we obtain for all $\beta \ge \beta_0$ for sufficiently large $\beta_0$,
$$
\mathbb{E}_{(0,2\beta^{-1})}\left(\alpha_2\right) \ge (1-\e^{-3/4})\inf_{y \in (0,2\beta^{-1}]}\mathbb{E}_{(-\beta/4,y)}\left(\tau_1(0)\right) \ge C_1\e^{C_2\beta^2},
$$
which yields the lower bound claimed in the lemma. $\hfill\qed$

\subsection{Proofs of the main results}\label{ssec:proof-large}
Now, we can prove Theorem~\ref{largestat} about the detailed behavior of the stationary distribution of $Q_2$ in the large-$\beta$ regime.
\begin{proof}[Proof of Theorem~\ref{largestat}]
Fix $\beta\geq \beta_0$ large enough.
Recall \eqref{eq:pi} with $B=\beta^{-1}$.
Both the proof of the upper bound and the lower bound consist of two cases: (a) when $y \ge 4\beta^{-1}$ and (b) when  $y \in [4\beta \e^{-\mathcal{C}_1^-\e^{\mathcal{C}_2^-\beta^2}}, 4\beta^{-1})$.\\

\noindent
\textit{Proof of the upper bound.}
(a)
Let $y \ge 4\beta^{-1}$. 
Then for $\alpha_1, \alpha_2$ as defined in \eqref{rendef} with $B=\beta^{-1}$,
\begin{eq}\label{s1}
&\mathbb{E}_{(0, 2\beta^{-1})}\left(\int_{0}^{\Xi_0}\mathbf{1}_{[Q_2(s) \ge y]}ds\right) = \mathbb{E}_{(0, 2\beta^{-1})}\left(\int_{0}^{\alpha_1}\mathbf{1}_{[Q_2(s) \ge y]}ds\right)\\
&\le \mathbb{E}_{(0, 2\beta^{-1})}\left(\mathbf{1}_{[\tau_2(y) < \tau_2(\beta^{-1})]}\left(\tau_2(\beta^{-1})-\tau_2(y)\right) \right)
 = \prob_{(0,2\beta^{-1})}\left(\tau_2(y) < \tau_2(\beta^{-1})\right) \mathbb{E}_{(0, y)}\left(\tau_2(\beta^{-1})\right),
\end{eq}
where the last step follows from the strong Markov property. 
For the first term on the right side of~\eqref{s1}, note that by Lemma~\ref{largebeta},
\begin{equation}\label{s2}
\prob_{(0,2\beta^{-1})}\left(\tau_2(y) < \tau_2(\beta^{-1})\right) \le C\e^{-C'\beta y}.
\end{equation}
Now, for the second term on the right side of~\eqref{s1}, we will consider two cases depending on whether $y\in[4\beta^{-1}, \beta/4]$ or $y\geq \beta/4$.
When $y \in [4\beta^{-1}, \beta/4]$, by Lemma \ref{middle},
\begin{equation}\label{detail2}
\mathbb{E}_{(0, y)}\left(\tau_2(\beta^{-1})\right) \le C\log(\beta y).
\end{equation}
For $y\geq \beta/4$, note that
\begin{multline}\label{eq:s2-1}
\mathbb{E}_{(0, y)}\left(\tau_2(\beta^{-1})\right) = \mathbb{E}_{(0, y)}\left(\tau_2(\beta/4)\right) + \mathbb{E}_{(0, y)}\left[(\sigma\left(\tau_2(\beta/4)\right) \wedge \tau_2(\beta^{-1})) - \tau_2(\beta/4)\right]\\
 + \mathbb{E}_{(0, y)}\left[\tau_2(\beta^{-1}) - (\sigma\left(\tau_2(\beta/4)\right) \wedge \tau_2(\beta^{-1}))\right],
\end{multline}
where recall that  $\sigma(t) = \inf\{s \ge t: Q_1(s) = 0\}$.
Now, by Lemma \ref{linfall}, $\mathbb{E}_{(0, y)}\left(\tau_2(\beta/4)\right) \le Cy/\beta$. 
Also, since $Q_2$ decreases exponentially when $Q_1$ is negative, 
$$\mathbb{E}_{(0, y)}\left[(\sigma\left(\tau_2(\beta/4)\right) \wedge \tau_2(\beta^{-1})) - \tau_2(\beta/4)\right] \le \log(\beta^2/4).$$ 
Furthermore, by the strong Markov property and Lemma \ref{middle},
\begin{align*}
\mathbb{E}_{(0, y)}\left[\tau_2(\beta^{-1}) - (\sigma\left(\tau_2(\beta/4)\right) \wedge \tau_2(\beta^{-1}))\right] &\le \sup_{z \in [\beta^{-1}, \beta/4]}\mathbb{E}_{(0,z)}(\tau_2((2\beta)^{-1}))\\
& \le \sup_{z \in [\beta^{-1}, \beta/4]}C \log(2\beta z) \le C\log(\beta^2/2).
\end{align*}
Thus, using the above bounds in~\eqref{eq:s2-1} we obtain for $y\geq \beta/4$,
\begin{equation}\label{detail1}
\mathbb{E}_{(0, y)}\left(\tau_2(\beta^{-1})\right) \le C\left(\frac{y}{\beta} + \log \beta \right).
\end{equation}
Using \eqref{s2}, \eqref{detail2}, and \eqref{detail1} in~\eqref{s1}, and the lower bound on $\mathbb{E}_{(0, 2\beta^{-1})}\left(\Xi_0\right)$ obtained in Lemma~\ref{renexp}, we get for $y \ge 4\beta^{-1}$,
\begin{equation}\label{upart1}
\pi(Q_2(\infty) \ge y) \le C_1^+\e^{-C_2^+\beta^2}\e^{-C_2^+ \beta y}
\end{equation}
for appropriate choice of $C_1^+, C_2^+$. 

\noindent
(b) Now, consider $y \in [4\beta \e^{-\mathcal{C}_1^-\e^{\mathcal{C}_2^-\beta^2}}, 4\beta^{-1})$. Then
\begin{eq}\label{s5}
\mathbb{E}_{(0, 2\beta^{-1})}\Big(\int_{0}^{\Xi_0}\mathbf{1}_{[Q_2(s) \ge y]}ds\Big) \le \mathbb{E}_{(0, 2\beta^{-1})}\Big(\tau_2(y/4)\Big) &+ \mathbb{E}_{(0,y/4)}\Big(\int_0^{\tau_2(2\beta^{-1})}\mathbf{1}_{[Q_2(s) \ge y/2]}ds\Big)\\
&\le C \log\Big(\frac{8}{\beta y}\Big),
\end{eq}
where the last step is a consequence of Lemma \ref{middle} and Lemma \ref{smallval}. This, along with the lower bound on $\mathbb{E}_{(0, 2\beta^{-1})}\left(\Xi_0\right)$ obtained in Lemma \ref{renexp}, gives for $y \in [4\beta \e^{-\mathcal{C}_1^-\e^{\mathcal{C}_2^-\beta^2}}, 4\beta^{-1})$,
\begin{equation}\label{upart2}
\pi(Q_2(\infty) \ge y) \le C_1^+\e^{-C_2^+\beta^2}\log\left(\frac{8}{\beta y}\right).
\end{equation}
It is straightforward to check that the upper bound claimed in the theorem follows from \eqref{upart1} and \eqref{upart2}. \\

\noindent
\textit{Proof of the lower bound.}
(a) As before, we will first consider $y \ge \beta^{-1}$. 
Writing $\tau_2' = \inf\{ t \ge \tau_2(2y): Q_2(t) = y\}$,
\begin{eq}\label{s6}
\mathbb{E}_{(0, 2\beta^{-1})}\left(\int_{0}^{\Xi_0}\mathbf{1}_{[Q_2(s) \ge y]}ds\right)
&\ge \mathbb{E}_{(0, 2\beta^{-1})}\left(\mathbf{1}_{[\tau_2(2y) < \tau_2(\beta^{-1})]}\left(\tau_2'-\tau_2(2y)\right) \right)\\
 &= \prob_{(0,2\beta^{-1})}\left(\tau_2(2y) < \tau_2(\beta^{-1})\right) \mathbb{E}_{(0, 2y)}\left(\tau_2(y)\right).
\end{eq}
As $2y \ge 2\beta^{-1}$, therefore, by Lemma~\ref{Q2lb},
\begin{equation}\label{s7}
\prob_{(0,2\beta^{-1}}\left(\tau_2(2y) < \tau_2(\beta^{-1})\right) \ge (1-\e^{-1})\e^{-\beta(2y-2\beta^{-1})}.
\end{equation}
Furthermore, as $Q_2(t) \ge S(t) \ge S(0) + \sqrt{2}W(t) - \beta t$ for all $t$, the hitting time of level $y$ for $Q_2$ when $(Q_1(0),Q_2(0))=(0,2y)$ is stochastically bounded below by the hitting time of level $y$ by $2y + \sqrt{2}W(t) - \beta t$. Therefore,
\begin{equation}\label{s8}
\mathbb{E}_{(0, 2y)}\left(\tau_2(y)\right) \ge C \frac{y}{\beta} \ge C\frac{1}{\beta^2},
\end{equation}
where $C$ does not depend on $\beta, y$. Using \eqref{s6}, \eqref{s7} and the upper bound on $\mathbb{E}_{(0, 2\beta^{-1})}\left(\Xi_0\right)$ obtained in Lemma \ref{renexp}, we obtain for $y \ge \beta^{-1}$,
\begin{equation}\label{lpart1}
\pi(Q_2(\infty) \ge y) \ge C_1^-\e^{-C_2^-\beta^2}\e^{-C_2^- \beta y}
\end{equation}
for appropriate choice of $C_1^-, C_2^-$. 

\noindent
(b) Now we consider $y \in  [4\beta \e^{-\mathcal{C}_1^-\e^{\mathcal{C}_2^-\beta^2}}, \beta^{-1})$.
As $Q_2(t) \ge Q_2(s)\e^{-(t-s)}$ for any $0 \le s<t$, $\tau_2(y) \ge \log(2/(\beta y))$ (when $(Q_1(0), Q_2(0)) = (0,2\beta^{-1})$). Therefore,
\begin{eq}\label{s9}
&\mathbb{E}_{(0, 2\beta^{-1})}\Big(\int_{0}^{\Xi_0}\mathbf{1}_{[Q_2(s) \ge y]}ds\Big) \ge \mathbb{E}_{(0, 2\beta^{-1})}\left(\alpha_2 \wedge \tau_2(y)\right) \\
&\ge \mathbb{E}_{(0, 2\beta^{-1})}\Big(\alpha_2 \wedge \log\Big(\frac{2}{\beta y}\Big)\Big)
\ge \log\Big(\frac{2}{\beta y}\Big)\prob_{(0,2\beta^{-1})}\Big(\alpha_2 \ge \log\Big(\frac{2}{\beta y}\Big)\Big).
\end{eq}
Define the stopping times:
$$
\mathbf{G}_1 = \inf\{t \ge \alpha_1: Q_1(t)=0\}, \ \ \mathbf{G}_2 = \inf\{t \ge \mathbf{G}_1: Q_1(t)=-\beta/4\}.
$$
$Q_2$ is decreasing on $[\alpha_1, \mathbf{G}_1]$ and $Q_1(t) \le S(t) \le S(\mathbf{G}_1) + \sqrt{2}(W(t) - W(\mathbf{G}_1)) -3\beta (t-\mathbf{G}_1)/4$ for $t \in [\mathbf{G}_1, \mathbf{G}_2]$. As $Q_2(t)=2\beta^{-1}$ for some $t \in [\mathbf{G}_1, \mathbf{G}_2]$ if and only if $S(t)=2\beta^{-1}$, therefore by applying the strong Markov property at $\mathbf{G}_1$, for any $\beta \ge 2$,
\begin{align*}
&\prob_{(0,2\beta^{-1})}\Big(\sup_{t \in [0, \mathbf{G}_2]}Q_2(t) < 2\beta^{-1}\Big) \ge \inf_{z \in (0, \beta^{-1})}\prob_{(0,z)}\left(\tau_1(-\beta/4) < \tau_2(2\beta^{-1})\right)\\
&\ge \inf_{z \in (0, \beta^{-1})}\prob_{(0,z)}\left(S(t) \text{ hits } -\beta/4 \text{ before } 2\beta^{-1}\right) \ge \prob\left(\sqrt{2}W(t) -3\beta t/4 \text{ hits } -\beta/2 \text{ before } \beta^{-1}\right)\\
&=\frac{\e^{3/4}-1}{\e^{3/4} - \e^{-3\beta^2/8}} \ge 1- \e^{-3/4}.
\end{align*}
By applying the strong Markov property at $\mathbf{G}_2$, for $\beta \ge \beta_0$ for sufficiently large $\beta_0 \ge 1$,
\begin{eq}\label{s10}
&\prob_{(0,2\beta^{-1})}\Big(\alpha_2 \ge \log\left(\frac{2}{\beta y}\right)\Big) \ge \prob_{(0,2\beta^{-1})}\Big(\sup_{t \in [0, \mathbf{G}_2]}Q_2(t) < 2\beta^{-1}, \alpha_2 - \mathbf{G}_2\ge \log\left(\frac{2}{\beta y}\right)\Big)\\
&\ge \prob_{(0,2\beta^{-1})}\Big(\sup_{t \in [0, \mathbf{G}_2]}Q_2(t) < 2\beta^{-1}\Big)\inf_{z \in (0, 2\beta^{-1})}\prob_{(-\beta/4,z)}\Big(\tau_1(0) \ge \log\left(\frac{2}{\beta y}\right)\Big)\\
&\ge (1- \e^{-3/4})\inf_{z \in (0, 2\beta^{-1})}\prob_{(-\beta/4,z)}\left(\tau_1(0) \ge \mathcal{C}_1^-\e^{\mathcal{C}_2^-\beta^2}\right) \ \text{(as $y \ge 4\beta \e^{-\mathcal{C}_1^-\e^{\mathcal{C}_2^-\beta^2}}$)}\\
&\ge (1- \e^{-3/4})\left(1-\mathcal{D}_1 \e^{-\mathcal{D}_2\beta^2}\right) \quad \text{(by \eqref{problb})}\qquad \ge \frac{1}{2}(1- \e^{-3/4}).
\end{eq}
From \eqref{s9} and \eqref{s10}, we obtain
\begin{equation*}
\mathbb{E}_{(0, 2\beta^{-1})}\left(\int_{0}^{\Xi_0}\mathbf{1}_{[Q_2(s) \ge y]}ds\right) \ge C \log\left(\frac{2}{\beta y}\right),
\end{equation*}
which, along with the upper bound on the expectation of the renewal time obtained in Lemma \ref{renexp}, gives us for $y \in  [4\beta \e^{-\mathcal{C}_1^-\e^{\mathcal{C}_2^-\beta^2}}, \beta^{-1})$,
\begin{equation}\label{lpart2}
\pi(Q_2(\infty) \ge y) \ge C_1^-\e^{-C_2^-\beta^2}\log\left(\frac{2}{\beta y}\right).
\end{equation}
It is straightforward to check that the lower bound claimed in the theorem follows from \eqref{lpart1} and \eqref{lpart2}. The $L^p$ convergence claimed in the theorem is immediate from the upper bound.
\end{proof}

\begin{proof}[Proof of Corollary~\ref{intermit}]
From the lower bound in Theorem \ref{largestat},
$$
\mathbb{E}_{\pi}\left(Q_2(\infty)\right) \ge \int_{\beta^{-1}}^{\infty}C_1^-\e^{-C_2^-\beta^2}\e^{-C_2^- \beta y}dy = \Big(\frac{C_1^-\e^{-C_2^-}}{C_2^-\beta}\Big)\e^{-C_2^-\beta^2},
$$
which proves the lower bound on the expectation of $Q_2(\infty)$. 
To get the upper bound, we will first prove the condensation result. Note that from the upper bound in Theorem \ref{largestat}, it is clear that if we pick a positive constant $\mathcal{C}^*$ (not depending on $\beta$) such that
$$
\e^{-\e^{\mathcal{C}^*\beta^2}} \ge \Big(2\beta^{-1}\e^{-\e^{C_2^+\beta^2/2}}\Big) \vee \Big(4\beta \e^{-\mathcal{C}_1^-\e^{\mathcal{C}_2^-\beta^2}}\Big),
$$
then for all $\beta \ge \beta_0$ sufficiently large,
$$
\pi\Big(Q_2(\infty) \ge \e^{-\e^{\mathcal{C}^*\beta^2}}\Big) \le 2C_1^+\e^{-C_2^+}\e^{-C_2^+\beta^2/2}.
$$
This, in turn, gives the upper bound on the expectation using the upper bound in Theorem \ref{largestat} as follows:
\begin{align*}
\mathbb{E}_{\pi}\left(Q_2(\infty)\right) &\le \int_0^{\e^{-\e^{\mathcal{C}^*\beta^2}}}\pi\left(Q_2(\infty) \ge y\right)dy + \int_{\e^{-\e^{\mathcal{C}^*\beta^2}}}^{\beta^{-1}}\pi\left(Q_2(\infty) \ge y\right)dy + \int_{\beta^{-1}}^{\infty}\pi\left(Q_2(\infty) \ge y\right)dy\\
&\le \e^{-\e^{\mathcal{C}^*\beta^2}} + \frac{2C_1^+\e^{-C_2^+}}{\beta}\e^{-C_2^+\beta^2/2}+\left(\frac{C_1^+\e^{-C_2^+}}{C_2^+\beta}\right)\e^{-C_2^+\beta^2},
\end{align*}
which proves the upper bound on the expectation.
\end{proof}

\begin{proof}[Proof of Proposition~\ref{Q1fin}]
Initiate the diffusion process at stationarity, i.e., $(Q_1(0), Q_2(0))$ is distributed as the steady state distribution $\pi$. To avoid more notation, we will use $\mathbb{E}_{\pi}$ to also denote the expectation operator corresponding to the law of the stationary diffusion process on the path space with initial distribution $\pi$.  
For any $n \ge 1$, applying Ito's formula to $(Q_1(t)+\beta)^{2n}$, we obtain
\begin{eq}\label{ito1}
 (Q_1(t) + \beta)^{2n} &= (Q_1(0) + \beta)^{2n} + 2n\sqrt{2}\int_0^t (Q_1(s) + \beta)^{2n-1}dW(s) -2n\int_0^t (Q_1(s) + \beta)^{2n}ds\\
 &\hspace{2.25cm}+ 2n\int_0^t (Q_1(s) + \beta)^{2n-1}Q_2(s)ds - 2n\int_0^t (Q_1(s) + \beta)^{2n-1}dL(s)\\ 
 &\hspace{6.75cm}+ 2n(2n-1)\int_0^t (Q_1(s) + \beta)^{2n-2}ds.
\end{eq}
By \cite[Theorem 2.1]{BM18}, for any $\beta>0$, $\pi$ has an exponential tail in $Q_2$ and a Gaussian tail in $Q_1$, and hence, for any $m, n \ge 1$, $\mathbb{E}_{\pi}\left(|Q_1(0)|^m|Q_2(0)|^n)\right) < \infty$. 
From this observation, we conclude that the local martingale $\int_0^tQ_1^m(s)Q_2^n(s)dW(s)$ has a finite expected quadratic variation for each $t$ and thus, by \cite[Pg. 73, Corollary 3]{Protter05}, it is indeed a true martingale having zero expectation for each $t$. 
Further, note that the times of increase of $L$ are precisely the times $s$ when $Q_1(s) = 0$ and therefore, we can replace the integral $2n\int_0^t (Q_1(s) + \beta)^{2n-1}dL(s)$ above by $2n\int_0^t \beta^{2n-1}dL(s)$. 
Moreover, as the initial distribution is the stationary distribution $\pi$, for any integers $k,l \ge 0$ and any $t \ge 0$, $\mathbb{E}_{\pi}\left((Q_1(t) + \beta)^kQ_2(t)^l\right) = \mathbb{E}_{\pi}\left((Q_1(0) + \beta)^kQ_2(0)^l\right)$. Thus, taking expectation with respect to $\mathbb{E}_{\pi}$ in \eqref{ito1} and applying Fubini's theorem, we obtain for any $\beta>0, t >0$,
\begin{align*}
-2n\mathbb{E}_{\pi}(Q_1(0) + \beta)^{2n} &+ 2n\mathbb{E}_{\pi}((Q_1(0) + \beta)^{2n-1}Q_2(0)) \\
&+ 2n(2n-1)\mathbb{E}_{\pi}(Q_1(0) + \beta)^{2n-2}
 - 2n\beta^{2n-1}\frac{\mathbb{E}_{\pi}(L(t))}{t} = 0.
\end{align*}
Note that $L(t) \ge 0$ for all $t \ge 0$. Moreover, as $Q_1(t) + \beta \le \beta$ and $2n-1$ is odd, $(Q_1(t) + \beta)^{2n-1} \le \beta^{2n-1}$. Using these observations and Corollary \ref{intermit} in the above equation,
\begin{align*}
2n\mathbb{E}_{\pi}(Q_1(0) + \beta)^{2n} &\le 2n\beta^{2n-1}\mathbb{E}_{\pi}(Q_2(0)) + 2n(2n-1)\mathbb{E}_{\pi}(Q_1(0) + \beta)^{2n-2}\\
&\le 2n\beta^{2n-1}\e^{-C_2\beta^2} + 2n(2n-1)\mathbb{E}_{\pi}(Q_1(0) + \beta)^{2n-2}.
\end{align*}
The lemma now follows by induction.
\end{proof}
\begin{proof}[Proof of Theorem~\ref{betainfty}]
For cleaner notation, we will suppress the dependence of the stationary distribution on $\beta$. 

We will use the method of moments. For any $n \ge 0$, the $2n$-th moment of the standard normal distribution is given by $m^*_{2n} = \frac{(2n)!}{2^n n!}$ and $(2n+1)$-moment is $m^*_{2n+1}=0$. 
Thus, this distribution has moderately growing moments (that is, $m^*_n \le AC^n n!$ for some $A,C>0$ and all integers $n \ge 1$) in the sense of Definition 2.51 of \cite{kirsch2015}. By \cite[Theorem 2.56]{kirsch2015}, it suffices to prove that for all $n \ge 1$, $\mathbb{E}_{\pi}\left((Q_1(\infty) + \beta)^n\right) \rightarrow m^*_n$ as $\beta \rightarrow \infty$. 

As before, let the diffusion process start at stationarity, i.e., $(Q_1(0), Q_2(0))$ is distributed as the stationary distribution $\pi$. 
For any $n \ge 2$, applying Ito's formula to $(Q_1(t)+\beta)^n$, taking expectation with respect to $\mathbb{E}_{\pi}$, and then applying Fubini's theorem as in the proof of Proposition~\ref{Q1fin} above, we obtain for any $\beta>0, t >0$,
\begin{eq}\label{cl1}
-n\mathbb{E}_{\pi}(Q_1(0) + \beta)^n &+ n\mathbb{E}_{\pi}((Q_1(0) + \beta)^{n-1}Q_2(0)) \\
&+ n(n-1)\mathbb{E}_{\pi}(Q_1(0) + \beta)^{n-2}
- n\beta^{n-1}\frac{\mathbb{E}_{\pi}(L(t))}{t} = 0.
\end{eq}
From the evolution equation of $Q_2(t)$, it readily follows that for any $t>0$,
$$
\frac{\mathbb{E}_{\pi}(L(t))}{t} = \mathbb{E}_{\pi}(Q_2(0)).
$$
From this observation and Corollary \ref{intermit},
\begin{equation}\label{cl2}
n\beta^{n-1}\frac{\mathbb{E}_{\pi}(L(t))}{t} \le n\beta^{n-1}\e^{-C_2\beta^2} \rightarrow 0 \ \text{ as } \beta \rightarrow \infty.
\end{equation}
By Proposition \ref{Q1fin}, there is $\beta_0>0$ such that for all $\beta \ge \beta_0$, the following holds: for each $n\ge 2$, there is $C'_n >0$ not depending on $\beta$ such that $\mathbb{E}_{\pi}\left((Q_1(0) + \beta)^{2n-2}\right) \le C'_n$. Moreover, from the upper tail estimate in Theorem \ref{largestat}, $\mathbb{E}_{\pi}(Q_2^2(0)) \rightarrow 0$ as $\beta \rightarrow \infty$. Using these observations along with the Cauchy-Schwarz inequality, we have
\begin{equation}\label{cl3}
\mathbb{E}_{\pi}(|Q_1(0) + \beta|^{n-1}Q_2(0)) \le \left(\mathbb{E}_{\pi}((Q_1(0) + \beta)^{2n-2})\right)^{1/2}\left(\mathbb{E}_{\pi}Q_2^2(0)\right)^{1/2} \rightarrow 0 \ \text{ as } \beta \rightarrow \infty.
\end{equation}
We proceed by induction. It follows readily from the evolution equation of $Q_1(t) + Q_2(t)$ that $\mathbb{E}_{\pi}(Q_1(0) + \beta) = 0$. Hence using induction along with \eqref{cl2} and \eqref{cl3} in \eqref{cl1}, we conclude that for each $n \ge 0$, $\lim_{\beta \rightarrow \infty} \mathbb{E}_{\pi}\left((Q_1(0) + \beta)^{2n+1}\right)$ exists and equals zero. Using $n=2$ in \eqref{cl1} along with \eqref{cl2} and \eqref{cl3}, it follows that $\lim_{\beta \rightarrow \infty} \mathbb{E}_{\pi}\left((Q_1(0) + \beta)^{2}\right)$ exists and equals $2$. Using this and induction along with \eqref{cl2} and \eqref{cl3} in \eqref{cl1}, we conclude that for each $n \ge 1$, $\lim_{\beta \rightarrow \infty} \mathbb{E}_{\pi}\left((Q_1(0) + \beta)^{2n}\right)$ exists and equals $m^*_{2n}$, completing the proof of the theorem.
\end{proof}

\section{Analysis in the small-$\beta$ regime}\label{sec:small-beta}
In this section, we will investigate the behavior of the stationary distribution in the regime when $\beta \le \beta_0$ for sufficiently small $\beta_0$.
For any such fixed $\beta$, take $B=2M_0\beta^{-1}$  in \eqref{rendef}, where $M_0 = c_1'$ is a fixed constant (independent of $\beta$) that appears in Lemma \ref{lem:q2regeneration}.
As in the large-$\beta$ regime in Section~\ref{sec:analysis-large}, our analysis in the small-$\beta$ regime 
relies on several key hitting time estimates. 
We state these results on hitting times in Subsection~\ref{ssec:hitting-small}.
In Subsection~\ref{ssec:inter-small} we will provide upper and lower bounds in the expected inter-regeneration time $\mathbb{E}_{(0, 4M_0\beta^{-1})}\left(\Xi_0\right)$.
The hitting-time results of Subsections~\ref{ssec:hitting-small} will be used to obtain sharp bounds on the numerator on the right-side of~\eqref{eq:pi}, i.e., the amount of time the process spends on various regions within one renewal cycle.
Combining the results of Subsections~\ref{ssec:hitting-small} and \ref{ssec:inter-small} we prove in Subsection~\ref{ssec:proof-small}  the main results for the small-$\beta$ regime. 

As before, we will use $\beta_0$ to denote the upper bound on $\beta$ for the assertion of each of the following lemmas to hold (specific upper bounds change between lemmas and a common upper bound is obtained by taking the minimum of these bounds). 
Also, in the proofs, $C, C', C", C_1, C_2$ will represent generic positive constants that do not depend on $\beta$ whose values will change between steps and from line to line.

\subsection{Hitting time estimates}\label{ssec:hitting-small}
As mentioned in Subsection \ref{ssec:small}, the main challenge in this regime is to patch up the different behaviors of $Q_1$ for small and large values of $Q_2$. We record the following estimates, which describe these different behaviors individually.
\begin{lemma}\label{qoneexc-i}
There exist $\beta_0 \in (0,1)$ and positive constant $C$, such that the following hold for all fixed $\beta \in (0, \beta_0]$.
For every $x \ge 2\beta^{1/4}$,
$$\sup_{0<y \le \beta^{-1/2}}\mathbb{E}_{(0,y)}\Big(\int_0^{\tau_2(2\beta^{-1/2})}\mathbf{1}_{[Q_1(s) \le -x]}ds\Big) \le C\beta^{-5/4}\e^{-(x-\beta)^2/8}.$$
\end{lemma}

Lemma \ref{qoneexc-i} captures the behavior of $Q_1$ when $Q_2 \le 2\beta^{-1/2}$. In this region, $Q_1$ behaves like a reflected Ornstein-Uhlenbeck process resulting in the Gaussian exponent in the bound.

\begin{lemma}\label{qoneexc-ii}
There exist $\beta_0 \in (0,1)$ and positive constants $C, C'$, such that the following hold for all fixed $\beta \in (0, \beta_0]$.
For every $x \ge 2\beta$,
$$\mathbb{E}_{(0,4M_0\beta^{-1})}\Big(\int_0^{\tau_2(2M_0\beta^{-1})}\mathbf{1}_{[Q_1(s) \le -x]}ds\Big) \le Cx\beta^{-3}\e^{-C'\frac{x}{\beta}}.$$
\end{lemma}

Lemma \ref{qoneexc-ii} considers the region when $Q_2$ is large, namely $\Theta\left(\beta^{-1}\right)$. In this region, $Q_1$ behaves roughly like a reflected Brownian motion with drift $-\beta^{-1}$, resulting in the exponential decay of the tail, i.e., $\e^{-C'\frac{x}{\beta}}$ term in the bound.

\begin{lemma}\label{qoneexc-iii}
There exist $\beta_0 \in (0,1)$ and positive constants $C, C'$, such that the following hold for all fixed $\beta \in (0, \beta_0]$.
For every $x \ge 2\beta^{1/2}$,
$$\sup_{\beta^{-1/2} \le y \le 2M_0\beta^{-1}}\mathbb{E}_{(0,y)}\Big(\int_0^{\tau_2(\beta^{-1/2}) \wedge \tau_2(4M_0\beta^{-1})}\mathbf{1}_{[Q_1(s) \le -x]}ds\Big) \le Cx\beta^{-5/2}\e^{-C'\frac{x}{\sqrt{\beta}}}.$$
\end{lemma}

Lemma \ref{qoneexc-iii} patches up the behavior in the region where $Q_2$ is $O(\beta^{-1/2})$ with the behavior where $Q_2$ is $\Theta\left(\beta^{-1}\right)$. The resulting bound we obtain decays exponentially in $x/\sqrt{\beta}$.\\

Recall that  $\sigma(t) = \inf\{s \ge t: Q_1(s) = 0\}$. The next two lemmas give estimates for the time spent by $Q_1$ below some negative threshold before $Q_1$ hits zero, for the regions where $Q_2$ is $O(\beta^{-1/2})$ and $Q_2$ is $\Theta\left(\beta^{-1}\right)$, respectively.
\begin{lemma}\label{qoneexc-iv}
There exist $\beta_0 \in (0,1)$ and positive constants $C, C'$, such that the following hold for all fixed $\beta \in (0, \beta_0]$.
For every $x \ge 4\beta^{1/2}$,
\begin{multline*}
\sup_{\beta^{-1/2} \le y \le 4M_0\beta^{-1}}\mathbb{E}_{(0,y)}\Big(\Big(\int_{\tau_2(\beta^{-1/2})}^{\sigma\left(\tau_2(\beta^{-1/2})\right)}\mathbf{1}_{[Q_1(s) \le -x]}ds\Big)\mathbf{1}_{[\tau_2(\beta^{-1/2}) < \tau_2(4M_0\beta^{-1})]}\Big)\\
\le C\left(\beta^{-7/2}\e^{-C'\frac{x}{\sqrt{\beta}}} + \e^{-(x-2\beta)^2/4}\right).
\end{multline*}
\end{lemma}

\begin{lemma}\label{qoneexc-v}
There exist $\beta_0 \in (0,1)$ and positive constants $C, C'$, such that the following hold for all fixed $\beta \in (0, \beta_0]$.
For every $x \ge 4\beta$,
$$
\mathbb{E}_{(0,4M_0\beta^{-1})}\Big(\int_{\tau_2(2M_0\beta^{-1})}^{\sigma\left(\tau_2(2M_0\beta^{-1})\right)}\mathbf{1}_{[Q_1(s) \le -x]}ds\Big)
\le C\left(\beta^{-4}\e^{-C'\frac{x}{\beta}} + \e^{-(x-2\beta)^2/4}\right).
$$
\end{lemma}
Lemmas~\ref{qoneexc-i}--\ref{qoneexc-v} will be combined to compute upper and lower bounds on the expected regeneration times and to estimate the expected time spent by $Q_1$ below $-x$ between two successive regeneration times, for any $x \ge 2\beta^{1/4}$. 
These, in turn, will be used to estimate $\pi(Q_1(\infty) \le -x)$ using the expression given in~\eqref{eq:pi}. 

The next lemma supplies a key technical estimate by giving an upper bound on the probability of $Q_2$ hitting a positive level $y \ge 8M_0\beta^{-1}$ between two successive renewal times. 

\begin{lemma}\label{smallbetafluc}
There exist positive constants $D_S, D_S', M_0, \beta_0$ such that for all $\beta \le \beta_0$,
$$
\prob_{(0,4M_0\beta^{-1})}\left(\tau_2(y) < \tau_2(2M_0\beta^{-1})\right) \le D_S\e^{-D_S'\beta y}, \ \ y \ge 8M_0\beta^{-1}.
$$
\end{lemma}

\begin{remark}\label{rem:extn}\normalfont
Observe that Lemma~\ref{smallbetafluc} is a major improvement over \cite[Lemma 5.3]{BM18} (see the statement in Lemma \ref{Q2gebeta2}).
In Lemma \ref{Q2gebeta2} a similar tail-bound is given for $y \ge \beta^{-1} \log \beta^{-1}$. 
Lemma~\ref{smallbetafluc} extends this tail bound for all $y$ in the region $[\Theta(\beta^{-1}), \infty)$.
This extension is crucial in capturing the behavior of the stationary distribution near the steady-state mean of $Q_2$, which, as we will prove, is of order $\Theta\left(\beta^{-1}\right)$.
\end{remark}

Appendix \ref{app:small-aux} is devoted to the proofs of Lemmas \ref{qoneexc-i} -- \ref{qoneexc-v}, and Lemma \ref{smallbetafluc} is proved in Appendix \ref{app:lem5.6}.

\subsection{Bounds on the inter-regeneration times}\label{ssec:inter-small}
In this section we state and prove upper and lower bounds on the expected inter-regeneration times $\mathbb{E}_{(0,4M_0\beta^{-1})}\left(\Xi_0\right)$, which will be used in Subsection~\ref{ssec:proof-small} to prove the main results.

\begin{lemma}\label{regsmall}
There exist positive constants $E_1, E_2, \beta_0$ such that for all fixed $\beta \le \beta_0$,
$$
\frac{E_1}{\beta^2} \le \mathbb{E}_{(0,4M_0\beta^{-1})}\left(\Xi_0\right) \le \frac{E_2}{\beta^2}.
$$
\end{lemma}
Lemma~\ref{regsmall} should be contrasted with Lemma~\ref{renexp}, where the expected inter-regeneration time grows exponentially with $\beta^2$, instead of inverse-quadratically  when $\beta$ is small.
\begin{proof}[Proof of Lemma~\ref{regsmall}]
Let $(Q_1(0), Q_2(0)) = (0, 4M_0 \beta^{-1})$. Recall that $\Xi_0= \alpha_2$ 
where $\alpha_1$ and $\alpha_2$ are as defined in \eqref{rendef} with $B= 2M_0\beta^{-1}$.
The proof consists of two main parts: (i) First we establish upper and lower bounds on the expected value of $\alpha_1$, and (ii) Next, we obtain an upper bound on the expected value of $\alpha_2-\alpha_1 $.
As we will see, since the lower bound for $\alpha_1 $ matches with that of 
$ \Xi_0$ as stated in the lemma, this will complete the proof.\\

\noindent
(i) {\em Upper and lower bounds for $\alpha_1$.}
As $Q_2(t) \ge S(t) \ge S(0) + \sqrt{2}W(t) - \beta t$ for all $t \ge 0$, $\alpha_1$ stochastically dominates the hitting time of level $2M_0\beta^{-1}$ by a Brownian motion with drift $-\beta$ starting from $4M_0\beta^{-1}$. Therefore,
\begin{equation}\label{smallr1}
\mathbb{E}_{(0,4M_0\beta^{-1})}\left(\alpha_1\right) \ge \frac{C}{\beta^2}.
\end{equation}
From part (ii) of Lemma \ref{lem:q2regeneration}, for $\beta$ small enough to ensure $3M_0\beta^{-1} \ge 1$ and for $t \ge C\beta^{-2}$,
\begin{multline*}
\prob_{(0,4M_0\beta^{-1})}\left(\alpha_1 > t\right) \le \prob_{(0,4M_0\beta^{-1})}\big(\inf_{s\leq t}Q_2(s) > M_0\beta^{-1}\big)\\
\le c'_3\left(\exp(-c'_2\beta^{-2/5}t^{1/5}) + \exp(-c'_2\beta^2 t)
+ \beta^{-2}\exp(-c'_2 t)\right)
\end{multline*}
for positive constants $C, c_2',c_3'$ not depending on $\beta$. Using this bound, we obtain
\begin{equation}\label{smallr2}
\mathbb{E}_{(0,4M_0\beta^{-1})}\left(\alpha_1\right) = \int_0^{\infty} \prob_{(0,4M_0\beta^{-1})}\left(\alpha_1 > t\right)dt \le \frac{C'}{\beta^2}.
\end{equation}

\noindent
(ii) \emph{ Upper bound for $\alpha_2-\alpha_1$.}
The proof follows the similar notation and arguments as in the proof of Lemma \ref{renexp}. 
Recall $\alpha_{1,-\beta} = \inf\{t \ge \alpha_1: Q_1(t) = -\beta\}$. Then by repeating the computation exactly along the lines of \eqref{ren21}, we have
\begin{multline}\label{smallr3}
\mathbb{E}_{(0,4M_0\beta^{-1})}\left(\alpha_2 - \alpha_1\right) 
\le \sup_{y \in (0, 4M_0\beta^{-1})}\mathbb{E}_{(-\beta,y)}\left(\tau_2(4M_0\beta^{-1})\right)\\
 + \mathbb{E}_{(0,4M_0\beta^{-1})}\left(\mathbb{E}_{(Q_1(\alpha_1), 2M_0\beta^{-1})}\left(\tau_1(-\beta) \wedge \tau_2(4M_0\beta^{-1})\right)\right).
\end{multline}
In the rest of the proof we will estimate the two expectation on the right side of \eqref{smallr3}.
We start with $\sup_{y \in (0, 4M_0\beta^{-1})}\mathbb{E}_{(-\beta,y)}\left(\tau_2(4M_0\beta^{-1})\right)$. Take any $y \in (0, 4M_0\beta^{-1})$ and set the starting configuration as $(Q_1(0), Q_2(0))=(-\beta, y)$. Recall $S(t) = Q_1(t) + Q_2(t)$. Define the following stopping times: $\mathcal{S}_0 = 0$ and for $k \ge 0$,
\begin{align*}
\mathcal{S}_{2k+1} &= \inf\{ t \ge \mathcal{S}_{2k}: S(t) = 4M_0\beta^{-1} \text{ or } S(t) \le -\beta^{-1}\},\\
\mathcal{S}_{2k+2} &= \inf\{ t \ge \mathcal{S}_{2k+1}: S(t) = 4M_0\beta^{-1} \text{ or } S(t) = -\beta\}.
\end{align*}
Let $N_{\mathcal{S}} = \inf\{k \ge 0: S(\mathcal{S}_{2k+1}) = 4M_0 \beta^{-1}\}$. As $S(t) \ge S(0) + \sqrt{2}W(t) - \beta t$ for all $t \ge 0$, for any $k \ge 0$ and $\beta \le 2^{-1/2}$,
\begin{multline*}
\prob_{(-\beta, y)}\left(\sup_{t \in [\mathcal{S}_{2k}, \mathcal{S}_{2k+1}]}S(t) \ge 4M_0 \beta^{-1}\right) \ge \prob\left(\sqrt{2}W(t) - \beta t \text{ hits } (4M_0 + 1)\beta^{-1} \text{ before } -(2 \beta)^{-1}\right)\\
= \frac{1-\e^{-1/2}}{\e^{4M_0 + 1} - \e^{-1/2}} = p > 0
\end{multline*} 
using the fact that the scale function for $\sqrt{2}W(t) - \beta t$ is $s(x) = \e^{\beta x}$, where $p$ does not depend on $\beta, y$. This immediately gives us
$$
\prob_{(-\beta, y)}\left(N_{\mathcal{S}} \ge n\right) \le (1-p)^n, \ \ n \ge 1,
$$
which implies
\begin{equation}\label{smallr4}
\sup_{y \in (0, 4M_0\beta^{-1})}\mathbb{E}_{(-\beta, y)}\left(N_{\mathcal{S}}\right) \le C_1.
\end{equation}
To estimate $\mathbb{E}_{(-\beta, y)}\left(\mathcal{S}_1\right)$, observe that for $n \ge 1$,
\begin{eq}\label{three0}
&\prob_{(-\beta, y)}\left(\mathcal{S}_1 \ge n \beta^{-2}\right)\\
&\le \mathbb{E}_{(-\beta, y)}\Big(\mathbf{1}_{[\mathcal{S}_1 \ge (n-1)\beta^{-2}]}\sup_{\substack{x \in [-\beta^{-1}, 0], \\ y \in (0, 4M_0\beta^{-1}]}}\prob_{(x,y)}\left(S(t) < 4M_0\beta^{-1} \text{ for all } t \in [0, \beta^{-2}]\right) \Big)\\
&\le \mathbb{E}_{(-\beta, y)}\Big(\mathbf{1}_{[\mathcal{S}_1 \ge (n-1)\beta^{-2}]}\prob\Big(\sup_{t < \infty}(\sqrt{2}W(t) - \beta t) < (4M_0+1)\beta^{-1} \Big) \Big)\\
&=\prob_{(-\beta, y)}\left(\mathcal{S}_1 \ge (n-1) \beta^{-2}\right)(1- \e^{-(4M_0+1)}) \le (1- \e^{-(4M_0+1)})^n
\end{eq}
where the last step follows by induction. This gives us
\begin{multline}\label{smallr6}
\mathbb{E}_{(-\beta, y)}\left(\mathcal{S}_1\right) = \int_0^{\infty}\prob_{(-\beta, y)}\left(\mathcal{S}_1 \ge t\right)dt
\le 1 + \beta^{-2}\sum_{n=1}^{\infty}\prob_{(-\beta, y)}\left(\mathcal{S}_1 \ge n \beta^{-2}\right) \le \frac{C'}{\beta^2}
\end{multline}
where $C'$ does not depend on $\beta, y$. 
For $t \le \tau_1(0)$, $Q_1^*(t) = Q_1(t) + \beta$ is stochastically bounded below by an Ornstein-Uhlenbeck process
$$
X^*(t) = Q_1(0) + \beta + \sqrt{2}W(t)  - \int_0^t X^*(s)ds,
$$
which has the Doob representation $X^*(t)  =\left(Q_1(0) + \beta\right)\e^{-t} + \e^{-t}W^*\left(\e^{2t}-1\right)$ (where $W^*$ is a standard Brownian motion). Furthermore, note that if $Q_2(0) \le 4M_0 \beta^{-1}$, then for all $t \le \tau_2(4M_0 \beta^{-1})$, $Q_1(t) \ge S(t) - 4M_0\beta^{-1}$. From these facts, it is straightforward to check that $\mathcal{S}_{2} - \mathcal{S}_{1}$ is stochastically dominated by the hitting time of level $-\beta$ by the Ornstein-Uhlenbeck process $X^*$ defined above taking $Q_1(0) = -(4M_0 + 1)\beta^{-1}$ and hence,
\begin{equation}\label{smallr5}
\mathbb{E}_{(-\beta, y)}\left(\mathcal{S}_{2} - \mathcal{S}_{1}\right) \le \frac{C''}{\beta}
\end{equation}
where $C''$ does not depend on $\beta, y$. 
Combining \eqref{smallr5} and \eqref{smallr6}, we obtain,
\begin{equation}\label{smallr7}
\sup_{y \in (0, 4M_0\beta^{-1})}\mathbb{E}_{(-\beta, y)}\left(\mathcal{S}_2\right) \le \frac{C_2}{\beta^2}.
\end{equation}
Therefore, using \eqref{smallr4}, \eqref{smallr7} and the strong Markov property, 
\begin{eq}\label{smallmain1}
&\sup_{y \in (0, 4M_0\beta^{-1})}\mathbb{E}_{(-\beta,y)}\left(\tau_2(4M_0\beta^{-1})\right) = \sup_{y \in (0, 4M_0\beta^{-1})}\mathbb{E}_{(-\beta,y)}\left(\mathcal{S}_{2N_{\mathcal{S}} + 1}\right)\\
&=  \sup_{y \in (0, 4M_0\beta^{-1})}\mathbb{E}_{(-\beta,y)}\left(\mathcal{S}_{2N_{\mathcal{S}} + 2}\right) = \sup_{y \in (0, 4M_0\beta^{-1})}\mathbb{E}_{(-\beta,y)}\Big(\sum_{j=0}^{\infty}(\mathcal{S}_{2j + 2} - \mathcal{S}_{2j})\mathbf{1}_{[N_{\mathcal{S}}\ge j]}\Big)\\
&  \le \sup_{y \in (0, 4M_0\beta^{-1})}\mathbb{E}_{(-\beta,y)}\Big(\sum_{j=0}^{\infty}\mathbf{1}_{[N_{\mathcal{S}} \ge j]}\sup_{z \in (0, 4M_0\beta^{-1})}\mathbb{E}_{(-\beta,z)}(\mathcal{S}_{2})\Big)\\
&  = \sup_{y \in (0, 4M_0\beta^{-1})}\mathbb{E}_{(-\beta, y)}\left(\mathcal{S}_2\right)\sup_{y \in (0, 4M_0\beta^{-1})}\mathbb{E}_{(-\beta, y)}\left(N_{\mathcal{S}}\right) \le \frac{C_1 C_2}{\beta^2}.
\end{eq}
Now we estimate $\mathbb{E}_{(0,4M_0\beta^{-1})}\left(\mathbb{E}_{(Q_1(\alpha_1), 2M_0\beta^{-1})}\left(\tau_1(-\beta) \wedge \tau_2(4M_0\beta^{-1})\right)\right)$. Observe that
\begin{eq}\label{three1}
&\mathbb{E}_{(0,4M_0\beta^{-1})}\left(\mathbb{E}_{(Q_1(\alpha_1), 2M_0\beta^{-1})}\left(\tau_1(-\beta) \wedge \tau_2(4M_0\beta^{-1})\right)\right)\\
&\hspace{2cm} \le \sup_{x \in [-\beta,0], y \in (0, 4M_0\beta^{-1})}\mathbb{E}_{(x,y)}\left(\tau_1(-\beta) \wedge \tau_2(4M_0\beta^{-1})\right)\\
 &\hspace{4cm}+ \mathbb{E}_{(0,4M_0\beta^{-1})}\left(\mathbf{1}_{[Q_1(\alpha_1) < -\beta]}\mathbb{E}_{(Q_1(\alpha_1), 2M_0\beta^{-1})}\left(\tau_1(-\beta)\right)\right).
\end{eq}
Writing for $n \ge 1$
\begin{multline*}
 \sup_{x \in [-\beta,0], y \in (0, 4M_0\beta^{-1})}\prob_{(x,y)}\left(\tau_1(-\beta) \wedge \tau_2(4M_0\beta^{-1}) \ge n\beta^{-2}\right)\\
\le \sup_{x \in [-\beta,0], y \in (0, 4M_0\beta^{-1})}\mathbb{E}_{(-\beta, y)}\left(\mathbf{1}_{[\tau_1(-\beta) \wedge \tau_2(4M_0\beta^{-1}) \ge (n-1)\beta^{-2}]}\right.\\
\left.\sup_{z \in [-\beta, 0], y \in (0, 4M_0\beta^{-1}]}\prob_{(z,y)}\left(S(t) < 4M_0\beta^{-1} \text{ for all } t \in [0, \beta^{-2}]\right) \right)
\end{multline*}
and following the computations in \eqref{three0}and \eqref{smallr6}, we obtain
\begin{equation}\label{three2}
\sup_{x \in [-\beta,0], y \in (0, 4M_0\beta^{-1})}\mathbb{E}_{(x,y)}\left(\tau_1(-\beta) \wedge \tau_2(4M_0\beta^{-1})\right) \le \frac{C_3}{\beta^2}.
\end{equation}
Therefore, to complete the proof of the upper bound of $\mathbb{E}_{(0,4M_0\beta^{-1})}\left(\alpha_2 - \alpha_1\right)$, we need to estimate the second term on the right hand side of \eqref{three1}. 
From Claim~\eqref{cl:threeuse} note that for any $x<-\beta, y>0$,
\begin{equation*}
\mathbb{E}_{(x,y)}\left(\tau_1(-\beta)\right) \le C \log\left(2 + |x + \beta|\right).
\end{equation*}
Thus, by Jensen's inequality,
\begin{eq}\label{three3}
&\mathbb{E}_{(0,4M_0\beta^{-1})}\left(\mathbf{1}_{[Q_1(\alpha_1) < -\beta]}\mathbb{E}_{(Q_1(\alpha_1), \beta^{-1})}\left(\tau_1(-\beta)\right)\right)\\
 &\le \mathbb{E}_{(0,4M_0\beta^{-1})}\left(\log(2+ |Q_1(\alpha_1) + \beta|)\right)
\le \log\left(2+ \mathbb{E}_{(0,4M_0\beta^{-1})}\left(-Q_1(\alpha_1)\right) + \beta\right).
\end{eq}
Thus, we need an estimate on $\mathbb{E}_{(0, 4M_0\beta^{-1})}\left(-Q_1(\alpha_1)\right)$. Following the calculations in \eqref{threeest}, we obtain for $x \ge (8M_0\beta^{-1})^4$,
\begin{eq}\label{three4}
&\prob_{(0,4M_0\beta^{-1})}\left(Q_1(\alpha_1) \le -x\right)\\
&\le \prob_{(0,4M_0\beta^{-1})}\left(\tau_2(x^{1/4}) < \alpha_1\right) + \sup_{y \le x^{1/4}}\prob_{(-x/2, y)}\Big(\tau_1(0) \le \log\Big(\frac{\beta x^{1/4}}{2M_0}\Big)\Big)\\
&\hspace{3cm}+ \sup_{y \le x^{1/4}}\prob_{(-x/2, y)}\Big(\tau_1(-x) \le \log\Big(\frac{\beta x^{1/4}}{2M_0}\Big), \tau_1(-x) \le \tau_1(0)\Big).
\end{eq}
From Lemma~\ref{smallbetafluc}, we obtain $C, C', \beta_0>0$ such that for all $\beta \le \beta_0$ and all $x \ge (8M_0\beta^{-1})^4$,
\begin{equation*}
\prob_{(0,4M_0\beta^{-1})}\left(\tau_2(x^{1/4}) < \alpha_1\right) \le C\e^{-C'\beta x^{1/4}}.
\end{equation*}
By exactly the same argument used in deriving \eqref{qalpha2} and \eqref{qalpha3}, we obtain constants $C, C', \beta''$ such that for all $\beta \le \beta''$ and all $x \ge (8M_0\beta^{-1})^4$,
\begin{equation*}
\sup_{y \le x^{1/4}}\prob_{(-x/2, y)}\left(\tau_1(0) \le \log\left(\frac{\beta x^{1/4}}{2M_0}\right)\right) \le C\e^{-C'x}
\end{equation*}
and
\begin{equation*}
\sup_{y \le x^{1/4}}\prob_{(-x/2, y)}\left(\tau_1(-x) \le \log\left(\frac{\beta x^{1/4}}{2M_0}\right), \tau_1(-x) \le \tau_1(0)\right) \le C\e^{-C'x}.
\end{equation*}
Using the above three bounds in \eqref{three4}, we obtain $\beta'''>0$ such that for all $\beta \le \beta'''$ and all $x \ge (8M_0\beta^{-1})^4$,
$$
\prob_{(0,4M_0\beta^{-1})}\left(Q_1(\alpha_1) \le -x\right) \le C\e^{-C'\beta x^{1/4}}.
$$
which implies
\begin{equation}\label{three5}
\mathbb{E}_{(0, 4M_0\beta^{-1})}\left(-Q_1(\alpha_1)\right) = \int_0^{\infty}\prob_{(0,4M_0\beta^{-1})}\left(Q_1(\alpha_1) \le -x\right)dx \le \frac{C_4}{\beta^4}.
\end{equation}
Using \eqref{three5} in \eqref{three3},
\begin{equation}\label{three6}
\mathbb{E}_{(0,4M_0\beta^{-1})}\left(\mathbf{1}_{[Q_1(\alpha_1) < -\beta]}\mathbb{E}_{(Q_1(\alpha_1), \beta^{-1})}\left(\tau_1(-\beta)\right)\right)
\le \log\left(2+ \frac{C_4}{\beta^4} + \beta\right).
\end{equation}
Using \eqref{three2} and \eqref{three6} in \eqref{three1}, we obtain $\beta_s'>0$ such that for all $\beta \le \beta_s'$,
\begin{equation}\label{three7}
\mathbb{E}_{(0,4M_0\beta^{-1})}\left(\mathbb{E}_{(Q_1(\alpha_1), 2M_0\beta^{-1})}\left(\tau_1(-\beta) \wedge \tau_2(4M_0\beta^{-1})\right)\right) \le \frac{C_3}{\beta^2} + \log\left(2+ \frac{C_4}{\beta^4} + \beta\right) \le \frac{C_5}{\beta^2}.
\end{equation}
Using \eqref{smallmain1} and \eqref{three7} in \eqref{smallr3}, we obtain $\beta_0>0$ such that for all $\beta \le \beta_0$,
\begin{equation}\label{smallmain}
\mathbb{E}_{(0,4M_0\beta^{-1})}\left(\alpha_2 - \alpha_1\right) \le \frac{C_6}{\beta^2}.
\end{equation}
\eqref{smallr1}, \eqref{smallr2} and \eqref{smallmain} together prove the lemma.
\end{proof}

\subsection{Proofs of the main results}\label{ssec:proof-small}

\begin{proof}[Proof of Theorem~\ref{smallbetathm}]
As before, we will use Theorem \ref{stationary} with $B= 2M_0\beta^{-1}$.
To prove the upper bound on the stationary probability, observe that for $y \ge 8M_0\beta^{-1}$,
\begin{eq}\label{smalls1}
&\mathbb{E}_{(0, 4M_0\beta^{-1})}\Big(\int_{0}^{\Xi_0}\mathbf{1}_{[Q_2(s) \ge y]}ds\Big) = \mathbb{E}_{(0, 4M_0\beta^{-1})}\Big(\int_{0}^{\alpha_1}\mathbf{1}_{[Q_2(s) \ge y]}ds\Big)\\
&\le \mathbb{E}_{(0, 4M_0\beta^{-1})}\left(\mathbf{1}_{[\tau_2(y) < \tau_2(2M_0\beta^{-1})]}\left(\tau_2(2M_0\beta^{-1})-\tau_2(y)\right) \right)\\
&= \prob_{(0,4M_0\beta^{-1})}\left(\tau_2(y) < \tau_2(2M_0\beta^{-1})\right) \mathbb{E}_{(0, y)}\left(\tau_2(2M_0\beta^{-1})\right).
\end{eq}
From Lemma~\ref{smallbetafluc}, there exist positive constants $D_S, D_S', M_0, \beta_0$ such that for all $\beta \le \beta_0$, $y \ge 8M_0\beta^{-1}$,
\begin{equation}\label{smalls2}
\prob_{(0,4M_0\beta^{-1})}\left(\tau_2(y) < \tau_2(2M_0\beta^{-1})\right) \le D_S\e^{-D_S'\beta y}.
\end{equation}
From part (ii) of Lemma \ref{lem:q2regeneration} and recalling that $M_0 = c_1'$ and choosing $\beta$ small enough such that $6M_0\beta^{-1} \ge 1$, we obtain for $t\geq c'_4 \left(\frac{(y-M_0\beta^{-1})}{\beta} \vee \beta^{-2}\right)$,
\begin{align*}
\prob_{(0, y)}\left(\tau_2(2M_0\beta^{-1}) > t\right) &\le \prob_{(0,y)}\big(\inf_{s\leq t}Q_2(s) > M_0\beta^{-1}\big)\\
&\le c'_3\left(\exp(-c'_2\beta^{-2/5}t^{1/5}) + \exp(-c'_2\beta^2 t)
+ \beta^{-2}\exp(-c'_2 t)\right)
\end{align*}
for positive constants $c_1', c_2', c_3', c_4'$ not depending on $\beta$. From this, we obtain
\begin{equation}\label{smalls3}
\mathbb{E}_{(0, y)}\left(\tau_2(2M_0\beta^{-1})\right) = \int_0^{\infty}\prob_{(0, y)}\left(\tau_2(2M_0\beta^{-1}) > t\right)dt \le \frac{Cy}{\beta}
\end{equation}
where $C$ does not depend on $\beta, y$. Using \eqref{smalls2} and \eqref{smalls3} in \eqref{smalls1},
\begin{equation}\label{smalls4}
\mathbb{E}_{(0, 4M_0\beta^{-1})}\Big(\int_{0}^{\Xi_0}\mathbf{1}_{[Q_2(s) \ge y]}ds\Big) \le D_S\frac{Cy}{\beta}\e^{-D_S'\beta y}.
\end{equation}
From Lemma \ref{regsmall},
\begin{equation}\label{smalls5}
\mathbb{E}_{(0,4M_0\beta^{-1})}\left(\Xi_0\right) \ge \frac{E_1}{\beta^2}.
\end{equation}
Using the estimates \eqref{smalls4} and \eqref{smalls5} in the representation \eqref{eq:pi} of the stationary distribution, we obtain
$$
\pi(Q_2(\infty) \ge y) \le \frac{D_S C}{E_1} \ y\beta \e^{-D_S'\beta y}, \ \ y \ge 8M_0\beta^{-1}
$$
which proves the upper bound on the stationary probability claimed in the theorem. To prove the lower bound, note that for $y \ge 8M_0\beta^{-1}$, writing $\tau_2' = \inf\{ t \ge \tau_2(2y): Q_2(t) = y\}$,
\begin{eq}\label{lows1}
&\mathbb{E}_{(0, 4M_0\beta^{-1})}\Big(\int_{0}^{\Xi_0}\mathbf{1}_{[Q_2(s) \ge y]}ds\Big) \ge \mathbb{E}_{(0, 4M_0\beta^{-1})}\Big(\int_{0}^{\alpha_1}\mathbf{1}_{[Q_2(s) \ge 2y]}ds\Big)\\
&\ge \mathbb{E}_{(0, 4M_0\beta^{-1})}\left(\mathbf{1}_{[\tau_2(2y) < \tau_2(2M_0\beta^{-1})]}\left(\tau_2'-\tau_2(2y)\right) \right)\\
&= \prob_{(0,4M_0\beta^{-1})}\left(\tau_2(2y) < \tau_2(2M_0\beta^{-1})\right) \mathbb{E}_{(0, 2y)}\left(\tau_2(y)\right).
\end{eq}
From Lemma~\ref{Q2lb}, taking $B = 2M_0 \beta^{-1}$,
\begin{equation}\label{lows2}
\prob_{(0,4M_0 \beta^{-1})}\left(\tau_2(y) < \tau_2(2M_0 \beta^{-1})\right) \ge (1-\e^{-2M_0})\e^{-\beta(y-4M_0 \beta^{-1})}, \ \ y \ge 4M_0 \beta^{-1}.
\end{equation}
Recall that $Q_2(t) \ge S(t) \ge S(0) + \sqrt{2}W(t) - \beta t$ for all $t \ge 0$. Therefore, if the starting configuration is $(Q_1(0), Q_2(0)) = (0, 2y)$, $\tau_2(y)$ is stochastically lower bounded by the hitting time of level $-y$ by the process $(\sqrt{2}W(t) - \beta t)_{t \ge 0}$. This implies
\begin{equation}\label{lows3}
\mathbb{E}_{(0, 2y)}\left(\tau_2(y)\right) \ge \frac{C' y}{\beta}
\end{equation}
for some constant $C'$ not depending on $y, \beta$. Using \eqref{lows2} and \eqref{lows3} in \eqref{lows1},
\begin{equation}\label{lows4}
\mathbb{E}_{(0, 4M_0\beta^{-1})}\Big(\int_{0}^{\Xi_0}\mathbf{1}_{[Q_2(s) \ge y]}ds\Big) \ge (\e^{4M_0} - \e^{2M_0})\frac{C' y}{\beta} \e^{-\beta y}.
\end{equation}
From Lemma \ref{regsmall},
\begin{equation}\label{lows5}
\mathbb{E}_{(0,4M_0\beta^{-1})}\left(\Xi_0\right) \le \frac{E_2}{\beta^2}.
\end{equation}
Using \eqref{lows4} and \eqref{lows5} in the representation \eqref{eq:pi} of the stationary distribution, we obtain
$$
\pi(Q_2(\infty) \ge y) \ge \frac{(\e^{4M_0} - \e^{2M_0})C'}{E_2} \ y\beta \e^{-\beta y}, \ \ y \ge 8M_0\beta^{-1}
$$
which proves the lower bound on the stationary probability claimed in the theorem. The bounds on the expectation  follow from the probability bounds upon noting the following:
\begin{align*}
\mathbb{E}_{\pi}(Q_2(\infty)) &\ge \int_{8M_0\beta^{-1}}^{\infty}\pi(Q_2(\infty) \ge y)dy \ge \frac{C_s^{1-}\e^{-8C_s^{2-}M_0}}{C_s^{2-}\beta},\\
\mathbb{E}_{\pi}(Q_2(\infty)) &\le 8M_0\beta^{-1} + \int_{8M_0\beta^{-1}}^{\infty}\pi(Q_2(\infty) \ge y)dy \le \left(8M_0 + \frac{C_s^{1+}}{C_s^{2+}}\right)\frac{1}{\beta}.
\end{align*}
\end{proof}

\begin{proof}[Proof of Theorem~\ref{Q1statsmall}]
From \cite[Theorem 2.1]{BM18} and its proof, we know that there exist constants $R,C,C">0$ and $\beta_0>0$, such that for all $\beta \le \beta_0$ and all $x \ge \frac{R}{\beta}\log \left(\frac{1}{\beta}\right)$, $\pi(Q_1(\infty) \le -x) \le C \e^{-C"x^2}$. Thus, it suffices to prove that $\pi(Q_1(\infty) \le -x) \le C\left(\beta^{1/4}\e^{-(x-2\beta)^2/8} + \beta^2\e^{-C'\frac{x}{\sqrt{\beta}}}\right)$ for all $x \ge 2\beta^{1/4}$.

As mentioned earlier, in the proof, $C, C'$ will denote generic positive constants whose values do not depend on $\beta, x$ and might change from line to line. 
Take any $\beta \le \beta_0$, where $\beta_0$ is minimum over all $\beta_0$'s given in Lemmas \ref{qoneexc-i}-\ref{qoneexc-v}. 
Define the stopping time $\Lambda_{-1} = \inf\{t \ge \tau_2(2M_0\beta^{-1}) : \text{ either } \{Q_2(t)\le \beta^{-1/2}, Q_1(t)=0\} \text{ or } Q_2(t) = 4M_0\beta^{-1}\}$. We can write
\begin{eq}\label{lp1}
&\mathbb{E}_{(0, 4M_0\beta^{-1})}\Big(\int_{0}^{\Xi_0}\mathbf{1}_{[(Q_1(s) \le -x]}ds\Big)\\
& = \mathbb{E}_{(0, 4M_0\beta^{-1})}\Big(\int_{0}^{\Lambda_{-1}}\mathbf{1}_{[(Q_1(s) \le -x]}ds + \mathbf{1}_{[Q_2(\Lambda_{-1}) \le \beta^{-1/2}]}\int_{\Lambda_{-1}}^{\Xi_0}\mathbf{1}_{[(Q_1(s) \le -x]}ds\Big)\\
& \le  \mathbb{E}_{(0, 4M_0\beta^{-1})}\Big(\int_{0}^{\Lambda_{-1}}\mathbf{1}_{[(Q_1(s) \le -x]}ds\Big) + \sup_{0<y \le \beta^{-1/2}}\mathbb{E}_{(0,y)}\Big(\int_0^{\tau_2(4M_0\beta^{-1})}\mathbf{1}_{[(Q_1(s) \le -x]}ds\Big).
\end{eq}
Note that
\begin{eq}\label{lp2}
&\mathbb{E}_{(0, 4M_0\beta^{-1})}\Big(\int_{0}^{\Lambda_{-1}}\mathbf{1}_{[(Q_1(s) \le -x]}ds\Big)\\
& = \mathbb{E}_{(0, 4M_0\beta^{-1})}\Big(\int_{0}^{\tau_2(2M_0\beta^{-1})}\mathbf{1}_{[(Q_1(s) \le -x]}ds\Big)
 + \mathbb{E}_{(0, 4M_0\beta^{-1})}\Big(\int_{\tau_2(2M_0\beta^{-1})}^{\sigma\left(\tau_2(2M_0\beta^{-1})\right)}\mathbf{1}_{[(Q_1(s) \le -x]}ds\Big)\\
&\qquad + \mathbb{E}_{(0, 4M_0\beta^{-1})}\Big(\mathbf{1}_{[Q_2(\sigma\left(\tau_2(2M_0\beta^{-1})\right))>\beta^{-1/2}]}\int_{\sigma\left(\tau_2(2M_0\beta^{-1})\right)}^{\sigma\left(\tau_2(\beta^{-1/2})\right) \wedge \tau_2(4M_0\beta^{-1})}\mathbf{1}_{[(Q_1(s) \le -x]}ds\Big).
\end{eq}
By Lemmas \ref{qoneexc-ii} and  \ref{qoneexc-v}, for any $x \ge 4\beta$,
\begin{multline}\label{lp21}
\mathbb{E}_{(0, 4M_0\beta^{-1})}\Big(\int_{0}^{\tau_2(2M_0\beta^{-1})}\mathbf{1}_{[(Q_1(s) \le -x]}ds\Big)
 + \mathbb{E}_{(0, 4M_0\beta^{-1})}\Big(\int_{\tau_2(2M_0\beta^{-1})}^{\sigma\left(\tau_2(2M_0\beta^{-1})\right)}\mathbf{1}_{[(Q_1(s) \le -x]}ds\Big)\\
 \le C\left(\beta^{-4}\e^{-C'\frac{x}{\beta}} + \e^{-(x-2\beta)^2/4}\right).
\end{multline}
By Lemmas \ref{qoneexc-iii} and~\ref{qoneexc-iv}, for any $x \ge 4\beta^{1/2}$,
\begin{eq}\label{lp22}
&\mathbb{E}_{(0, 4M_0\beta^{-1})}\Big(\mathbf{1}_{[Q_2(\sigma\left(\tau_2(2M_0\beta^{-1})\right)) > \beta^{-1/2}]}\int_{\sigma\left(\tau_2(2M_0\beta^{-1})\right)}^{\sigma\left(\tau_2(\beta^{-1/2})\right) \wedge \tau_2(4M_0\beta^{-1})}\mathbf{1}_{[(Q_1(s) \le -x]}ds\Big)\\
 &\le \sup_{\beta^{-1/2} \le y \le 2M_0\beta^{-1}}\mathbb{E}_{(0, y)}\Big(\int_{0}^{\tau_2(\beta^{-1/2}) \wedge \tau_2(4M_0\beta^{-1})}\mathbf{1}_{[(Q_1(s) \le -x]}ds\Big)\\
 & \qquad+ \sup_{\beta^{-1/2} \le y \le 2M_0\beta^{-1}}\mathbb{E}_{(0, y)}\Big(\Big(\int_{\tau_2(\beta^{-1/2})}^{\sigma\left(\tau_2(\beta^{-1/2})\right)}\mathbf{1}_{[(Q_1(s) \le -x]}ds\Big)\mathbf{1}_{[\tau_2(\beta^{-1/2}) < \tau_2(4M_0\beta^{-1})]}\Big)\\
& \le C\left(\beta^{-7/2}\e^{-C'\frac{x}{\sqrt{\beta}}} + \e^{-(x-2\beta)^2/4}\right).
\end{eq}
Using \eqref{lp21} and \eqref{lp22} in \eqref{lp2}, we obtain for any $x \ge 4\beta^{1/2}$,
\begin{equation}\label{lp3}
\mathbb{E}_{(0, 4M_0\beta^{-1})}\left(\int_{0}^{\Lambda_{-1}}\mathbf{1}_{[(Q_1(s) \le -x]}ds\right) \le C\left(\beta^{-4}\e^{-C'\frac{x}{\sqrt{\beta}}} + \e^{-(x-2\beta)^2/4}\right).
\end{equation}
Now we estimate the second term appearing on the right hand side of \eqref{lp1}. With any starting configuration $(Q_1(0), Q_2(0)) = (0,y)$ where $y \in (0,\beta^{-1/2}]$, define the stopping times $\Lambda_0=0$ and for $k \ge 0$,
\begin{align*}
\Lambda_{2k+1} &= \inf\{t \ge \Lambda_{2k}: Q_2(t) = 2\beta^{-1/2}\},\\
\Lambda_{2k+2} &= \inf\{t \ge \Lambda_{2k+1}: \text{ either } \{Q_2(t)\le \beta^{-1/2}, Q_1(t)=0\} \text{ or } Q_2(t) = 4M_0\beta^{-1}\}.
\end{align*}
Let $\mathcal{N}_{\Lambda} = \inf\{ k \ge 1: Q_2(\Lambda_{2k}) = 4M_0\beta^{-1}\}.$ From Lemma \ref{qoneexc-i}, for any $x \ge 2\beta^{1/4}$,
\begin{multline}\label{lp4}
\sup_{0<y \le \beta^{-1/2}}\mathbb{E}_{(0,y)}\Big(\int_{0}^{\Lambda_1} \mathbf{1}_{[Q_1(s) \le -x]}ds\Big)\\
= \sup_{0 < y \le \beta^{-1/2}}\mathbb{E}_{(0,y)}\Big(\int_{0}^{\tau_2(2\beta^{-1/2})} \mathbf{1}_{[Q_1(s) \le -x]}ds\Big) \le C\beta^{-5/4}\e^{-(x-\beta)^2/8}.
\end{multline}
From Lemmas \ref{qoneexc-iii} and \ref{qoneexc-iv}, for any $x \ge 2\beta^{1/2}$,
\begin{eq}\label{lp5}
&\sup_{0<y \le \beta^{-1/2}}\mathbb{E}_{(0,y)}\Big(\int_{\Lambda_1}^{\Lambda_2} \mathbf{1}_{[Q_1(s) \le -x]}ds\Big)\\
&\le \mathbb{E}_{(0,2\beta^{-1/2})}\Big(\int_{0}^{\tau_2(\beta^{-1/2})\wedge \tau_2(4M_0\beta^{-1})} \mathbf{1}_{[Q_1(s) \le -x]}ds\Big)\\
 &\hspace{3cm}+ \mathbb{E}_{(0, 2\beta^{-1/2})}\Big(\Big(\int_{\tau_2(\beta^{-1/2})}^{\sigma\left(\tau_2(\beta^{-1/2})\right)}\mathbf{1}_{[(Q_1(s) \le -x]}ds\Big)\mathbf{1}_{[\tau_2(\beta^{-1/2}) < \tau_2(4M_0\beta^{-1})]}\Big)\\
& \le C\left(\beta^{-7/2}\e^{-C'\frac{x}{\sqrt{\beta}}} + \e^{-(x-2\beta)^2/4}\right).
\end{eq}
As $Q_2(t) \ge S(t) \ge S(0) + \sqrt{2}W(t) - \beta t$,
\begin{align*}
\prob_{(0,2\beta^{-1/2})}\left(\tau_2(4M_0\beta^{-1}) < \tau_2(\beta^{-1/2})\right)& \ge \prob\left(2\beta^{-1/2} + \sqrt{2}W(t) - \beta t \text{ hits } 4M_0\beta^{-1} \text{ before } \beta^{-1/2}\right)\\
&=\frac{\e^{2\sqrt{\beta}} - \e^{\sqrt{\beta}}}{\e^{4M_0} - \e^{\sqrt{\beta}}} \ge C\sqrt{\beta}.
\end{align*}
This gives us
\begin{equation}\label{lp6}
\sup_{0<y \le \beta^{-1/2}}\mathbb{E}_{(0,y)}\left(\mathcal{N}_{\Lambda}\right) \le C'\beta^{-1/2}.
\end{equation}
From \eqref{lp4}, \eqref{lp5} and \eqref{lp6}, we obtain for any $x \ge 2\beta^{1/4}$,
\begin{eq}\label{lp7}
&\sup_{0<y \le \beta^{-1/2}}\mathbb{E}_{(0,y)}\Big(\int_0^{\tau_2(4M_0\beta^{-1})}\mathbf{1}_{[(Q_1(s) \le -x]}ds\Big)\\
& = \sup_{0<y \le \beta^{-1/2}}\sum_{k=0}^{\infty}\Big(\mathbf{1}_{[\mathcal{N}_{\Lambda} \ge 2k+2]}\int_{\Lambda_{2k}}^{\Lambda_{2k+2}}\mathbf{1}_{[(Q_1(s) \le -x]}ds\Big)\\
&\le \sup_{0<y \le \beta^{-1/2}}\mathbb{E}_{(0,y)}\Big(\int_{0}^{\Lambda_2} \mathbf{1}_{[Q_1(s) \le -x]}ds\Big)\sup_{0<y \le \beta^{-1/2}}\mathbb{E}_{(0,y)}\left(\mathcal{N}_{\Lambda}\right)\\
&\le C\left(\beta^{-7/4}\e^{-(x-\beta)^2/8} + \beta^{-4}\e^{-C'\frac{x}{\sqrt{\beta}}} + \beta^{-1/2}\e^{-(x-2\beta)^2/4}\right).
\end{eq}
Using \eqref{lp3} and \eqref{lp7} in \eqref{lp1}, we obtain for any $x \ge \max\{4\beta^{1/2}, 2\beta^{1/4}\}$,
\begin{equation*}
\mathbb{E}_{(0, 4M_0\beta^{-1})}\Big(\int_{0}^{\Xi_0}\mathbf{1}_{[(Q_1(s) \le -x]}ds\Big) \le C\left(\beta^{-7/4}\e^{-(x-2\beta)^2/8} + \beta^{-4}\e^{-C'\frac{x}{\sqrt{\beta}}}\right).
\end{equation*}
Take $\beta'_0 \in (0, \beta_0]$ small enough such that $\beta^{1/4} \ge 2\beta^{1/2}$ and $\beta^{-4}\e^{-C'\beta^{-1/4}} \le 1$, where $C'$ is the constant appearing in the above equation. Then for every $\beta \le \beta'_0$ and every $x \ge 2\beta^{1/4}$,
\begin{equation*}
\mathbb{E}_{(0, 4M_0\beta^{-1})}\Big(\int_{0}^{\Xi_0}\mathbf{1}_{[(Q_1(s) \le -x]}ds\Big) \le C\left(\beta^{-7/4}\e^{-(x-2\beta)^2/8} + \e^{-C'\frac{x}{2\sqrt{\beta}}}\right).
\end{equation*}
Using the above bound and the lower bound on $\mathbb{E}_{(0, 4M_0\beta^{-1})}\left(\Xi_0\right)$ obtained in Lemma \ref{regsmall}, the theorem is proved.
\end{proof}

\begin{proof}[Proof of Theorem~\ref{betazero}]
For cleaner notation, we will suppress the dependence of the stationary distribution on $\beta$. As before, we will use $\mathbb{E}_{\pi}$ to also denote the expectation operator corresponding to the law of the stationary diffusion process on the path space with initial distribution $\pi$.

For any $n \ge 0$, as in the proof of Proposition~\ref{Q1fin}, applying Ito's formula to $Q_1(t)Q_2^n(t)$ and then taking expectation with respect to $\mathbb{E}_{\pi}$ and applying Fubini's theorem, we obtain for any $\beta>0, t > 0$,
\begin{equation*}
\int_0^t\mathbb{E}_{\pi}\left(-(n+1)Q_1(s)Q_2^n(s) + Q_2^{n+1}(s) -\beta Q_2^n(s)\right)ds
 - \mathbb{E}_{\pi}\int_0^tQ_2^n(s)dL(s) = 0
\end{equation*}
which, using the fact that $\pi$ is stationary, gives for each $t>0$,
\begin{equation}\label{cs1}
-(n+1)\mathbb{E}_{\pi}(Q_1(0)Q_2^n(0)) + \mathbb{E}_{\pi}(Q_2^{n+1}(0)) -\beta \mathbb{E}_{\pi}(Q_2^n(0)) = \frac{\mathbb{E}_{\pi}\int_0^tQ_2^n(s)dL(s)}{t}.
\end{equation}
Using Ito's formula and Fubini's theorem for $Q_2^{n+1}(t)$, we get for each $t>0$,
\begin{equation}\label{cs2}
\frac{\mathbb{E}_{\pi}\int_0^tQ_2^n(s)dL(s)}{t} = \mathbb{E}_{\pi}(Q_2^{n+1}(0)).
\end{equation}
Using \eqref{cs2} in \eqref{cs1},
\begin{equation}\label{cs3}
\mathbb{E}_{\pi}(Q_1(0)Q_2^n(0)) = -\frac{\beta}{n+1}\mathbb{E}_{\pi}(Q_2^n(0)).
\end{equation}
Applying the same procedure to $Q_1^2(t)Q_2^n(t)$, we obtain
\begin{equation}\label{cs4}
-(n+2)\mathbb{E}_{\pi}(Q_1^2(0)Q_2^n(0)) + 2\mathbb{E}_{\pi}(Q_1(0)Q_2^{n+1}(0))) + 2\mathbb{E}_{\pi}(Q_2^n(0)) - 2\beta\mathbb{E}_{\pi}(Q_1(0)Q_2^n(0)) = 0.
\end{equation}
Using \eqref{cs3} to replace $\mathbb{E}_{\pi}(Q_1(0)Q_2^n(0))$ and $\mathbb{E}_{\pi}(Q_1(0)Q_2^{n+1}(0))$ appearing in \eqref{cs4} with the terms $-\frac{\beta}{n+1}\mathbb{E}_{\pi}(Q_2^n(0))$ and $-\frac{\beta}{n+2}\mathbb{E}_{\pi}(Q_2^{n+1}(0))$ respectively,
\begin{equation}\label{cs5}
\mathbb{E}_{\pi}(Q_1^2(0)Q_2^n(0)) = -\frac{2\beta}{(n+2)^2}\mathbb{E}_{\pi}(Q_2^{n+1}(0)) + \left(\frac{2}{n+2} + \frac{2\beta^2}{(n+1)(n+2)}\right)\mathbb{E}_{\pi}(Q_2^n(0)).
\end{equation}
By Theorem \ref{smallbetathm}, there is $\beta^*_s>0$ such that for all $\beta \le \beta^*_s$, the following holds: for each $n \ge 0$, there is $C_n>0$ not depending on $\beta$ satisfying $\mathbb{E}_{\pi}((\beta Q_2(0))^{2n}) \le C_n$. Moreover, by Theorem \ref{Q1statsmall}, $\mathbb{E}_{\pi}(|Q_1(0)|^4) \rightarrow 0$ as $\beta \rightarrow 0$. Using these observations and Cauchy-Schwarz inequality,
\begin{equation}\label{cs6}
\mathbb{E}_{\pi}(Q_1^2(0)(\beta Q_2(0))^n) \le \left(\mathbb{E}_{\pi}(Q_1^4(0))\right)^{1/2}\left(\mathbb{E}_{\pi}((\beta Q_2(0)))^{2n}\right)^{1/2} \rightarrow 0 \ \text{ as } \beta \rightarrow 0.
\end{equation}
We proceed via induction. For $n=0$, using \eqref{cs6} in \eqref{cs5}, we conclude that $\lim_{\beta \rightarrow 0}\mathbb{E}_{\pi}\left(\beta Q_2(0)\right)$ exists and
$$
\lim_{\beta \rightarrow 0}\mathbb{E}_{\pi}\left(\beta Q_2(0)\right) = 2.
$$
Suppose we have proved for some integer $n \ge 0$ that $\lim_{\beta \rightarrow 0}\mathbb{E}_{\pi}\left((\beta Q_2(0))^{n}\right) = (n+1)!$. Using this and \eqref{cs6} in \eqref{cs5}, we conclude that $\lim_{\beta \rightarrow 0}\mathbb{E}_{\pi}\left((\beta Q_2(0))^{n+1}\right)$ exists and
$$\lim_{\beta \rightarrow 0}\mathbb{E}_{\pi}\left((\beta Q_2(0))^{n+1}\right) = (n+2)!
$$
completing the proof of the theorem.
\end{proof}

\appendix
\section{Summary of required known auxiliary results}\label{app:recall}
In this appendix we recall some useful probability estimates from~\cite{BM18}.
\begin{lemma}[{\cite[Lemma 4.3]{BM18}}]\label{lem:q2regeneration}
There exist positive constants $c'_1, c'_2, c'_3, c'_4$ not depending on $\beta$ such that the following hold: 
\begin{itemize}
\item[(i)] For $\beta \ge 1$ and any $y \ge 1$,
\begin{align*}
\prob_{(0,\ y + c'_1\beta)}\big(\inf_{s\leq t}Q_2(s) > c'_1\beta \big)\leq c'_3\exp(-c'_2\beta^{2/5}t^{1/5})
\qquad \mbox{for all } t\geq c'_4 y/\beta.
\end{align*}
\item[(ii)] For $\beta \in (0,1)$ and any $y \ge 1$, for all $t\geq c'_4 \left(\frac{y}{\beta} \vee \beta^{-2}\right)$
\begin{align*}
\prob_{(0,\ y + c'_1\beta^{-1})}\big(\inf_{s\leq t}Q_2(s) > c'_1\beta^{-1}\big)&\leq c'_3\left(\exp(-c'_2\beta^{-2/5}t^{1/5}) + \exp(-c'_2\beta^2 t)\right.\\
&\hspace{2cm}\left. + \beta^{-2}\exp(-c'_2 t)\right).
\end{align*}
\end{itemize}
\end{lemma}

Recall the inter-regeneration times from Section~\ref{sec:reg}.
The next theorem guarantees that for any $B>0$, the process $\{Q_1(t), Q_2(t)\}_{t\geq 0}$ is a classical regenerative process with regeneration times given by $\{\Xi_k\}_{k\geq 0}$.
It also provides a tractable form for the steady state measure.

\begin{theorem}[{\cite[Proposition 3.2 \& Theorem 3.3]{BM18}}]
\label{stationary}
Fix any $B>0$.
Set $(Q_1(0), Q_2(0))= (x,y)$ ($x \le 0, y >0$). There exist constants $ c^{(1)}_\Xi, c^{(2)}_\Xi>0$ (depending on $x,y,B,\beta$), such that 
$$\prob_{(x,y)}(\Xi_0>t)\leq c^{(1)}_\Xi\exp(-c^{(2)}_\Xi t^{1/6}).$$
In particular, $\expt_{(x,y)}\Xi_0<\infty.$
The process described by Equation~\eqref{eq:diffusionjsq} has a unique stationary distribution $\pi$ which can be represented as
$$
\pi(A) = \frac{\mathbb{E}_{(0, 2B)}\left(\int_{0}^{\Xi_0}\mathbf{1}_{[(Q_1(s),Q_2(s)) \in A]}ds\right)}{\mathbb{E}_{(0, 2B)}\left(\Xi_0\right)}
$$
for any measurable set $A \in (-\infty, 0] \times (0, \infty)$. Moreover, the process is ergodic in the sense that for any measurable function $f$ satisfying $\mathbb{E}_{(0, 2B)}\left(\int_{0}^{\Xi_0}f((Q_1(s),Q_2(s)))ds\right) < \infty$,
\begin{equation}\label{ergodicity}
\frac{1}{t}\int_0^t f((Q_1(s),Q_2(s)))ds \rightarrow \frac{\mathbb{E}_{(0, 2B)}\left(\int_{0}^{\Xi_0}f((Q_1(s),Q_2(s))ds\right)}{\mathbb{E}_{(0, 2B)}\left(\Xi_0\right)}
\end{equation}
almost surely as $t \rightarrow \infty$.
\end{theorem}

\begin{lemma}[{\cite[Lemma 5.3]{BM18}}]\label{Q2gebeta2}
There exist constants $R^+, R^-, C^*_1, C^*_2 >0$ that do not depend on $\beta$ such that
\begin{equation*}
\prob_{(0, y+ \beta)}\left(\tau_2\left(2y + \beta\right) \le \tau_2\left(y_0(\beta) + \beta\right)\right) \le C^*_1 \e^{-C^*_2 \beta y}
\end{equation*}
for all $y \ge y_0(\beta)$, where $y_0(\beta) = \frac{R^+}{\beta}$ if $\beta \ge 1$ and $y_0(\beta) = R^-\max\left\lbrace\frac{1}{\beta}\log \frac{1}{\beta}, \frac{1}{\beta}\right\rbrace$ if $\beta <1$.
\end{lemma}
\begin{lemma}[{\cite[Lemma 5.4]{BM18}}]\label{Q2lb}
For any $B>0$,
\begin{equation*}
\prob_{(0,2B)}\left(\tau_2(y) < \tau_2(B)\right) \ge (1-\e^{-\beta B})\e^{-\beta(y-2B)}
\end{equation*}
for all $y \ge 2B$.
\end{lemma}
\begin{lemma}[{\cite[Lemma A.4]{BM18}}]\label{lem:q1integral}
There exist $c'_1, c'_2, c'_3>0$, not depending on $\beta$ such that for any $y > c'_1\left(\beta \vee \beta^{-1}\right)+ \beta$,
\begin{align*}
&\prob_{(0, y)}\Bigg(\int_0^t(-Q_1(s))\dif s>\left(\beta \wedge \beta^{-1}\right)\frac{t}{2},\ \inf_{s\leq t}Q_2(s) \geq c'_1\left(\beta \vee \beta^{-1}\right)+ \beta\Bigg)\\
&\leq \exp\Big(-c'_2t^{1/5}\left(\beta \vee \beta^{-1}\right)^{2/5}\Big)\qquad \text{for}\quad t \ge c'_3 \left(\beta \wedge \beta^{-1}\right)^2.
\end{align*} 
\end{lemma}

\section{Proofs of hitting time estimates in the large-$\beta$ regime}\label{app:large-aux}
In this Appendix, we will assume $\beta$ to be sufficiently large in many calculations, often without explicitly mentioning it.
\begin{proof}[Proof of Lemma~\ref{linfall}]
Set $(Q_1(0), Q_2(0)) = (0, y)$. 
From Lemma~\ref{lem:q2regeneration}, there exist positive constants $c_1',c_2',c_3',c_4'$ not depending on $\beta$ such that for any $\beta \ge 1$ and any $z \ge 1$,
\begin{align*}
\prob_{(0,\ z + c'_1\beta)}\big(\inf_{s\leq t}Q_2(s) > c'_1\beta \big)\leq c'_3\exp(-c'_2\beta^{2/5}t^{1/5})
\end{align*}
for all $t\geq c'_4 z/\beta$. Without loss of generality, we can assume $c_1'>1$. 
This implies that there exists $C>0$ (not depending on $\beta$), such that for $y \ge c_1'\beta+1$,
\begin{equation}\label{hightomed}
\mathbb{E}_{(0,y)}\left(\tau_2(c_1'\beta)\right) \le C\frac{y}{\beta}.
\end{equation}
Thus, the lemma will be proved if we can show $\sup_{x \le 0, y \in [\beta/4,c_1'\beta+1]}\mathbb{E}_{(x,y)}(\tau_2(\beta/4)) \le C$. Note that
\begin{align*}
\sup_{x \le 0, y \in [\beta/4,c_1'\beta+1]}\mathbb{E}_{(x,y)}(\tau_2(\beta/4)) &= \sup_{x \le 0, y \in [\beta/4,c_1'\beta+1]}\mathbb{E}_{(x,y)}(\tau_2(\beta/4), \mathbf{1}_{[\tau_1(0) \le \tau_2(\beta/4)]})\\
& \qquad + \sup_{x \le 0, y \in [\beta/4,c_1'\beta+1]}\mathbb{E}_{(x,y)}(\tau_2(\beta/4), \mathbf{1}_{[\tau_1(0) > \tau_2(\beta/4)]}).
\end{align*}
As $Q_2$ decays exponentially when $Q_1<0$,
$$
\sup_{x \le 0, y \in [\beta/4,c_1'\beta+1]}\mathbb{E}_{(x,y)}(\tau_2(\beta/4), \mathbf{1}_{[\tau_1(0) > \tau_2(\beta/4)]}) \le C.
$$
Further, by the strong Markov property,
$$
\sup_{x \le 0, y \in [\beta/4,c_1'\beta+1]}\mathbb{E}_{(x,y)}(\tau_2(\beta/4), \mathbf{1}_{[\tau_1(0) \le \tau_2(\beta/4)]}) \le \sup_{y \in [\beta/4,c_1'\beta+1]} \mathbb{E}_{(0,y)}(\tau_2(\beta/4)).
$$
Thus, to complete the proof, it suffices to show that
\begin{equation}\label{expectbd}
\sup_{y \in [\beta/4,c_1'\beta+1]} \mathbb{E}_{(0,y)}(\tau_2(\beta/4)) \le C.
\end{equation}
Thus, assume that in the starting configuration $\beta/4\leq y \le c_1'\beta + 1$. Define the following stopping times: $\Theta^*_0=0$ and for $k \ge 0$,
\begin{align*}
\Theta^*_{3k+1} &= \inf\{ t \ge \Theta^*_{3k}: Q_1(t) = -\beta/2 \text{ or } Q_2(t) = c_1'\beta + 2\},\\
\Theta^*_{3k+2} &= \inf\{ t \ge \Theta^*_{3k+1}: Q_2(t) \le  c_1'\beta + 1 \},\\
\Theta^*_{3k+3} &= \inf\{ t \ge \Theta^*_{3k+2}: Q_1(t) =0 \text{ or } Q_2(t) \le \beta/4\}.
\end{align*}
Let $\mathcal{N}_{\Theta} = \inf\{ k \ge 0: Q_2(\Theta^*_{3k}) \le \beta/4\}$. Define $\Theta_{n}=\Theta^*_{n \wedge \mathcal{N}_{\Theta}}$ for $n \ge 1$. Define $q=\e^{-\frac{1}{4(c_1'+2)}}$. Choose $\beta_0 \ge \max\left\lbrace1,\left(\frac{8q}{1-q}\right)^{1/2}\right\rbrace$ satisfying $q(c_1'\beta + 1 + \beta^{-1}) \le c_1'\beta + 1$ for all $\beta \ge \beta_0$. Take any $\beta \ge \beta_0$ and choose $n_0$ not depending on $\beta$ such that $q^{n_0}(c_1'\beta + 1) \le \beta/8$. Define $\Theta_{n}=\Theta_{\mathcal{N}_{\Theta}}$ for all $n \ge \mathcal{N}_{\Theta}$.
We will show that there exists $p > 0$ such that for $\beta \ge \beta_0$,
\begin{equation}\label{contraction}
\inf_{y \in [\beta/4, c_1'\beta + 1]}\prob_{(0,y)}\left(Q_2(\Theta_{3n_0}) \le \beta/4\right) \ge p^{n_0}>0.
\end{equation}
To see this, note that for any $y \in [\beta/4, c_1'\beta + 1]$,
\begin{eq}\label{lin1}
&\prob_{(0,y)}\left(Q_2(\Theta_{3}) \le \max\{q(y + \beta^{-1}), \beta/4\}\right)\\
&\ge \prob_{(0,y)}\left(Q_2(\Theta_{1}) < y + \beta^{-1}, \ \inf\{t \ge \Theta_1: Q_1(t)=0\} - \Theta_1 \ge \frac{1}{4(c_1'+2)}\right)\\
&\ge \prob_{(0,y)}\left(Q_2(\Theta_{1}) < y + \beta^{-1}\right) \inf_{z \le c_1'\beta+2}\prob_{(-\beta/2, z)}\Big(\tau_1(0) \ge \frac{1}{4(c_1'+2)}\Big)
\end{eq}
where for the first inequality, note that if the process starts from a state with $Q_2<y+\beta^{-1}$ and $Q_1$ stays away from 0 for more than $1/(4(c_1'+2))$ time, then $Q_2$ must hit $q(y + \beta^{-1})$ (due to exponential decay of $Q_2$),
and the last step follows from the strong Markov property of the process applied at time $\Theta_1$. 
Recall that $S(t) < y + \beta^{-1}$ for all $t \le \Theta_1$ implies $Q_2(t) < y + \beta^{-1}$ for all $t \le \Theta_1$ (this is because if $t^*$ denotes the first time $Q_2$ hits the level $y + \beta^{-1}$ from below, $Q_1(t^*)=0$ and consequently, $S(t^*)=Q_2(t^*)=y + \beta^{-1}$). Further, for $t \le \Theta_1$, 
$$
S(t) = S(0) + \sqrt{2}W(t) - \beta t + \int_0^t\left(-Q_1(s)\right)ds \le y + \sqrt{2}W(t) - \frac{\beta}{2} t. 
$$
Therefore,
\begin{align}\label{lin2}
 \prob_{(0,y)}\left(Q_2(\Theta_{1}) < y + \beta^{-1}\right) \ge \prob\left(\sup_{t < \infty}\left(\sqrt{2}W(t) - \frac{\beta}{2} t\right) < \beta^{-1}\right) = 1-\e^{-1/2}.
\end{align}
Now we estimate $\inf_{z \le c_1'\beta+2}\prob_{(-\beta/2, z)}\left(\tau_1(0) \ge \frac{1}{4(c_1'+2)}\right)$. Note that for $t \le \tau_1(0)$, $Q_2$ is decreasing and hence if $(Q_1(0), Q_2(0)) = (-\beta/2, z)$ with $z \le c_1'\beta +2$, then for $t \le \tau_1(0) \wedge \tau_1(-\beta)$,
$$
Q_1(t) \le -\frac{\beta}{2} + \sqrt{2}W(t) + \left(c_1'\beta + 2\right)t \le -\frac{\beta}{2} + \sqrt{2}W(t) + \left(c_1' + 2\right)\beta t, \ \ \text{ as } \beta \ge 1.
$$
Moreover, for $t \le \tau_1(0)$, 
$$
Q_1(t) \ge -\frac{\beta}{2} +  \sqrt{2}W(t) -\beta t.
$$
Thus,
\begin{eq}\label{lin3}
&\sup_{z \le c_1'\beta+2}\prob_{(-\beta/2, z)}\left(\tau_1(0) \le \frac{1}{4(c_1'+2)}\right) \le \sup_{z \le c_1'\beta+2}\prob_{(-\beta/2, z)}\left(\tau_1(0) \wedge \tau_1(-\beta) \le \frac{1}{4(c_1'+2)}\right)\\
&\le \prob\left(-\frac{\beta}{2} + \sqrt{2}W(t) + \left(c_1' + 2\right)\beta t \text{ hits zero before time } \frac{1}{4(c_1'+2)}\right)\\
 &\hspace{5cm}+ \prob\left(-\frac{\beta}{2} + \sqrt{2}W(t) - \beta t \text{ hits } -\beta \text{ before time } \frac{1}{4(c_1'+2)}\right)\\
& = \prob\Big(\sup_{t \le \frac{1}{4(c_1'+2)}}\left( \sqrt{2}W(t) + \left(c_1' + 2\right)\beta t \right) \ge \beta/2\Big) + \prob\Big(\inf_{t \le \frac{1}{4(c_1'+2)}}\left( \sqrt{2}W(t) -\beta t \right) \le -\beta/2\Big)\\
& \le \prob\Big(\sup_{t \le \frac{1}{4(c_1'+2)}}\left( \sqrt{2}W(t)\right) \ge \beta/4\Big) + \prob\Big(\inf_{t \le \frac{1}{4(c_1'+2)}}\left( \sqrt{2}W(t)\right) \le -\beta/4\Big) \le \e^{-C\beta^2},
\end{eq}
where $C$ does not depend on $\beta, y$. Using \eqref{lin2} and \eqref{lin3} in \eqref{lin1}, we obtain $p>0$ such that for all $\beta \ge \beta_0$,
\begin{align}\label{lin4}
\inf_{y \in [\beta/4, c_1'\beta + 1]}\prob_{(0,y)}\left(Q_2(\Theta_{3}) \le \max\{q(y + \beta^{-1}), \beta/4\}\right) \ge p>0.
\end{align}
If $q(y + \beta^{-1}) \le \beta/4$, this proves \eqref{contraction}. If $q(y + \beta^{-1}) > \beta/4$, we obtain
\begin{eq}\label{eq:l41-d1}
&\inf_{y \in [\beta/4, c_1'\beta + 1]}\prob_{(0,y)}\left(Q_2(\Theta_{6}) \le \max\{q^2(y + \beta^{-1}) + q\beta^{-1}, \beta/4\}\right)\\
& \ge \inf_{y \in [\beta/4, c_1'\beta + 1]}\left(\prob_{(0,y)}\left(Q_2(\Theta_{6}) \le \max\{q^2(y + \beta^{-1}) + q\beta^{-1}, \beta/4\}, \frac{\beta}{4} <Q_2(\Theta_{3}) \le q(y + \beta^{-1})\right)\right.\\
 &\qquad\qquad \left. + \prob_{(0,y)}\left(Q_2(\Theta_{6}) \le \max\{q^2(y + \beta^{-1}) + q\beta^{-1}, \beta/4\}, Q_2(\Theta_{3}) \le \frac{\beta}{4}\right)\right)
\end{eq}
Applying the strong Markov property at $\Theta_3$ to the first probability on the right side of \eqref{eq:l41-d1} and noting that if $Q_2(\Theta_{3})>\beta/4$, then $Q_1(\Theta_3)=0$, we obtain
\begin{multline*}
\prob_{(0,y)}\left(Q_2(\Theta_{6}) \le \max\{q^2(y + \beta^{-1}) + q\beta^{-1}, \beta/4\}, \frac{\beta}{4} <Q_2(\Theta_{3}) \le q(y + \beta^{-1})\right)\\
\ge \prob_{(0,y)}\left(\frac{\beta}{4} < Q_2(\Theta_{3}) \le q(y + \beta^{-1})\right)p,
\end{multline*}
where we have used \eqref{lin4} to the process started at $\Theta_3$ along with the fact that $q(y + \beta^{-1}) \le c_1'\beta + 1$ for $\beta \ge \beta_0$. Moreover, for the second probability on the right side of \eqref{eq:l41-d1},
\begin{align*}
\prob_{(0,y)}\left(Q_2(\Theta_{6}) \le \max\{q^2(y + \beta^{-1}) + q\beta^{-1}, \beta/4\}, Q_2(\Theta_{3}) \le \frac{\beta}{4}\right)\\
= \prob_{(0,y)}\left(Q_2(\Theta_{3}) \le \frac{\beta}{4}\right) \ge \prob_{(0,y)}\left(Q_2(\Theta_{3}) \le \frac{\beta}{4}\right)p.
\end{align*}
From the above and \eqref{lin4}, we get
\begin{multline*}
\inf_{y \in [\beta/4, c_1'\beta + 1]}\prob_{(0,y)}\left(Q_2(\Theta_{6}) \le \max\{q^2(y + \beta^{-1}) + q\beta^{-1}, \beta/4\}\right)\\
\ge \left(\inf_{y \in [\beta/4, c_1'\beta + 1]}\prob_{(0,y)}\left(Q_2(\Theta_{3}) \le q(y + \beta^{-1})\right)\right)p \ge p^2.
\end{multline*}
Writing $f(y) = q(y + \beta^{-1})$, note that for any $\beta \ge \max\left\lbrace1,\left(\frac{8q}{1-q}\right)^{1/2}\right\rbrace$ and any $y \le c_1'\beta + 1$, applying $f$ $n_0$ times to $y$ gives us a number less than or equal to $\beta/4$. Thus, iterating \eqref{lin4} $n_0$ times, we obtain \eqref{contraction}. This, in turn, implies for $k \ge 1$,
\begin{equation}\label{Ntheta}
\sup_{y \in [\beta/4,c_1'\beta + 1]}\prob_{(0,y)}\left(\mathcal{N}_{\Theta} \ge kn_0\right) \le (1-p^{n_0})^k.
\end{equation}
Next, we want to prove that there exists a positive constant $C>0$ that does not depend on $\beta \ge 1$ such that
\begin{equation}\label{finexp}
\sup_{y \in [\beta/4, c_1'\beta + 1]}\mathbb{E}_{(0,y)}(\Theta_3) \le C.
\end{equation}
To see this, observe that for $t \le \Theta_1$, $Q_1(t) \le S(t) \le c_1'\beta + 1 + \sqrt{2}W(t) - \frac{\beta}{2} t$ and hence, $\Theta_1$ is stochastically dominated by the hitting time of level $-\beta/2$ by $c_1'\beta + 1 + \sqrt{2}W(t) - \frac{\beta}{2} t$ and hence, $\sup_{y \in [\beta/4, c_1'\beta + 1]}\mathbb{E}_{(0,y)}(\Theta_1) \le C(2c_1' + 3)$. From Lemma \ref{lem:q2regeneration}, we get $\sup_{y \in [\beta/4, c_1'\beta + 1]}\mathbb{E}_{(0,y)}(\Theta_2 - \Theta_1) \le C/\beta$. Further, as $Q_2$ decreases exponentially in $[\Theta_2, \Theta_3]$, $\Theta_3 - \Theta_2 \le \log (4(c_1'+1))$. These observations lead to \eqref{finexp}.
Now, we can write
\begin{align*}
\sup_{y \in [\beta/4,c_1'\beta + 1]} \mathbb{E}_{(0,y)}(\tau_2(\beta/4)) &\le \sup_{y \in [\beta/4,c_1'\beta + 1]} \mathbb{E}_{(0,y)}\Big(\sum_{k=1}^{\infty}\mathbf{1}_{[\mathcal{N}_{\Theta} = k]}\Theta_{3k}\Big)\\
&= \sup_{y \in [\beta/4,c_1'\beta + 1]} \mathbb{E}_{(0,y)}\Big(\sum_{k=1}^{\infty}\left(\Theta_{3k}-\Theta_{3k-3}\right)\mathbf{1}_{[\mathcal{N}_{\Theta} > k-1]}\Big)\\
&\le \sum_{k=1}^{\infty}\sup_{y \in [\beta/4, c_1'\beta + 1]}\mathbb{E}_{(0,y)}(\Theta_3)\sup_{y \in [\beta/4,c_1'\beta + 1]}\prob_{(0,y)}\left(\mathcal{N}_{\Theta} > k-1\right),
\end{align*}
where the second equality follows by interchanging summation and the third inequality follows by applying the strong Markov property at $\Theta_{3k-3}$.
Finally, combining \eqref{Ntheta} and \eqref{finexp}, we obtain
\begin{equation*}
\sup_{y \in [\beta/4,c_1'\beta + 1]} \mathbb{E}_{(0,y)}(\tau_2(\beta/4)) \le \sum_{k=0}^{\infty}Cn_0(1-p^{n_0})^k \le C'
\end{equation*}
which, in particular, proves \eqref{expectbd} and hence completes the proof of the lemma.
\end{proof}

\begin{proof}[Proof of Lemma~\ref{OUhit}]
First we will prove the upper bound in \eqref{expexc}. Start from $(Q_1(0), Q_2(0))=(-\beta,y)$ for $0 < y \le \beta$. Define the stopping times: $\tau^+_0=0$ and for $k \ge 0$,
\begin{align*}
\tau^+_{2k+1} &= \inf\{t \ge \tau^+_{2k}: Q_1(t) = 0 \text{ or } Q_1(t)=-\beta - 1\},\\
\tau^+_{2k+2} &= \inf\{t \ge \tau^+_{2k+1}: Q_1(t) = 0 \text{ or } Q_1(t)=-\beta\}.
\end{align*}
Let $\mathcal{N}^+ = \inf\{ k \ge 0: Q_1(\tau^+_{2k+1}) = 0\}.$ Note that for $t \le \tau_1(0)$, $Q_1(t) \ge Q_1(0) + \sqrt{2}W(t) - \beta t$. Therefore, for any $y>0$,
\begin{multline*}
\prob_{(-\beta,y)}\left(\tau_1(0) < \tau_1(-\beta-1)\right) \ge \prob\left(-\beta + \sqrt{2}W(t) - \beta t \text{ hits } 0 \text{ before } -\beta -1\right)\\
 = \frac{1-\e^{-\beta}}{\e^{\beta^2} - \e^{-\beta}} \ge \frac{1}{2}\e^{-\beta^2}
\end{multline*}
for sufficiently large $\beta$.
Therefore, for any $k \ge 0$,
\begin{equation}\label{Nplus}
\sup_{y \in (0,\beta]}\prob_{(-\beta,y)}\left(\mathcal{N}^+ \ge k\right) \le \left(1-\frac{1}{2}\e^{-\beta^2}\right)^k.
\end{equation}
for all $\beta \ge \beta_0$ where $\beta_0$ is chosen sufficiently large. If the starting configuration is set to $(Q_1(0),Q_2(0)) = (-\beta-1, z)$ for any $z \in (0, \beta]$, then for $t \le \tau_1(0)$, $Q^*_1 = Q_1 + \beta$ is stochastically bounded from below by an Ornstein-Uhlenbeck process (that does not depend on $\beta$) started from $-1$. To see this, note that for $t \le \tau_1(0)$, $Q_2(t) = z\e^{-t}$. Therefore, $Q^*_1$ has the following representation for $t \le \tau_1(0)$:
$$
Q^*_1(t) = -1 + \sqrt{2}W(t) + \int_0^t \left(-Q^*_1(s) + z\e^{-s}\right)ds.
$$
By Proposition 2.18 of \cite{Karatzas}, $Q^*_1(t) \ge Z(t)$ for all $t \le \tau_1(0)$, where $Z$ is the Ornstein-Uhlenbeck process that is the solution to the following SDE:
$$
Z(t) = -1 + \sqrt{2}W(t) -\int_0^tZ(s)ds,
$$
where $W$ is the same Brownian motion that drives $Q^*_1$. Thus, $\tau_1(-\beta)$ is stochastically bounded above by the hitting time of $0$ by $Z$. Therefore,
\begin{equation}\label{ub1}
\sup_{z \in (0, \beta]}\mathbb{E}_{(-\beta-1,z)}\left(\tau_1(-\beta)\right) \le C.
\end{equation}
Now consider the starting configuration $(Q_1(0), Q_2(0)) = (-\beta, z)$ for $z \in (0,\beta]$. Note that on the event $\{\tau_1(0) \ge \log (1/\beta)\}$, $Q_2(t) \le 1$ for $t \in [\log (1/\beta), \tau_1(0)]$. Therefore, again using Proposition 2.18 of \cite{Karatzas}, for any $t \ge \log (1/\beta)$, $Q^*_1(t) \le \widetilde{Z}(t)$ where $\widetilde{Z}$ is the solution to the SDE:
$$
\widetilde{Z}(t) = \sqrt{2}W(t) + \int_0^t(1-\widetilde{Z}(s))ds,
$$
where $W$ is the same Brownian motion that drives $Q^*_1$. Thus, writing $\widetilde{\tau}$ for the hitting time of $\widetilde{Z}$ on level $-1$, the event $\{\tau_1(0) \wedge \tau_1(-\beta-1) > t\}$ implies $\{\widetilde{\tau} > t - \log (1/\beta)\}$ for any $t \ge \log (1/\beta)$. Hence,
\begin{multline}\label{ub2}
\sup_{z \in (0,\beta]}\mathbb{E}_{(-\beta,z)}\left(\tau_1(0) \wedge \tau_1(-\beta-1)\right) \le \log (1/\beta) + \sup_{z \in (0,\beta]}\int_{\log (1/\beta)}^{\infty}\prob_{(-\beta, z)}\left(\tau_1(0) \wedge \tau_1(-\beta-1) > t\right)\\
\le \log (1/\beta) + \int_{\log (1/\beta)}^{\infty}\prob\left(\widetilde{\tau} > t - \log (1/\beta)\right) = \log (1/\beta) + \mathbb{E}\left(\widetilde{\tau}\right) \le \log(1/\beta) + C.
\end{multline}
Define the stopping time $\tau^+ = \inf\{t \ge \tau_1(-\beta): Q_1(t) = -\beta-1 \text{ or } Q_1(t) = 0\}$. Combining \eqref{ub1} and \eqref{ub2}, we get
\begin{equation}\label{ub3}
\sup_{z \in (0, \beta]}\mathbb{E}_{(-\beta-1,z)}\left(\tau^+\right) \le \log(1/\beta) + C.
\end{equation}
Finally, combining \eqref{Nplus}, \eqref{ub2} and \eqref{ub3} and using strong Markov property, we have
\begin{align*}
&\sup_{y \in (0,\beta]} \mathbb{E}_{(-\beta,y)}(\tau_1(0)) = \sup_{y \in (0,\beta]} \mathbb{E}_{(-\beta,y)}\Big(\sum_{k=0}^{\infty}\mathbf{1}_{[\mathcal{N}^+ = k]}\tau^+_{2k+1}\Big)\\
&\le \sup_{y \in (0,\beta]}\mathbb{E}_{(-\beta,y)}\left(\tau_1(0) \wedge \tau_1(-\beta-1)\right) + \sup_{y \in (0,\beta]} \mathbb{E}_{(-\beta,y)}\Big(\sum_{k=1}^{\infty}\left(\tau^+_{2k+1}-\tau^+_{2k-1}\right)\mathbf{1}_{[\mathcal{N}^+ \ge k]}\Big)\\
&\le \sup_{y \in (0,\beta]}\mathbb{E}_{(-\beta,y)}\left(\tau_1(0) \wedge \tau_1(-\beta-1)\right) + \sum_{k=1}^{\infty}\Big(\sup_{y \in (0, \beta]}\mathbb{E}_{(-\beta-1,y)}\left(\tau^+\right)\Big)\sup_{y \in (0,\beta]}\prob_{(-\beta,y)}\left(\mathcal{N}^+ \ge k\right)\\
&\le \left(\log(1/\beta) + C\right)\sum_{k=0}^{\infty}\sup_{y \in (0,\beta]}\prob_{(-\beta,y)}\left(\mathcal{N}^+ \ge k\right)\\
&\le \left(\log(1/\beta) + C\right)\sum_{k=0}^{\infty}\Big(1-\frac{1}{2}\e^{-\beta^2}\Big)^{k} 
= 2\left(\log(1/\beta) + C\right)\e^{\beta^2},
\end{align*}
which gives the required upper bound for $\beta \ge \beta_0$ for sufficiently large $\beta_0$.\\

Now, we prove the lower bound in \eqref{expexc}. Start from $(Q_1(0), Q_2(0))=(-\beta/4, y)$ where $y \le \beta/2$.
Define the stopping times: $\tau^-_0=0$ and for $k \ge 0$,
\begin{align*}
\tau^-_{2k+1} &= \inf\{t \ge \tau^-_{2k}: Q_1(t) = 0 \text{ or } Q_1(t)=-3\beta/8\},\\
\tau^-_{2k+2} &= \inf\{t \ge \tau^-_{2k+1}: Q_1(t) = 0 \text{ or } Q_1(t)=-\beta/4\}.
\end{align*}
Let $\mathcal{N}^- = \inf\{ k \ge 0: Q_1(\tau^-_{2k+1}) = 0\}.$ Note that for $t \in [\tau^-_{2k}, \tau^-_{2k+1}]$ for any $k \ge 0$, 
$$
Q_1(t) =Q_1(0) + \sqrt{2}W(t) -\beta t + \int_0^t(-Q_1(s) + Q_2(s))ds \le Q_1(0) + \sqrt{2}W(t) - \beta t/8.
$$
Therefore, for any $y>0$,
\begin{multline*}
\prob_{(-\beta/4,y)}\left(\tau_1(0) < \tau_1(-3\beta/8)\right) \le \prob\left(-\beta/4 + \sqrt{2}W(t) - \beta t/8 \text{ hits } 0 \text{ before } -3\beta/8\right)\\
 = \frac{1-\e^{-\beta^2/64}}{\e^{\beta^2/32} - \e^{-\beta^2/64}} \le \e^{-\beta^2/32}.
\end{multline*}
Therefore, for any $k \ge 0$, by the union bound,
\begin{equation}\label{Nminus}
\sup_{y \in (0,\beta/2]}\prob_{(-\beta/4,y)}\left(\mathcal{N}^- \le k\right) \le (k+1)\e^{-\beta^2/32}.
\end{equation}
Next, we show that for $\beta \ge 1$,
\begin{equation}\label{explow}
\inf_{y \in (0,\beta/2]}\mathbb{E}_{(-\beta/4,y)}\left(\tau^-_{1}\right) \ge \mu,
\end{equation}
where $\mu$ does not depend on $\beta$. To see this, note that for $t \in [0, \tau^-_{1}]$,
$$
Q_1(0) + \sqrt{2}W(t) - \beta t \le Q_1(t) \le Q_1(0) + \sqrt{2}W(t) - \beta t/8.
$$
Thus, for any $\beta \ge 1$,
\begin{multline*}
\inf_{y \in (0,\beta/2]}\prob_{(-\beta/4,y)}\left(\tau^-_{1}\ge 1/16\right)  \ge \prob\left(\sup_{t \le 1/16}\sqrt{2}W(t) \le \beta/8, \inf_{t \le 1/16}\sqrt{2}W(t) > -\beta/16\right)\\
\ge \prob\left(\sup_{t \le 1/16}\sqrt{2}W(t) \le 1/8, \inf_{t \le 1/16}\sqrt{2}W(t) > -1/16\right) = p^-> 0,
\end{multline*}
where $p^-$ does not depend on $\beta$. Therefore, for every $\beta \ge 1$,
$$
\inf_{y \in (0,\beta/2]}\mathbb{E}_{(-\beta/4,y)}\left(\tau^-_{1}\right) \ge \int_0^{1/16}\inf_{y \in (0,\beta/2]}\prob_{(-\beta/4,y)}\left(\tau^-_{1}\ge t\right)dt \ge p^-/16>0.
$$
We then have the following:
\begin{eq}\label{onetwo}
&\inf_{y \in (0,\beta/2]}\mathbb{E}_{(-\beta/4,y)}\left(\tau_1(0)\right) = \inf_{y \in (0,\beta/2]}\mathbb{E}_{(-\beta/4,y)}\sum_{k=0}^{\mathcal{N}^-}\left(\tau^-_{2k+1} - \tau^-_{2k}\right)\\
&= \inf_{y \in (0,\beta/2]}\mathbb{E}_{(-\beta/4,y)}\sum_{k=0}^{\infty}\left(\tau^-_{2k+1} - \tau^-_{2k}\right)\mathbf{1}_{[\mathcal{N}^- \ge k]}\\
&\ge \sum_{k=0}^{\infty}\Big(\inf_{y \in (0,\beta/2]}\mathbb{E}_{(-\beta/4,y)}\left(\tau^-_{1}\right)\Big)\Big(\inf_{y \in (0,\beta/2]}\prob_{(-\beta/4,y)}\left(\mathcal{N}^- \ge k\right)\Big) \\
&\ge \mu\sum_{k=0}^{\lfloor \e^{\beta^2/64}\rfloor + 1}\inf_{y \in (0,\beta/2]}\prob_{(-\beta/4,y)}\left(\mathcal{N}^- \ge k\right),
\end{eq} 
where the first inequality follows using the strong Markov property.
Using \eqref{Nminus}, for all $k \le \lfloor \e^{\beta^2/64}\rfloor + 1$,
$$
\inf_{y \in (0,\beta/2]}\prob_{(-\beta/4,y)}\left(\mathcal{N}^- \ge k\right) \ge \frac{1}{2}
$$
for all $\beta \ge \beta_0$ for sufficiently large $\beta_0$. This fact, along with \eqref{onetwo}, implies that for all $\beta \ge \beta_0$,
$$
\inf_{y \in (0,\beta/2]}\mathbb{E}_{(-\beta/4,y)}\left(\tau_1(0)\right) \ge \frac{\mu}{2}\e^{\beta^2/64},
$$
which proves the lower bound in \eqref{expexc}.

To prove \eqref{problb}, observe that for each $k\ge 0$, $\tau^-_{2k+1} - \tau^-_{2k}$ is stochastically dominated by the hitting time of level $-\beta/8$ by $\sqrt{2}W(t) - \beta t/8$. We have for any $\beta \ge 1$ and any $t \ge 1$,
\begin{multline}\label{exptailconc2}
\sup_{k\ge 0}\sup_{y \in (0,\beta/2]}\prob_{(-\beta/4,y)}\left(\tau^-_{2k+1} - \tau^-_{2k} \ge t\right) \le \prob\left(\sqrt{2}W(t) - \beta t/8 \text{ hits } -\beta/8 \text{ after time } t\right)\\
\le \prob\left(\sqrt{2}W(t)> \frac{\beta}{8}(t-1)\right) \le  \prob\left(\sqrt{2}W(t)> \frac{1}{8}(t-1)\right) \le C\e^{-C't}.
\end{multline}
By \eqref{exptailconc2}, \eqref{explow} and using Chernoff's inequality (see \cite[Pg.~16, Equation (2.2)]{Massart07}),
\begin{equation*}
\sup_{y \in (0,\beta/2]}\prob_{(-\beta/4,y)}\Big(\sum_{k=0}^n\left(\tau^-_{2k+1} - \tau^-_{2k}\right) \le n\mu/2, \ \mathcal{N}^- > n\Big) \le C\e^{-C'n},
\end{equation*}
where $\mu$ is the constant (independent of $\beta$) that appears in \eqref{explow}. Therefore, recalling \eqref{Nminus}, we obtain
\begin{multline*}
\sup_{y \in (0,\beta/2]}\prob_{(-\beta/4,y)}\left(\tau_1(0) \le \frac{\mu}{2} \e^{\beta^2/64}\right) \le \sup_{y \in (0,\beta/2]}\prob_{(-\beta/4,y)}\left(\mathcal{N}^- \le \e^{\beta^2/64}\right)\\
+ \sup_{y \in (0,\beta/2]}\prob_{(-\beta/4,y)}\Big(\sum_{k=0}^{\lfloor \e^{\beta^2/64}\rfloor + 1}\left(\tau^-_{2k+1} - \tau^-_{2k}\right) \le \frac{\mu}{2} \e^{\beta^2/64}, \ \mathcal{N}^- > \e^{\beta^2/64}\Big) \le C\e^{-\beta^2/64} + C\e^{-C'\e^{\beta^2/64}},
\end{multline*}
which proves \eqref{problb}.
\end{proof}

\begin{proof}[Proof of Lemma~\ref{middletosmall}]
Fix $y \in [\beta \e^{-\mathcal{C}_1^- \e^{\mathcal{C}_2^-\beta^2}}, \beta/4]$. Set the starting configuration $(Q_1(0), Q_2(0)) = (0,z)$ where $z \in [y,\beta/4]$.

Recall from Lemma \ref{lem:q2regeneration} that there exist positive constants $c_1',c_2',c_3',c_4'$ not depending on $\beta$ such that for any $\beta \ge 1$ and any $w \ge 1$,
\begin{align*}
\prob_{(0,\ w + c'_1\beta)}\big(\inf_{s\leq t}Q_2(s) > c'_1\beta \big)\leq c'_3\exp(-c'_2\beta^{2/5}t^{1/5})
\end{align*}
for all $t\geq c'_4 w/\beta$. Without loss of generality, we can assume $c_1'>1$. Define the following stopping times. $\Gamma_0 = 0$ and for $k \ge 0$,
\begin{align*}
\Gamma_{5k+1} &= \inf\{ t \ge \Gamma_{5k}: Q_1(t) = -\beta/4 \text{ or } Q_2(t) = c_1'\beta + 2\},\\
\Gamma_{5k+2} &= \inf\{ t \ge \Gamma_{5k+1}: Q_2(t) \le  c_1'\beta + 1 \},\\
\Gamma_{5k+3} &= \inf\{ t \ge \Gamma_{5k+2}: Q_1(t) =0 \text{ or } Q_2(t) = y\},\\
\Gamma_{5k+4} &= \inf\{ t \ge \Gamma_{5k+3}: Q_2(t) \le \beta/4\text{ or } Q_2(t)=y\},\\
\Gamma_{5k+5} &= \inf\{ t \ge \Gamma_{5k+4}: Q_1(t) =0 \text{ or } Q_2(t) = y\}.
\end{align*}
Define $\mathcal{N}_{\Gamma} = \inf\{ k \ge 1: Q_2(\Gamma_{5k})=y\}$.

We will first show that there exists $p_0>0$ that do not depend on $\beta$ such that for all $\beta \ge \beta_0$ for large enough $\beta_0$,
\begin{equation}\label{downprob}
\inf_{w \in [y, \beta/4]}\prob_{(0,w)}\left(Q_2(\Gamma_3) = y\right) \ge p_0>0.
\end{equation}
To see this, first observe that if $Q_1<0$, then $Q_2$ decreases exponentially. Thus, applying the strong Markov property at $\Gamma_1$ and recalling that $y \ge \beta \e^{-\mathcal{C}_1^- \e^{\mathcal{C}_2^-\beta^2}}$, we get 
\begin{multline}\label{exp1}
\inf_{w \in [y, \beta/4]}\prob_{(0,w)}\left(Q_2(\Gamma_3) = y\right) \ge \inf_{w \in [y, \beta/4]}\prob_{(0,w)}\left(Q_2(\Gamma_1) \le \beta/2\right)\\
\times \inf_{w \in (0, \beta/2]}\prob_{(-\beta/4,w)}\left(\tau_1(0) \ge \mathcal{C}_1^- \e^{\mathcal{C}_2^-\beta^2}\right).
\end{multline}
Now, $S(t) \le \beta/2$ for $t \le \Gamma_1$ implies $Q_2(t) \le \beta/2$ for $t \le \Gamma_1$ (by the same argument appearing after \eqref{lin1}) and for  $t \le \Gamma_1$,
$$
S(t) = S(0) + \sqrt{2}W(t) - \beta t + \int_0^t\left(-Q_1(s)\right)ds \le \frac{\beta}{4} + \sqrt{2}W(t) - \frac{3\beta}{4} t. 
$$
Therefore, for any $w \in [\beta \e^{-\mathcal{C}_1^- \e^{\mathcal{C}_2^-\beta^2}}, \beta/4]$,
\begin{align}\label{exp2}
 \prob_{(0,w)}\left(Q_2(\Gamma_1) \le \beta/2\right) \ge \prob\left(\sup_{t < \infty}\left(\sqrt{2}W(t) - \frac{3\beta}{4} t\right) < \beta/4\right) = 1-\e^{-3\beta^2/16}.
\end{align}
By Lemma \ref{OUhit},
\begin{align}\label{exp3}
\inf_{w \in (0, \beta/2]}\prob_{(-\beta/4,w)}\left(\tau_1(0) \ge \mathcal{C}_1^- \e^{\mathcal{C}_2^-\beta^2}\right) \ge 1- \mathcal{D}_1 \e^{-\mathcal{D}_2\beta^2}.
\end{align}
Using \eqref{exp2} and \eqref{exp3} in \eqref{exp1} gives \eqref{downprob}. This, in turn, implies for $k \ge 1$,
\begin{equation}\label{NGammaest}
\sup_{w \in [y, \beta/4]}\prob_{(0,w)}\left(\mathcal{N}_{\Gamma} \ge k\right) \le (1-p_0)^k.
\end{equation}
Next, we will show that,
\begin{equation}\label{finexpdown}
\sup_{z \in [y, \beta/4]}\mathbb{E}_{(0,z)}\left(\Gamma_{5}\right) \le C\log\left(\frac{\beta}{y}\right).
\end{equation}
To verify this, observe that for $t \le \Gamma_{1}$, $Q_1(t) \le S(t) \le \frac{\beta}{4} + \sqrt{2}W(t) - \frac{3\beta}{4} t$ and hence, $\Gamma_{1}$ is stochastically dominated by the hitting time of level $-\beta/4$ by $\frac{\beta}{4} + \sqrt{2}W(t) - \frac{3\beta}{4} t$ and hence, $\sup_{z \in [y, \beta/4]}\mathbb{E}_{(0, z)}\left(\Gamma_{1}\right) \le C$. From Lemma \ref{lem:q2regeneration}, $\sup_{z \in [y, \beta/4]}\mathbb{E}_{(0, z)}\left(\Gamma_{2} - \Gamma_{1}\right) \le C/\beta$. Further, as $Q_2$ decreases exponentially in $[\Gamma_{2}, \Gamma_{3}]$, $\Gamma_{3} -\Gamma_{2} \le \log\left(\frac{c_1'\beta}{y}\right)$. Moreover, by the strong Markov property and Lemma \ref{linfall},
\begin{multline*}
\sup_{z \in [y, \beta/4]}\mathbb{E}_{(0,z)}\left(\Gamma_{4} - \Gamma_{3}\right) = \sup_{z \in [y, \beta/4]}\mathbb{E}_{(0,z)}\left((\Gamma_{4} - \Gamma_{3})\mathbf{1}_{[Q_1(\Gamma_{3})=0]}\right)\\
\le \sup_{\beta/4 \le w \le c_1'\beta} \mathbb{E}_{(0,w)}\left(\tau_2(\beta/4)\right) \le C\frac{c_1'\beta}{\beta} = Cc_1'.
\end{multline*}
Finally, as $Q_2$ decreases exponentially on $[\Gamma_4, \Gamma_5]$, $\Gamma_5 - \Gamma_4 \le \log\left(\frac{\beta}{4y}\right)$.
These observations yield \eqref{finexpdown}.

Finally, using \eqref{NGammaest} and \eqref{finexpdown}, we obtain,
\begin{align*}
&\sup_{z \in [y, \beta/4]}\mathbb{E}_{(0,z)}(\tau_2(y)) \le \sup_{z \in [y, \beta/4]}\mathbb{E}_{(0,z)}\Big(\sum_{k=1}^{\infty}\mathbf{1}_{[\mathcal{N}_{\Gamma} = k]}\Gamma_{5k}\Big)\\
&= \sup_{z \in [y, \beta/4]}\mathbb{E}_{(0,z)}\Big(\sum_{k=1}^{\infty}\left(\Gamma_{5k}-\Gamma_{5k-5}\right)\mathbf{1}_{[\mathcal{N}_{\Gamma} \ge k]}\Big)
\le \sum_{k=1}^{\infty}\sup_{z \in [y, \beta/4]}\mathbb{E}_{(0,z)}\left(\Gamma_{5}\right)\sup_{w \in [y, \beta/4]}\prob_{(0,w)}\left(\mathcal{N}_{\Gamma} \ge k\right)\\
&\le C\log\left(\frac{\beta}{y}\right)\sum_{k=1}^{\infty}(1-p_0)^k = C'\log\left(\frac{\beta}{y}\right),
\end{align*}
which completes the proof of the lemma.
\end{proof}

\begin{proof}[Proof of Lemma~\ref{middle}]
Take any $z \in [\beta^{-1}, \beta/4]$ and any $y \in [\beta \e^{-\mathcal{C}_1^- \e^{\mathcal{C}_2^-\beta^2}}, \beta/8]$ satisfying $z \ge 2y$. We can write
\begin{align}\label{inter1}
\mathbb{E}_{(0,z)}\left(\tau_2(y)\right) = \mathbb{E}_{(0,z)}\left(\tau_2(y)\mathbf{1}_{[\tau_2(\beta/4)\wedge \tau_2(y) < \tau_1(-\beta/4)]}\right) + \mathbb{E}_{(0,z)}\left(\tau_2(y)\mathbf{1}_{[\tau_2(\beta/4)\wedge \tau_2(y) \ge \tau_1(-\beta/4)]}\right).
\end{align}
Using the strong Markov property, we obtain
\begin{eq}\label{inter2}
&\mathbb{E}_{(0,z)}\left(\tau_2(y)\mathbf{1}_{[\tau_2(\beta/4) \wedge \tau_2(y) < \tau_1(-\beta/4)]}\right)\\
&\le \mathbb{E}_{(0,z)}\left(\tau_2(y)\mathbf{1}_{[\tau_2(y) \le \tau_1(-\beta/4)]}\right) + \mathbb{E}_{(0,z)}\left(\tau_2(y)\mathbf{1}_{[\tau_2(\beta/4) < \tau_1(-\beta/4)\wedge \tau_2(y)]}\right)\\
&\le \mathbb{E}_{(0,z)}\left(\tau_2(y)\mathbf{1}_{[\tau_2(y) \le \tau_1(-\beta/4)]}\right) + \mathbb{E}_{(0,z)}\left(\tau_2(\beta/4)\mathbf{1}_{[\tau_2(\beta/4) < \tau_1(-\beta/4)\wedge \tau_2(y)]}\right)\\
&\hspace{5cm}+ \prob_{(0,z)}\left(\tau_2(\beta/4) < \tau_1(-\beta/4)\right) \mathbb{E}_{(0,\beta/4)}\left(\tau_2(y)\right)\\
&\le 2\mathbb{E}_{(0,z)}\left(\tau_1(-\beta/4)\right) + \prob_{(0,z)}\left(\tau_2(\beta/4) < \tau_1(-\beta/4)\right) \mathbb{E}_{(0,\beta/4)}\left(\tau_2(y)\right).
\end{eq}
For $t \le \tau_1(-\beta/4)$, $Q_1(t) \le S(t) \le S(0) + \sqrt{2}W(t) -3\beta t/4 \le \beta/4 + \sqrt{2}W(t) -3\beta t/4$. Therefore, $\mathbb{E}_{(0,z)}\left(\tau_1(-\beta/4)\right) \le C$. Furthermore, for any $u \ge z$, $Q_2(t)$ hits level $u$ if and only if $S(t)$ hits level $u$ and $Q_1(t) \le S(t) \le S(0) + \sqrt{2}W(t) -3\beta t/4$ for $t \le \tau_1(-\beta/4)$. Thus,
\begin{multline}\label{hites}
\prob_{(0,z)}\left(\tau_2(u) < \tau_1(-\beta/4)\right) \le \prob\left(z + \sqrt{2}W(t) -3\beta t/4 \text{ hits } u \text{ before } -\beta/4\right)\\
= \frac{\e^{3\beta z/4} - \e^{-3\beta^2/16}}{\e^{3\beta u/4} - \e^{-3\beta^2/16}} \le \e^{-\frac{3\beta}{4}\left(u - z\right)}.
\end{multline}
Combining the above estimate with $u=\beta/4$ with Lemma \ref{middletosmall} and noting that $z \in [\beta^{-1}, \beta/4]$ and $z \ge 2y$, we obtain
\begin{multline*}
\prob_{(0,z)}\left(\tau_2(\beta/4) < \tau_1(-\beta/4)\right) \mathbb{E}_{(0,\beta/4)}\left(\tau_2(y)\right) \le C\e^{-\frac{3\beta}{4}\left(\frac{\beta}{4} - z\right)}\log\left(\frac{\beta}{y}\right)\\
\le \log\left(\frac{z}{y}\right) + C\e^{-\frac{3\beta}{4}\left(\frac{\beta}{4} - z\right)}\log\left(\frac{\beta}{z}\right) \le  \log\left(\frac{z}{y}\right) + C\e^{-\frac{3}{4}\left(\frac{\beta}{4z}-1\right)}\log\left(\frac{\beta}{z}\right) \le C\log\left(\frac{z}{y}\right).
\end{multline*}
Since $z \ge 2y$, using the above estimates in \eqref{inter2},
\begin{align}\label{inter3}
\mathbb{E}_{(0,z)}\left(\tau_2(y)\mathbf{1}_{[\tau_2(\beta/4) \wedge \tau_2(y) < \tau_1(-\beta/4)]}\right) \le C\log\left(\frac{z}{y}\right).
\end{align}
Now, we estimate the second term in \eqref{inter1}. 
Using the strong Markov property at $\tau_1(-\beta/4)$,
\begin{eq}\label{inter4}
&\mathbb{E}_{(0,z)}\left(\tau_2(y)\mathbf{1}_{[\tau_2(\beta/4)\wedge \tau_2(y) \ge \tau_1(-\beta/4)]}\right)\\
&\le \mathbb{E}_{(0,z)}\left(\tau_1(-\beta/4)\right) + \mathbb{E}_{(0,z)}\left(\mathbf{1}_{[\tau_2(\beta/4)\wedge \tau_2(y) \ge \tau_1(-\beta/4)]}\mathbb{E}_{(-\beta/4, Q_2(\tau_1(-\beta/4)))}\left(\tau_2(y)\right)\right)\\
&\le C + \mathbb{E}_{(0,z)}\left(\mathbf{1}_{[\tau_2(\beta/4)\wedge \tau_2(y) \ge \tau_1(-\beta/4)]}\mathbb{E}_{(-\beta/4, Q_2(\tau_1(-\beta/4)))}\left(\tau_2(y)\right)\right).
\end{eq}
Recall that when $Q_1<0$, $Q_2$ decays exponentially. Using this fact and the strong Markov property, we obtain for any $w \in [y, \beta/4]$,
\begin{eq}\label{smalluse}
\mathbb{E}_{(-\beta/4, w)}\left(\tau_2(y)\right) &= \mathbb{E}_{(-\beta/4, w)}\left(\tau_2(y)\mathbf{1}_{[\tau_1(0) \ge \tau_2(y)]}\right) + \mathbb{E}_{(-\beta/4, w)}\left(\tau_2(y)\mathbf{1}_{[\tau_1(0) < \tau_2(y)]}\right)\\
&\le \mathbb{E}_{(-\beta/4, w)}\left(\tau_2(y)\mathbf{1}_{[\tau_1(0) \ge \tau_2(y)]}\right) + \mathbb{E}_{(-\beta/4, w)}\left(\tau_1(0)\mathbf{1}_{[\tau_1(0) < \tau_2(y)]}\right)\\
 &\hspace{3cm}+ \prob_{(-\beta/4, w)}\left(\tau_1(0) < \tau_2(y)\right)\sup_{u \in [y,\beta/4]}\mathbb{E}_{(0, u)}\left(\tau_2(y)\right)\\
&\le  2\log\left(\frac{w}{y}\right) + \sup_{u \in (0,\beta/2]}\prob_{(-\beta/4,u)}\left(\tau_1(0) \le \mathcal{C}_1^- \e^{\mathcal{C}_2^-\beta^2}\right) \sup_{u \in [y,\beta/4]}\mathbb{E}_{(0, u)}\left(\tau_2(y)\right)\\
&\le 2\log\left(\frac{w}{y}\right) + \left(\mathcal{D}_1 \e^{-\mathcal{D}_2\beta^2}\right)\left(C\log\left(\frac{\beta}{y}\right)\right),
\end{eq}
where the last line follows from Lemma \ref{OUhit} and Lemma \ref{middletosmall}, and the second to last line can be understood as follows.
Note that $Q_2$ decreases exponentially before $\tau_1(0)$. 
Therefore, it is clear that starting from $(-\beta/4, w)$, $\tau_2(y)\mathbf{1}_{[\tau_1(0) \geq \tau_2(y)]}\leq \log(w/y)$.
Also, if $\tau_1(0) < \tau_2(y)$, then again by the same reasoning $\tau_1(0)\mathbf{1}_{[\tau_1(0) < \tau_2(y)]}\leq \log(w/y)$.

Now, using the above estimate in \eqref{inter4} and recalling $z \in [\beta^{-1}, \beta/4]$,
\begin{eq}\label{inter5}
&\mathbb{E}_{(0,z)}\left(\tau_2(y)\mathbf{1}_{[\tau_2(\beta/4)\wedge \tau_2(y) \ge \tau_1(-\beta/4)]}\right)\\
& \le C + 2\mathbb{E}_{(0,z)}\mathbf{1}_{[\tau_2(y) \ge \tau_1(-\beta/4)]}\log\left(\frac{Q_2(\tau_1(-\beta/4))}{y}\right)
 + \left(\mathcal{D}_1 \e^{-\mathcal{D}_2\beta^2}\right)\left(C\log\left(\frac{\beta}{y}\right)\right)\\
&\le C + 2\mathbb{E}_{(0,z)}\mathbf{1}_{[\tau_2(y) \ge \tau_1(-\beta/4)]}\log\left(\frac{Q_2(\tau_1(-\beta/4))}{y}\right) + C\log\left(\frac{z}{y}\right) + \left(\mathcal{D}_1 \e^{-\mathcal{D}_2\beta^2}\right)\left(C\log\left(\frac{\beta}{z}\right)\right)\\
&\le C + 2\mathbb{E}_{(0,z)}\mathbf{1}_{[\tau_2(y) \ge \tau_1(-\beta/4)]}\log\left(\frac{Q_2(\tau_1(-\beta/4))}{y}\right) + C\log\left(\frac{z}{y}\right) + \left(\mathcal{D}_1 \e^{-\mathcal{D}_2\beta^2}\right)\left(2C\log\left(\beta\right)\right)\\
&\le C' + 2\mathbb{E}_{(0,z)}\mathbf{1}_{[\tau_2(y) \ge \tau_1(-\beta/4)]}\log\left(\frac{Q_2(\tau_1(-\beta/4))}{y}\right) + C\log\left(\frac{z}{y}\right).
\end{eq}
Write $Q_2^* = \sup_{t \le \tau_1(-\beta/4)}Q_2(t)$. Then,
\begin{align*}
&\mathbb{E}_{(0,z)}\mathbf{1}_{[\tau_2(y) \ge \tau_1(-\beta/4)]}\log\left(\frac{Q_2(\tau_1(-\beta/4))}{y}\right)\\
&\le \mathbb{E}_{(0,z)}\log\left(\frac{Q_2^*}{y}\right) = \log\left(\frac{z}{y}\right) + \int_{\log\left(\frac{z}{y}\right)}^{\infty}\prob_{(0,z)}\left(Q_2^* \ge y\e^u\right)du\\
&\le \log\left(\frac{z}{y}\right) + \int_{\log\left(\frac{z}{y}\right)}^{\infty}\e^{-\frac{3\beta}{4}\left(y\e^u - z\right)}du \ \ \text{ (using \eqref{hites})}.
\end{align*}
Substituting $v = y\e^u - z$ and recalling $z \ge \beta^{-1}$, the integral can be estimated as
\begin{multline*}
\int_{\log\left(\frac{z}{y}\right)}^{\infty}\e^{-\frac{3\beta}{4}\left(y\e^u - z\right)}du = \int_0^{\infty}\e^{-\frac{3\beta}{4}v}\frac{dv}{v+z} = \int_0^{\infty}\e^{-\frac{3\beta z}{4}w}\frac{dw}{w+1} \le \int_0^{\infty}\e^{-\frac{3}{4}w}\frac{dw}{w+1} = C
\end{multline*}
where $C$ does not depend on $\beta, y,z$. Thus,
$$
\mathbb{E}_{(0,z)}\mathbf{1}_{[\tau_2(y) \ge \tau_1(-\beta/4)]}\log\left(\frac{Q_2(\tau_1(-\beta/4))}{y}\right) \le \log\left(\frac{z}{y}\right) + C.
$$
Using the above estimate in \eqref{inter5}, we obtain
\begin{equation}\label{inter6}
\mathbb{E}_{(0,z)}\left(\tau_2(y)\mathbf{1}_{[\tau_2(\beta/4)\wedge \tau_2(y) \ge \tau_1(-\beta/4)]}\right) \le C\log\left(\frac{z}{y}\right).
\end{equation}
Using \eqref{inter3} and \eqref{inter6} in \eqref{inter1}, we finally obtain
$$
\mathbb{E}_{(0,z)}\left(\tau_2(y)\right) \le C\log\left(\frac{z}{y}\right),
$$
which completes the proof of the lemma.
\end{proof}

\begin{proof}[Proof of Lemma~\ref{smallval}]
Let $(Q_1(0), Q_2(0)) = (0,y/2)$ for some $y \in [2\beta \e^{-\mathcal{C}_1^-\e^{\mathcal{C}_2^-\beta^2}},2\beta^{-1}]$. Define the stopping times: $\mathbf{f}_0 = 0$ and for $k \ge 0$,
$$
\mathbf{f}_{2k+1} = \inf\{t \ge \mathbf{f}_{2k}: Q_2(s) = y\}, \ \ \mathbf{f}_{2k+2} = \inf\{t \ge \mathbf{f}_{2k+1}: Q_2(s) = y/2 \text{ or } Q_2(s) = 2\beta^{-1}\}.
$$
Let $\mathcal{N}_{\mathbf{f}} = \inf\{k \ge 1: Q_2(\mathbf{f}_{2k}) = 2\beta^{-1}\}$. Using strong Markov property,
\begin{eq}\label{small0}
&\mathbb{E}_{(0,y/2)} \left(\mathbf{f}_{2} - \mathbf{f}_{1}\right) = \mathbb{E}_{(0,y)} \left(\tau_2(2\beta^{-1}) \wedge \tau_2(y/2)\right)\\
 &= \mathbb{E}_{(0,y)} \left(\tau_2(2\beta^{-1}) \wedge \tau_2(y/2)\mathbf{1}_{[\tau_2(2\beta^{-1}) \wedge \tau_2(y/2) < \tau_1(-\beta/4)]}\right)\\
&\hspace{5cm}  + \mathbb{E}_{(0,y)} \left(\tau_2(2\beta^{-1}) \wedge \tau_2(y/2)\mathbf{1}_{[\tau_2(2\beta^{-1}) \wedge \tau_2(y/2) \ge \tau_1(-\beta/4)]}\right)\\
&  \le 2\mathbb{E}_{(0,y)} \left( \tau_1(-\beta/4)\right) + \mathbb{E}_{(0,y)} \left[((\tau_2(2\beta^{-1}) \wedge \tau_2(y/2)) - \tau_1(-\beta/4))\mathbf{1}_{[\tau_2(2\beta^{-1}) \wedge \tau_2(y/2) \ge \tau_1(-\beta/4)]}\right]\\
&  \le 2\mathbb{E}_{(0,y)} \left( \tau_1(-\beta/4)\right) + \sup_{z \in (y/2, \ 2\beta^{-1})}\mathbb{E}_{(-\beta/4, z)}(\tau_2(y/2)).
\end{eq}
For $t \le \tau_1(-\beta/4)$, $Q_1(t) \le S(t) \le S(0) + \sqrt{2}W(t) -3\beta t/4$. Hence, $\displaystyle{\sup_{y \in (0, 2\beta^{-1}]}\mathbb{E}_{(0,y)}\left(\tau_1(-\beta/4)\right) \le C}$. Moreover, for any $y \in [2\beta \e^{-\mathcal{C}_1^-\e^{\mathcal{C}_2^-\beta^2}}, 2\beta^{-1}]$ and any $z \in (y/2, \ 2\beta^{-1})$, by \eqref{smalluse},
\begin{multline*}
\mathbb{E}_{(-\beta/4, z)}(\tau_2(y/2)) \le 2\log\left(\frac{2z}{y}\right) + \left(\mathcal{D}_1 \e^{-\mathcal{D}_2\beta^2}\right)\left(C\log\left(\frac{2\beta}{y}\right)\right)\\
\le 2\log\left(\frac{4}{\beta y}\right) + \left(\mathcal{D}_1 \e^{-\mathcal{D}_2\beta^2}\right)\left(C\log\left(\frac{2}{\beta y}\right)\right) + \left(\mathcal{D}_1 \e^{-\mathcal{D}_2\beta^2}\right)(2C \log \beta)
\le C \log\left(\frac{4}{\beta y}\right)
\end{multline*}
for all $\beta \ge \beta_0$ for sufficiently large $\beta_0$, where $C$ in the final bound does not depend on $\beta, y$. Using these estimates in \eqref{small0}, we obtain $C>0, \beta_0 \ge 1$ such that for all $\beta \ge \beta_0$ and all $y \in [2\beta \e^{-\mathcal{C}_1^-\e^{\mathcal{C}_2^-\beta^2}},2\beta^{-1}]$,
\begin{equation}\label{small1}
\mathbb{E}_{(0,y/2)} \left(\mathbf{f}_{2} - \mathbf{f}_{1}\right) \le C \log\left(\frac{4}{\beta y}\right).
\end{equation}
If $(Q_1(0), Q_2(0)) = (0,y)$, then $Q_2(t) \ge S(t) \ge y + \sqrt{2}W(t) - \beta t$. Furthermore, for any $t \le \beta^{-2}$, $Q_2(t) \ge \e^{-t}y \ge \e^{-\beta^{-2}}y > y/2$ for all $\beta \ge \beta_0$ if $\beta_0$ is chosen large enough. Therefore,
\begin{align*}
\inf_{y \in (0,2\beta^{-1}]}\prob_{(0,y)}\left(Q_2 \text{ hits } 2\beta^{-1} \text{ before } y/2\right) &\ge \inf_{y \in (0,2\beta^{-1}]}\prob_{(0,y)}\Big(\sup_{t \le \beta^{-2}}S(t) \ge 2\beta^{-1}\Big)\\
&\ge \prob\Big(\sup_{t \le \beta^{-2}}\sqrt{2}W(t) \ge 3\beta^{-1}\Big) \ge p_{\mathbf{f}}>0,
\end{align*}
where $p_{\mathbf{f}}$ does not depend on $\beta$. This gives us for $k \ge 1$,
\begin{equation}\label{small2}
\sup_{y \in (0,2\beta^{-1}]}\prob_{(0,y/2)}\left(\mathcal{N}_{\mathbf{f}} \ge k\right) \le (1-p_{\mathbf{f}})^k.
\end{equation}
Using \eqref{small1} and \eqref{small2}, we obtain
\begin{align*}
&\mathbb{E}_{(0,y/2)}\Big(\int_0^{\tau_2(2\beta^{-1})}\mathbf{1}_{[Q_2(s) \ge y]}ds\Big)
\le \mathbb{E}_{(0,y/2)}\Big(\sum_{k=1}^{\infty}\left(\mathbf{f}_{2k}-\mathbf{f}_{2k-1}\right)\mathbf{1}_{[\mathcal{N}_{\mathbf{f}} > k-1]}\Big)\\
&\le \sum_{k=1}^{\infty}\mathbb{E}_{(0,y/2)} \left(\mathbf{f}_{2} - \mathbf{f}_{1}\right)\sup_{y \in (0,2\beta^{-1}]}\prob_{(0,y/2)}\left(\mathcal{N}_{\mathbf{f}} \ge k\right)
\le C \log\left(\frac{4}{\beta y}\right)\sum_{k=1}^{\infty}(1-p_{\mathbf{f}})^k = C' \log\left(\frac{4}{\beta y}\right),
\end{align*}
where $C'$ does not depend on $\beta, y$. This completes the proof of the lemma.
\end{proof}

\section{Proof of Lemma \ref{largebeta}}
\label{app:lemma4.8}
Lemma \ref{Q2gebeta2} gives us an upper bound on the tail probabilities of $Q_2$ in the region $[\beta + y _0(\beta), \infty)$. 
In this appendix, we will extend these estimates to the region $[\beta^{-1}, \infty)$.
We start by recording a corollary to Lemma \ref{Q2gebeta2} which will be useful in proving finer tail estimate.
\begin{lemma}[Corollary to Lemma \ref{Q2gebeta2}]\label{ubQ2}
Take any $\epsilon \in (0, R^+)$, where $R^+$ is the constant in Lemma \ref{Q2gebeta2}. There exist positive constants $\beta_0, C^*_1,$ and $C^*_2$, such that for all fixed $\beta \ge \beta_0$,
\begin{equation*}
\prob_{(0, 2(1 + \epsilon)\beta)}\left(\tau_2(z) \le \tau_2((1+\epsilon)\beta)\right) \le C^*_1\e^{-C^*_2\beta z}
\end{equation*}
for $z \ge 4(1 + \epsilon)\beta$.
\end{lemma}
\begin{proof}
Take $\beta_0\ge 1$ satisfying $\frac{R^+}{\beta_0} = \epsilon \beta_0$. Consider any $\beta \ge \beta_0$. Recalling $y_0(\beta) = R^+\beta^{-1}$ from Lemma \ref{Q2gebeta2}, we have  $y_0(\beta) + \beta \le (1+\epsilon)\beta$. Also, $2y + \beta < 2(y+\beta)$ for any $y$. For any $z \ge 2(1 + \epsilon)\beta$, write $y=\frac{z}{2} - \beta$. Then $y \ge \epsilon \beta \ge R^+\beta^{-1}$ by our choice of $\beta_0$. Therefore, by the strong Markov property and Lemma \ref{Q2gebeta2}, for any $z \ge 4(1 + \epsilon)\beta$,
\begin{align*}
\prob_{(0, 2(1 + \epsilon)\beta)}\left(\tau_2(z) \le \tau_2((1+\epsilon)\beta)\right)  &= \prob_{(0, 2(1 + \epsilon)\beta)}\left(\tau_2(2(y+\beta)) \le \tau_2((1+\epsilon)\beta)\right)\\
&\le\prob_{(0, 2(1 + \epsilon)\beta)}\left(\tau_2(2y+\beta) \le \tau_2(y_0(\beta) + \beta)\right) \\
&\le \prob_{(0, y + \beta)}\left(\tau_2(2y+\beta) \le \tau_2(y_0(\beta) + \beta)\right)
\le C^*_1\e^{-C^*_2\beta y}.
\end{align*}
As $z \ge 4(1 + \epsilon)\beta$, $\beta \le \frac{z}{4}$ and hence, $y = \frac{z}{2} -\beta \ge \frac{z}{4}$, completing the proof of the corollary.
\end{proof}
As mentioned in Remark~\ref{rem:lem4.6} in detail, the diffusion process starting in the region $\{-Q_1 + Q_2 < \beta\}$ shows a different qualitative behavior than the $\{-Q_1+ Q_2 > \beta\}$ region. 
Lemma \ref{Q2gebeta2} exploits the linear drift of $Q_2$ to produce an exponential steady-state tail estimate in the latter region. 
The next lemma studies the tail behavior of $Q_2$ when $\{-Q_1 + Q_2 < \beta\}$. 
\begin{lemma}\label{Q2lebeta}
Fix any $\theta_0 \in (0,1)$ and any $A>\max\left\lbrace\theta_0, \frac{1}{2}\right\rbrace$. There exist constants $C_1, C_2>0$ (depending only on $\theta_0, A$) such that for all $\beta \ge \theta_0^{-1/2}$ and all $\theta \in [\beta^{-2}, \theta_0]$,
\begin{equation*}
\sup_{z \in [\beta^{-1}, \theta \beta]} \prob_{(0,z)}\left(\tau_2(z+y) \le \tau_2\left(\beta^{-1}\right) \right) \le C_1 \e^{-C_2 \beta y}, \  y \in [\theta \beta, A\beta].
\end{equation*}
\end{lemma}
\begin{proof}
Fix any $\theta_0 \in (0,1)$ and any $A>\max\left\lbrace\theta_0, \frac{1}{2}\right\rbrace$. Take any $\beta \ge \theta_0^{-1/2}$ and any $\theta \in [\beta^{-2}, \theta_0]$. Finally, take any $z \in [\beta^{-1}, \theta \beta]$. Define
$
\sigma = \inf\{t \ge \tau_1\left(-(1-\theta)\beta/8\right) : Q_1(t)=0\}.
$
Note that for $y \in [\theta \beta, A\beta]$,
\begin{multline}\label{Q2twopart}
\prob_{(0,z)}\left(\tau_2(z+y) \le \tau_2\left(\beta^{-1}\right) \right) \le \prob_{(0,z)}\Big(\tau_2\Big(z+\frac{(1-\theta)y}{2A}\Big) \le \tau_1\left(-(1-\theta)\beta/8\right) \Big)\\
+ \prob_{(0,z)}\Big(\tau_2\Big(z+\frac{(1-\theta)y}{2A}\Big) > \tau_1\left(-(1-\theta)\beta/8\right), \sigma \le \tau_2\left(\beta^{-1}\right)\Big),
\end{multline}
where the last probability uses the fact that as $A>\frac{1}{2}$, $z + \frac{(1-\theta)y}{2A} < z + y$.
We will estimate the two probabilities separately. Note that for $t \le \tau_1\left(-(1-\theta)\beta/8\right)$,
\begin{multline*}
L(t) = \sup_{s \le t}\left(\sqrt{2}W(s) - \beta s + \int_0^s(-Q_1(u) + Q_2(u))du\right) \le \sup_{s \le t}\left(\sqrt{2}W(s) - \frac{7+\theta}{8}\beta s\right) + \int_0^t Q_2(s)ds.
\end{multline*}
Recall that
$$
Q_2(t) - Q_2(0) = L(t) - \int_0^t Q_2(s)ds.
$$
Therefore, using the fact that the scale function of $\sqrt{2}W(t) - b\beta t$ is $s_b(u) = \e^{bu}$ for any $b >0$ and $u \in \mathbb{R}$,
\begin{multline}\label{first}
\prob_{(0,z)}\Big(\tau_2\Big(z+\frac{(1-\theta)y}{2A}\Big) \le \tau_1\left(-(1-\theta)\beta/8\right) \Big) \le \prob\Big(\sup_{s < \infty}\Big(\sqrt{2}W(s) - \frac{7+\theta}{8}\beta s\Big) > \frac{(1-\theta)y}{2A}\Big)\\ =\e^{-\frac{(7+\theta)(1-\theta)\beta y}{16A}} \le \e^{-\frac{7(1-\theta_0)\beta y}{16A}}.
\end{multline}
Now we estimate the second probability of \eqref{Q2twopart}. Applying the strong Markov property at $\tau_1\left(-(1-\theta)\beta/8\right)$,
\begin{multline}\label{SM}
\prob_{(0,z)}\left(\tau_2\left(z+\frac{(1-\theta)y}{2A}\right) > \tau_1\left(-(1-\theta)\beta/8\right), \sigma \le \tau_2\left(\beta^{-1}\right)\right)\\
\le \sup_{w \in \left[0, \ z+\frac{(1-\theta)y}{2A}\right]}\prob_{(-(1-\theta)\beta/8, w)}\left(\tau_1(0) \le \tau_2(\beta^{-1})\right).
\end{multline}
Therefore, it suffices to estimate the probability appearing in the right hand side above for $Q_1(0)= -(1-\theta)\beta/8$ and $Q_2(0) = w$ for $w \in \left[0, z+\frac{(1-\theta)y}{2A}\right]$. Towards this end, define the following stopping times: $\sigma_0=0$ and for $k \ge 0$,
\begin{align*}
\sigma_{2k+1} &= \inf\{t \ge \sigma_{2k}: Q_1(t) = -(1-\theta)\beta/4 \text{ or } 0\},\\
\sigma_{2k+2} &= \inf\{t \ge \sigma_{2k+1}: Q_1(t) = -(1-\theta)\beta/8 \text{ or } 0\}.
\end{align*}
Let $\mathcal{N}^* = \inf\{k \ge 0: Q_1(\sigma_{2k+1}) = 0\}$. Suppose $\mathcal{N}^*  \ge k$. For $t \in [\sigma_{2k}, \sigma_{2k+1}]$, $Q_1(t) \ge -(1-\theta)\beta/4$ and $Q_2(t) \le z+\frac{(1-\theta)y}{2A}  \le \frac{(1+\theta)\beta}{2}$ (as $\frac{(1-\theta)y}{2A} \le \frac{(1-\theta)\beta}{2}$ and $z \le \theta \beta$). Therefore,
\begin{equation}\label{rev1}
Q_1(t) = -(1-\theta)\beta/8 + \sqrt{2}W(t) - \beta t + \int_0^t(-Q_1(s) + Q_2(s))ds \le -(1-\theta)\beta/8 + \sqrt{2}W(t) - \frac{1- \theta}{4}\beta t.
\end{equation}
Therefore, by the strong Markov property and scale function arguments, for any $k \ge 0$,
\begin{align*}
&\prob_{(-(1-\theta)\beta/8, w)}\left(Q_1(\sigma_{2k+1})=0, \ \mathcal{N}^* \ge k\right)\\
&\le \prob\left(\sqrt{2}W(t) - \frac{1-\theta}{4}\beta t\text{ hits } (1-\theta)\beta/8 \text{ before } -(1-\theta)\beta/4\right)\\
&= \frac{1- \e^{-(1-\theta)^2\beta^2/16}}{\e^{(1-\theta)^2\beta^2/32}-\e^{-(1-\theta)^2\beta^2/16}} \le \e^{-(1-\theta)^2\beta^2/32} \le \e^{-(1-\theta_0)^2\beta^2/32}.
\end{align*}
Consequently, for any $n \ge 0$, $\prob_{(-(1-\theta)\beta/8, w)}(\mathcal{N}^* \le n) \le (n+1) \e^{-(1-\theta)^2\beta^2/32}$. Further, observe that $\tau_1(0) \ge \sum_{k=1}^{\mathcal{N}^*}\left(\sigma_{2k+1} - \sigma_{2k}\right)$. Moreover, if $\mathcal{N}^* \ge k$, then for $t \in [\sigma_{2k}, \sigma_{2k+1}]$, $Q_1(t) \ge -(1-\theta)\beta/8 + \sqrt{2}W(t) - \beta t$ and from \eqref{rev1}, $Q_1(t) \le -(1-\theta)\beta/8 + \sqrt{2}W(t) - \frac{1- \theta}{4}\beta t$. Therefore,
\begin{align*}
&\prob_{(-(1-\theta)\beta/8, w)}\left(\left(\sigma_{2k+1} - \sigma_{2k}) \le (1-\theta)/16\right), \ \mathcal{N}^* \ge k\right)\\
&\le \prob\Big(\inf_{t \le \frac{(1-\theta)}{16}}(\sqrt{2}W(t) - \beta t)\le -(1-\theta)\beta/8\Big) +  \prob\Big(\sup_{t \le \frac{(1-\theta)}{16}}\Big(\sqrt{2}W(t) - \frac{1- \theta}{4}\beta t\Big) \ge (1-\theta)\beta/8\Big)\\
&\le \prob\Big(\inf_{t \le (1-\theta)/16}\sqrt{2}W(t) \le -(1-\theta)\beta/16\Big) +  \prob\Big(\sup_{t \le (1-\theta)/16}\sqrt{2}W(t) \ge (1-\theta)\beta/8\Big)\\
&\le \frac{8\sqrt{2}}{\sqrt{2\pi}\beta \sqrt{1-\theta}} \e^{-(1-\theta)\beta^2/64} + \frac{4\sqrt{2}}{\sqrt{2\pi}\beta \sqrt{1-\theta}} \e^{-(1-\theta)\beta^2/16} \le \frac{12\sqrt{2\theta_0}}{\sqrt{2\pi}\sqrt{1-\theta_0}} \e^{-(1-\theta_0)\beta^2/64},
\end{align*}
where the last inequality follows from $\beta \ge \theta_0^{-1/2}$. Hence, for any $n \ge 0$ (whose value will be appropriately chosen later),
\begin{eq}\label{tauless}
 &\sup_{w \in \left[0, \ z+\frac{(1-\theta)y}{2A}\right]}\prob_{(-(1-\theta)\beta/8, w)}\left(\tau_1(0) \le \frac{n(1-\theta)}{16}\right) \\
& \le \sup_{w \in \left[0, \ z+\frac{(1-\theta)y}{2A}\right]}\prob_{(-(1-\theta)\beta/8, w)}\left(\sum_{k=0}^{n}\left(\sigma_{2k+1} - \sigma_{2k}\right) \le \frac{n(1-\theta)}{16}, \ \mathcal{N}^* \ge n\right)\\
&\hspace{7cm}+\sup_{w \in \left[0, \ z+\frac{(1-\theta)y}{2A}\right]}\prob_{(-(1-\theta)\beta/8, w)}(\mathcal{N}^* \le n)\\
 &\le (n+1) \e^{-(1-\theta_0)^2\beta^2/32} + (n+1) \frac{12\sqrt{2\theta_0}}{\sqrt{2\pi}\sqrt{1-\theta_0}} \e^{-(1-\theta_0)\beta^2/64}.
\end{eq}
If $Q_2(0) \le  z+\frac{(1-\theta)y}{2A}$, then as $Q_1(t) <0$ for all $t < \tau_1(0)$,
$$
Q_2(\tau_1(0)) \le \left(z+\frac{(1-\theta)y}{2A}\right)\e^{-\tau_1(0)} \le \left(\theta \beta +\frac{(1-\theta)y}{2A}\right)\e^{-\tau_1(0)} \le \left(1 +\frac{(1-\theta)}{2A}\right)y\e^{-\tau_1(0)}.
$$
Thus, for $\tau_1(0) \le \tau_2\left(\beta^{-1}\right)$ to hold, we must have $\left(1 +\frac{(1-\theta)}{2A}\right)y\e^{-\tau_1(0)} \ge \beta^{-1}$ or equivalently,
$$
\tau_1(0) \le \log \left[\left(1 +\frac{(1-\theta)}{2A}\right) \beta y\right].
$$
This observation, combined with \eqref{tauless}, taking $n$ to be the greatest integer greater than or equal to $\frac{16}{(1-\theta)}\log\left[\left(1 +\frac{(1-\theta)}{2A}\right) \beta y\right]$, yields the following estimate
\begin{align*}
&\sup_{w \in \left[0, \ z+\frac{(1-\theta)y}{2A}\right]}\prob_{(-(1-\theta)\beta/8, w)}\left(\tau_1(0) \le \tau_2(\beta^{-1})\right)\\
&\le \sup_{w \in \left[0, \ z+\frac{(1-\theta)y}{2A}\right]}\prob_{(-(1-\theta)\beta/8, w)}\left(\tau_1(0) \le \log \left[\left(1 +\frac{(1-\theta)}{2A}\right) \beta y\right]\right)\\
&\le \left(\frac{16}{(1-\theta)}\log\left[\left(1 +\frac{(1-\theta)}{2A}\right) \beta y\right] +2\right)\left(1 + \frac{12\sqrt{2\theta_0}}{\sqrt{2\pi}\sqrt{1-\theta_0}}\right)\e^{-(1-\theta_0)^2\beta^2/64}.
\end{align*}
This, by \eqref{SM}, yields
\begin{multline}\label{second}
\prob_{(0,z)}\left(\tau_2\left(z+\frac{(1-\theta)y}{2A}\right) > \tau_1\left(-(1-\theta)\beta/8\right), \sigma \le \tau_2\left(\beta^{-1}\right)\right)\\
\le \left(\frac{16}{(1-\theta)}\log\left[\left(1 +\frac{(1-\theta)}{2A}\right) \beta y\right]+2\right)\left(1 + \frac{12\sqrt{2\theta_0}}{\sqrt{2\pi}\sqrt{1-\theta_0}}\right)\e^{-(1-\theta_0)^2\beta^2/64}.
\end{multline}
Using the estimates \eqref{first} and \eqref{second} in \eqref{Q2twopart} and noting $y \le A\beta$, we finally obtain
\begin{align*}
&\prob_{(0,z)}\left(\tau_2(z+y) \le \tau_2\left(\beta^{-1}\right) \right)\\
&\le \e^{-\frac{7(1-\theta_0)\beta y}{16A}} +  \left(\frac{16}{(1-\theta)}\log\left[\left(1 +\frac{(1-\theta)}{2A}\right) \beta y\right]+2\right)\left(1 + \frac{12\sqrt{2\theta_0}}{\sqrt{2\pi}\sqrt{1-\theta_0}}\right)\e^{-(1-\theta_0)^2\beta^2/64}\\
&\le \e^{-\frac{7(1-\theta_0)\beta y}{16A}} +  \left(\frac{16}{(1-\theta)}\log\left[\left(1 +\frac{(1-\theta)}{2A}\right) \beta y\right]+2\right)\left(1 + \frac{12\sqrt{2\theta_0}}{\sqrt{2\pi}\sqrt{1-\theta_0}}\right)\e^{-\frac{(1-\theta_0)^2}{A}\beta y/64}\\
&\le \e^{-\frac{7(1-\theta_0)\beta y}{16A}} +  \left(\frac{16}{(1-\theta_0)}\log\left[\left(1 +\frac{1}{2A}\right) \beta y\right]+2\right)\left(1 + \frac{12\sqrt{2\theta_0}}{\sqrt{2\pi}\sqrt{1-\theta_0}}\right)\e^{-\frac{(1-\theta_0)^2}{A}\beta y/64},
\end{align*}
which proves the lemma.
\end{proof}
The following corollary to Lemma~\ref{Q2lebeta} records the tail bound of $Q_2$ in the region $\beta^{-1} \le y \le A \beta$, by taking $\theta_0 = \frac{1}{2}$, $\beta \ge 2$ and $\theta = 2\beta^{-2}$.
\begin{corollary}\label{ls}
Fix any $A>\frac{1}{2}$. Then there exist constants $C_1, C_2>0$ (depending only on $A$) such that for all $\beta \ge 2$,
\begin{equation*}
\prob_{(0,2\beta^{-1})}\left(\tau_2(2\beta^{-1}+y) \le \tau_2\left(\beta^{-1}\right) \right) \le C_1 \e^{-C_2 \beta y}, \  y \in [2\beta^{-1}, A\beta].
\end{equation*}
\end{corollary}
The next three lemmas "patch up" the different behaviors in the regions $\{-Q_1 + Q_2 < \beta\}$ and $\{-Q_1 + Q_2 > \beta\}$ to extend Lemma \ref{Q2gebeta2} to the region $Q_2 \in [\beta^{-1}, \infty)$ for large $\beta$. 
To achieve this, we will show that for sufficiently small $\epsilon>0$, starting from $(Q_1(0), Q_2(0))=(0, (1+\epsilon)\beta)$, the probability that $Q_2$ hits the level $2(1+\epsilon)\beta$ before $\beta^{-1}$ is bounded above by $3/4$ for sufficiently large $\beta$.
\begin{lemma}\label{down1}
Take any $\epsilon>0$ satisfying $(1+ 2\epsilon) \e^{-\frac{1}{4(1+2\epsilon)}} < 1$ and any $\psi \in ((1+ 2\epsilon) \e^{-\frac{1}{4(1+2\epsilon)}}, 1)$. Then there exists $\beta_0^* \ge 1$ depending only on $\epsilon$ such that for all $\beta \ge \beta^*_0$,
\begin{equation*}
\prob_{(0, (1+\epsilon)\beta)}\left(\tau_2((1+2\epsilon)\beta) < \tau_2(\psi \beta)\right) \le \frac{1}{2}.
\end{equation*}
\end{lemma}
\begin{proof}
Take any $\beta \ge 1$. For $t \le \tau_1(-\beta/2)$, the sum $S(t) = Q_1(t) + Q_2(t)$ is bounded above as
$$
S(t) \le (1 + \epsilon)\beta + \sqrt{2}W(t) - \beta t/2.
$$
For $Q_2$ to hit the level $(1+ 2\epsilon)\beta$, the sum $S$ should also hit the same level as times of increase of $Q_2$ correspond to precisely times when $Q_1=0$. Thus,
\begin{multline*}
\prob_{(0, (1+\epsilon)\beta)}\left(\tau_2((1+ 2\epsilon)\beta) < \tau_1(-\beta/2)\right) \le \prob\left(\sup_{t < \infty}\left(\sqrt{2}W(t) - \beta t/2\right) \ge \epsilon \beta \right) = \e^{-\beta^2\epsilon/2}.
\end{multline*}
Consider the event $\{\tau_1(-\beta/2) \le \tau_2((1+ 2\epsilon)\beta)\}$. Define the stopping time
$$
\sigma^* = \inf\{ t \ge \tau_1(-\beta/2) : Q_1(t) = 0 \text{ or } Q_1(t)=-\beta\}.
$$
Under the event $\{\tau_1(-\beta/2) \le \tau_2((1+ 2\epsilon)\beta)\}$, for $t \in [ \tau_1(-\beta/2), \sigma^*]$,
\begin{align*}
Q_1(t) &= -\beta/2 + \sqrt{2}(W(t) - W(\tau_1(-\beta/2))) - \beta\left(t- \tau_1(-\beta/2)\right) + \int_{\tau_1(-\beta/2)}^t (-Q_1(s) + Q_2(s))ds\\
&\le -\beta/2 + \sqrt{2}(W(t) - W(\tau_1(-\beta/2))) + (1+2\epsilon)\beta\left(t- \tau_1(-\beta/2)\right),
\end{align*}
and
$$
Q_1(t) \ge -\beta/2 + \sqrt{2}(W(t) - W(\tau_1(-\beta/2))) - \beta\left(t- \tau_1(-\beta/2)\right).
$$
Therefore,
\begin{align*}
&\prob_{(0, (1+\epsilon)\beta)}\left(\sigma^*- \tau_1(-\beta/2) \le \frac{1}{4(1+2\epsilon)}, \ \tau_1(-\beta/2) \le \tau_2((1+ 2\epsilon)\beta)\right)\\
&\le \prob\Big(-\beta/2 + \sup_{t \le \frac{1}{4(1+2\epsilon)}} \left(\sqrt{2}W(t) + (1+2\epsilon)\beta t\right) \ge 0\Big) + \prob\Big(-\beta/2 + \sup_{t \le \frac{1}{4(1+2\epsilon)}} \left(\sqrt{2}W(t) -\beta t\right) \le -\beta\Big)\\
&\le \prob\Big(-\beta/2 + \sup_{t \le \frac{1}{4(1+2\epsilon)}} \sqrt{2}W(t) + \beta/4 \ge 0\Big) + \prob\Big(-\beta/2 + \inf_{t \le \frac{1}{4(1+2\epsilon)}} \sqrt{2}W(t) -\frac{\beta}{4(1+2\epsilon)} \le -\beta\Big)\\
&\le \prob\Big(\sup_{t \le \frac{1}{4(1+2\epsilon)}} \sqrt{2}W(t) \ge \beta/4\Big) + \prob\Big(\inf_{t \le \frac{1}{4(1+2\epsilon)}} \sqrt{2}W(t)\le -\beta/4\Big) \le \frac{8\sqrt{2}}{\sqrt{2\pi}\beta\sqrt{1+2\epsilon}} \e^{-\frac{(1+2\epsilon)\beta^2}{16}}\\
&\le \frac{8}{\sqrt{\pi}} \e^{-\frac{\beta^2}{16}}.
\end{align*}
On the event $\left\lbrace\sigma^*- \tau_1(-\beta/2) > \frac{1}{4(1+2\epsilon)}, \ \tau_1(-\beta/2) \le \tau_2((1+ 2\epsilon)\beta)\right\rbrace$,
$$
Q_2(\sigma^*) < (1+2\epsilon)\beta \e^{-\frac{1}{4(1+2\epsilon)}} < \psi \beta.
$$
Therefore,
\begin{align*}
\prob_{(0, (1+\epsilon)\beta)}\left(\tau_2((1+2\epsilon)\beta) < \tau_2(\psi \beta)\right)
&\le \prob_{(0, (1+\epsilon)\beta)}\left(\tau_2((1+ 2\epsilon)\beta) < \tau_1(-\beta/2)\right)\\
 &\hspace{-3cm}+ \prob_{(0, (1+\epsilon)\beta)}\left(\sigma^*- \tau_1(-\beta/2) \le \frac{1}{4(1+2\epsilon)}, \ \tau_1(-\beta/2) \le \tau_2((1+ 2\epsilon)\beta)\right)\\
 &\le \e^{-\beta^2\epsilon/2} + \frac{8}{\sqrt{\pi}} \e^{-\frac{\beta^2}{16}}.
\end{align*}
Therefore, choosing any $\beta^*_0 \ge 1$ satisfying $\e^{-(\beta^*_0)^2\epsilon/2} + \frac{8}{\sqrt{\pi}} \e^{-\frac{(\beta^*_0)^2}{16}} \le 1/2$, we obtain for all $\beta \ge \beta^*_0$,
$$
\prob_{(0, (1+\epsilon)\beta)}\left(\tau_2((1+2\epsilon)\beta) < \tau_2(\psi \beta)\right) \le \frac{1}{2},
$$
proving the lemma.
\end{proof}
\begin{lemma}\label{down2}
Take any $\epsilon>0$ satisfying $(1+ 2\epsilon) \e^{-\frac{1}{4(1+2\epsilon)}} < 1$. Then there exists $\beta^{**}_0 \ge 1$ (depending only on $\epsilon$) such that for all $\beta \ge \beta^{**}_0$,
\begin{equation*}
\prob_{(0, (1+\epsilon)\beta)}\left(\tau_2(2(1+\epsilon)\beta) < \tau_2(\beta^{-1})\right) \le \frac{3}{4}.
\end{equation*}
\end{lemma}
\begin{proof}
Take $\psi \in \left((1+ 2\epsilon) \e^{-\frac{1}{4(1+2\epsilon)}}, 1\right)$. By Lemma \ref{down1}, there exists $\beta^*_0 \ge 1$ depending only on $\epsilon$ such that for all $\beta \ge \beta^*_0$,
\begin{equation}\label{d21}
\prob_{(0, (1+\epsilon)\beta)}\left(\tau_2(2(1+\epsilon)\beta) < \tau_2(\psi \beta)\right) \le \prob_{(0, (1+\epsilon)\beta)}\left(\tau_2((1+2\epsilon)\beta) < \tau_2(\psi \beta)\right) \le \frac{1}{2}.
\end{equation}
Now, choosing $\theta=\theta_0=\psi$, $A = 2(1+\epsilon)$, $z=\psi \beta$ and $y= 2(1+\epsilon)\beta - \psi \beta$ in Lemma \ref{Q2lebeta}, we obtain positive constants $C_1, C_2$ depending only on $\epsilon$ and $\beta'_0 \ge 1$ such that
\begin{equation}\label{d22}
\prob_{(0, \psi \beta)} \left(\tau_2(2(1+\epsilon)\beta) \le \tau_2(\beta^{-1})\right) \le C_1 \e^{-C_2\beta(2(1+\epsilon)\beta - \psi \beta)} \le C_1 \e^{-C_2\beta^2} \le \frac{1}{4}
\end{equation}
for all $\beta \ge \beta'_0$. Define $\sigma_{\psi} = \inf\{t > \tau_2(\psi \beta): Q_2(t) \ge \psi \beta\}$. Using \eqref{d21}, \eqref{d22} and the strong Markov property at $\sigma_{\psi}$, we obtain $\beta^{**}_0 = \max\{\beta^*_0, \beta'_0, \psi^{-1/2}\}$ such that for all $\beta \ge \beta^{**}_0$,
\begin{align*}
&\prob_{(0, (1+\epsilon)\beta)}\left(\tau_2(2(1+\epsilon)\beta) < \tau_2(\beta^{-1})\right)\\
&\le \prob_{(0, (1+\epsilon)\beta)}\left(\tau_2(2(1+\epsilon)\beta) < \tau_2(\psi \beta)\right) + \prob_{(0, (1+\epsilon)\beta)}\left(\tau_2(\psi \beta) \le \sigma_{\psi} \le \tau_2(2(1+\epsilon)\beta) < \tau_2(\beta^{-1})\right)\\
&\le \prob_{(0, (1+\epsilon)\beta)}\left(\tau_2(2(1+\epsilon)\beta) < \tau_2(\psi \beta)\right) + \prob_{(0, \psi \beta)} \left(\tau_2(2(1+\epsilon)\beta) \le \tau_2(\beta^{-1})\right) \le \frac{3}{4},
\end{align*}
proving the lemma.
\end{proof}

\begin{proof}[Proof of Lemma~\ref{largebeta}]
Take any $\epsilon>0$ satisfying $(1+ 2\epsilon) \e^{-\frac{1}{4(1+2\epsilon)}} < 1$. The result of the lemma with $\beta_0=2$ when $y \in [4\beta^{-1}, 4(1 + \epsilon)\beta]$ is directly implied by Corollary \ref{ls} taking $A = (6+4\epsilon)$. This, along with the strong Markov property, shows that it suffices to prove 
\begin{equation*}
\prob_{(0,(1 + \epsilon)\beta)}\left(\tau_2(y) < \tau_2(\beta^{-1})\right) \le C_L\e^{-C'_L \beta y}
\end{equation*}
for $y \in (4(1 + \epsilon)\beta, \infty)$. Therefore, we assume the starting configuration to be $Q_1(0)=0$ and $Q_2(0)=(1 + \epsilon)\beta$. For any $y \in (4(1 + \epsilon)\beta, \infty)$, define the following stopping times: $\phi_0=0$ and for $k \ge 0$,
\begin{align*}
\phi_{2k+1} &= \inf\{t \ge \phi_{2k}: Q_2(t)=2(1+\epsilon)\beta \text{ or } Q_2(t)=\beta^{-1}\}\\
\phi_{2k+2} &= \inf\{t \ge \phi_{2k+1}: Q_2(t)=(1+\epsilon)\beta \text{ or } Q_2(t)=\beta^{-1}\}.
\end{align*}
Let $\mathcal{N}^{L}= \inf\{k \ge 1: Q_2(\phi_{2k+1}) = \beta^{-1}\}$. By Lemma \ref{down2} and the strong Markov property, for any $\beta \ge \beta^{**}_0$ and any $k \ge 1$,
\begin{equation}\label{nl}
\prob_{(0,(1+\epsilon)\beta)}\left(\mathcal{N}^{L} \ge k\right) \le \left(\frac{3}{4}\right)^k.
\end{equation}
Therefore, for any $\beta \ge \max\left\lbrace \beta_0, \beta^{**}_0\right\rbrace$ (where $\beta_0$ and $\beta^{**}_0$ appear in Lemma~\ref{ubQ2} and Lemma \ref{down2} respectively) and any $y > 4(1 + \epsilon)\beta$,
\begin{align*}
&\prob_{(0, (1+\epsilon)\beta)}\left(\tau_2(y) < \tau_2(\beta^{-1})\right) = \prob_{(0, (1+\epsilon)\beta)}\Big(\sup_{0 \le t \le \phi_{2\mathcal{N}^{L}}}Q_2(t) > y\Big)\\
&\le \sum_{k=1}^{\infty}\prob_{(0, (1+\epsilon)\beta)}\Big(\sup_{\phi_{2k-1} \le t \le \phi_{2k}}Q_2(t) > y, \mathcal{N}^{L} \ge k\Big)\\
&\le \sum_{k=1}^{\infty}\mathbb{E}_{(0, (1+\epsilon)\beta)}\mathbb{I}\left(\mathcal{N}^{L} \ge k\right)\prob_{(0, 2(1+\epsilon)\beta)}\left(\tau_2(y) < \tau_2((1+\epsilon)\beta)\right) \\
&\le  \sum_{k=1}^{\infty}\prob_{(0, (1+\epsilon)\beta)}\left(\mathcal{N}^{L} \ge k\right)C^*_1 \e^{-C^*_2 \beta y} \ \ (\text{by Lemma~\ref{ubQ2}})\\
&\le \sum_{k=1}^{\infty}\left(\frac{3}{4}\right)^kC^*_1 \e^{-C^*_2 \beta y} \ \ (\text{by \eqref{nl}})
= 4C^*_1 \e^{-C^*_2 \beta y},
\end{align*}
where the second inequality uses the strong Markov property.
This completes the proof of the lemma.
\end{proof}

\section{Proofs of hitting time estimates in the small-$\beta$ regime}\label{app:small-aux}
In this Appendix, we will assume $\beta$ to be sufficiently small in many calculations, often without explicitly mentioning it.
\begin{proof}[Proof of Lemma~\ref{qoneexc-i}]
Assume $(Q_1(0), Q_2(0))=(0,y)$ for any fixed $0<y \le \beta^{-1/2}$. Define the following stopping times: $\Delta_0 = 0$ and for $k \ge 0$,
$$
\Delta_{2k+1} = \inf\{t \ge \Delta_{2k}: Q_1(t) = -\beta^{1/4}\}, \ \ \Delta_{2k+2} = \inf\{t \ge \Delta_{2k+1}: Q_1(t) = -\beta\}.
$$
Define $\mathcal{N}_t = \inf\{k \ge 1 : \Delta_{k} \ge t\}$. Then for any $x \ge 2\beta^{1/4}$,
\begin{eq}\label{qoneexc1}
&\mathbb{E}_{(0,y)}\Big(\int_0^{t}\mathbf{1}_{[Q_1(s) \le -x]}ds\Big) = \sum_{k=0}^{\infty}\mathbb{E}_{(0,y)}\Big(\Big(\int_{\Delta_{2k+1}}^{\Delta_{2k+2}}\mathbf{1}_{[Q_1(s) \le -x]}ds\Big)\mathbf{1}_{[\mathcal{N}_t \ge 2k+2]}\Big)\\
&\le \sum_{k=0}^{\infty}\prob_{(0,y)}(\mathcal{N}_t \ge 2k+2)\sup_{z \le 2\beta^{-1/2}}\mathbb{E}_{(-\beta^{1/4},z)}\Big(\int_{0}^{\tau_1(-\beta)}\mathbf{1}_{[Q_1(s) \le -x]}ds\Big)\\
&\le \mathbf{E}_{(0,y)}(\mathcal{N}_t)\sup_{z \le 2\beta^{-1/2}}\mathbb{E}_{(-\beta^{1/4},z)}\Big(\int_{0}^{\tau_1(-\beta)}\mathbf{1}_{[Q_1(s) \le -x]}ds\Big).
\end{eq}
Now, we use the fact that with starting configuration $(Q_1(0), Q_2(0))= (-\beta^{1/4},z)$, for all $t \le \tau_1(0)$, $Q_1(t) + \beta$ is stochastically bounded below by an Ornstein-Uhlenbeck process $\hat{X}_t$ with $\hat{X}_0 = -\beta^{1/4} + \beta$. Denote by $\hat{\mathbb{P}}_u$ and $\hat{\mathbb{E}}_u$ the probability and expectation under the law of an Ornstein-Uhlenbeck process starting from $u$ and $\hat{\tau}(v)$ the hitting time of level $v$ by $\hat{X}$. Using this, we obtain for any $x \ge 2\beta^{1/4}$,
\begin{eq}\label{qoneexc2}
\mathbb{E}_{(-\beta^{1/4},z)}\Big(\int_{0}^{\tau_1(-\beta)}\mathbf{1}_{[Q_1(s) \le -x]}ds\Big) &\le \hat{\mathbb{E}}_{-\beta^{1/4} + \beta}\Big(\int_{0}^{\hat{\tau}(0)}\mathbf{1}_{[\hat{X}(s) \le -x+\beta]}ds\Big)\\
&\le \hat{\mathbb{E}}_{-\beta^{1/4} + \beta}\left(\mathbf{1}(\hat{\tau}(-x + \beta) < \hat{\tau}(0))(\hat{\tau}(0) - \hat{\tau}(-x + \beta))\right)\\
&= \hat{\mathbb{P}}_{-\beta^{1/4} + \beta}\left(\hat{\tau}(-x + \beta) < \hat{\tau}(0))\right)\hat{\mathbb{E}}_{-x + \beta}(\hat{\tau}(0)).
\end{eq}
Recall that the scale function for the Ornstein-Uhlenbeck process is given by $\hat{s}(u) = \int_0^u \e^{v^2/2}dv$. Using this,
\begin{equation}\label{qoneexc3}
\hat{\mathbb{P}}_{-\beta^{1/4} + \beta}\left(\hat{\tau}(-x + \beta) < \hat{\tau}(0))\right) = \frac{\int_{0}^{\beta^{1/4} - \beta} \e^{v^2/2}dv}{\int_{0}^{x - \beta} \e^{v^2/2}dv} \le C\frac{\beta^{1/4}(x-\beta)}{\e^{(x-\beta)^2/2} - 1}.
\end{equation}
From the Doob representation of Ornstein-Uhlenbeck process, it is straightforward to check that there exists a positive constant $C$ not depending on $x$ such that
\begin{equation}\label{qoneexc4}
\hat{\mathbb{E}}_{-x + \beta}(\hat{\tau}(0)) \le C((x-\beta) \wedge \log(2+(x-\beta)^2)).
\end{equation}
Using \eqref{qoneexc3} and \eqref{qoneexc4} in \eqref{qoneexc2}, we obtain for any $x \ge 2\beta^{1/4}$,
\begin{equation}\label{qoneexc5}
\sup_{z \le 2\beta^{-1/2}}\mathbb{E}_{(-\beta^{1/4},z)}\Big(\int_{0}^{\tau_1(0)}\mathbf{1}_{[Q_1(s) \le -x]}ds\Big) \le \frac{C\beta^{1/4}((x-\beta)^2 \wedge (x-\beta)\log(2+(x-\beta)^2))}{\e^{(x-\beta)^2/2} - 1}.
\end{equation}
Next, we produce an estimate on $\mathbf{E}_{(0,y)}(\mathcal{N}_t)$. Note that for each $k \ge 0$, $\Delta_{2k+1} - \Delta_{2k}$ is stochastically bounded below by the hitting time of level $-\beta^{1/4}$ by the process $t \mapsto -\beta + (\sqrt{2}W(t) - \beta t) - \sup_{s \le t}(\sqrt{2}W(s) - \beta s)$. Denote the hitting time of level $-1$ by the process $t \mapsto -\frac{1}{2} + (\sqrt{2}W(t) - \beta^{5/4} t) - \sup_{s \le t}(\sqrt{2}W(s) - \beta^{5/4} s)$ by $\hat{\Delta}$. By Brownian scaling, choosing sufficiently small $\beta_0$ and $\beta \le \beta_0$, $\Delta_{2k+1} - \Delta_{2k}$ is stochastically bounded below by $\beta^{1/2}\hat{\Delta}$. As $\beta \le 1$, $\hat{\Delta}$, in turn, is stochastically bounded below by the hitting time of level $-1$ by the process $t \mapsto -\frac{1}{2} + (\sqrt{2}W(t) - t) - \sup_{s \le t}(\sqrt{2}W(s) - s)$, which we denote by $\hat \Delta^*$. Let $\{\hat\Delta^*_k\}_{k \ge 0}$ be i.i.d. copies of $\hat \Delta^*$. It is easy to check that $\hat \Delta^*$ is a sub-exponential random variable and thus, using Chernoff's inequality (see \cite[Pg.~16, Equation (2.2)]{Massart07}), we obtain for any $n \ge 2\beta^{-1/2}t/(\mathbb{E}(\hat \Delta^*))$,
\begin{multline*}
\prob_{(0,y)}(\mathcal{N}_t \ge 2n+1) \le \prob_{(0,y)}\Big(\sum_{k=0}^{n}(\Delta_{2k+1}-\Delta_{2k}) < t\Big)
\le \prob_{(0,y)}\Big(\sum_{k=0}^{n}\hat \Delta^*_k < \beta^{-1/2}t\Big)\\
 \le \prob_{(0,y)}\Big(\sum_{k=0}^{n}(\hat \Delta^*_k - \mathbb{E}(\hat \Delta^*_k)) < \beta^{-1/2}t - n\mathbb{E}(\hat \Delta^*_k)\Big) \le C\e^{-C'n}.
\end{multline*}
From this, we immediately obtain for $t \ge \mathbb{E}(\hat \Delta^*)\beta^{1/2}$,
\begin{equation}\label{qoneexc6}
\mathbf{E}_{(0,y)}(\mathcal{N}_t) \le C\beta^{-1/2}t.
\end{equation}
Plugging in the estimates \eqref{qoneexc5} and \eqref{qoneexc6} in \eqref{qoneexc1},
\begin{equation}\label{qonemain1}
\mathbb{E}_{(0,y)}\Big(\int_0^{t}\mathbf{1}_{[Q_1(s) \le -x]}ds\Big) \le C\frac{\beta^{-1/4}t((x-\beta)^2 \wedge (x-\beta)\log(2+(x-\beta)^2))}{\e^{(x-\beta)^2/2} - 1},
\end{equation}
where $C$ does not depend on $\beta, t, y$. Now, observe that for any $n \ge 1$,
\begin{eq}\label{qoneexc7}
&\mathbb{E}_{(0,y)}\Big(\int_0^{\tau_2(2\beta^{-1/2})}\mathbf{1}_{[Q_1(s) \le -x]}ds\Big)\\
& \le \mathbb{E}_{(0,y)}\Big(\int_0^{n\beta^{-1}}\mathbf{1}_{[Q_1(s) \le -x]}ds\Big)
 + \sum_{k=n+1}^{\infty}\mathbb{E}_{(0,y)}\Big(\Big(\int_0^{k\beta^{-1}}\mathbf{1}_{[Q_1(s) \le -x]}ds\Big)\mathbf{1}_{[(k-1)\beta^{-1} \le \tau_2(2\beta^{-1/2}) < k\beta^{-1}]}\Big)\\
& \le \mathbb{E}_{(0,y)}\Big(\int_0^{n\beta^{-1}}\mathbf{1}_{[Q_1(s) \le -x]}ds\Big) + \sum_{k=n+1}^{\infty}k\beta^{-1}\prob_{(0,y)}\Big(\tau_2(2\beta^{-1/2}) \ge (k-1)\beta^{-1}\Big).
\end{eq}
To estimate the probability appearing above, we define the stopping times $\mathcal{S}^*_0$ and for $k \ge 0$,
\begin{align*}
\mathcal{S}^*_{2k+1} &= \inf\{ t \ge \mathcal{S}^*_{2k}: S(t) = 2\beta^{-1/2} \text{ or } S(t) \le -\beta^{-1/2}\},\\
\mathcal{S}^*_{2k+2} &= \inf\{ t \ge \mathcal{S}^*_{2k+1}: S(t) = 2\beta^{-1/2} \text{ or } S(t) = -\beta\}.
\end{align*}
Let $N_{\mathcal{S}^*} = \inf\{k \ge 0: S(\mathcal{S}^*_{2k+1}) = 2\beta^{-1/2}\}$. Then proceeding along the same lines as the proofs of \eqref{smallr4} and \eqref{three0}, we obtain constants $p,q \in (0,1), C>0$ not depending on $\beta, t,y$ such that for any $n \ge 1$,
\begin{equation*}
\mathbb{P}_{(0,y)}(N_{\mathcal{S}^*} \ge n) \le (1-p)^n, \ \ \mathbb{P}_{(0,y)}(\mathcal{S}^*_1 \ge n\beta^{-1}) \le (1-q)^n.
\end{equation*}
To see the second bound above, note that along the lines of \eqref{three0},
\begin{align*}
\prob_{(0, y)}\left(\mathcal{S}^*_1 \ge n \beta^{-1}\right)
&\le \mathbb{E}_{(0, y)}\Big(\mathbf{1}_{[\mathcal{S}^*_1 \ge (n-1)\beta^{-1}]}\prob\Big(\sup_{t \le \beta^{-1}}(\sqrt{2}W(t) - \beta t) < 3\beta^{-1/2} \Big) \Big)\\
&\le \prob_{(0, y)}\left(\mathcal{S}^*_1 \ge (n-1) \beta^{-1}\right)\prob\Big(\sup_{t \le \beta^{-1}}(\sqrt{2}W(t)) < 3\beta^{-1/2} + 1\Big).
\end{align*}
Moreover, using the Doob representation for the Ornstein-Uhlenbeck process which can be used to bound $Q_1 + \beta$ from below, it is straightforward to show that for $n \ge 1$,
\begin{equation*}
\prob_{(0,y)}\left(\mathcal{S}^*_2 - \mathcal{S}^*_1 \ge n\beta^{-1/2}\right) \le C\beta^{-1/2}\e^{-n\beta^{-1/2}} \le C'\e^{-(n-1)\beta^{-1/2}}.
\end{equation*}
Observing that $\tau_2(2\beta^{-1/2}) = \mathcal{S}^*_{2N_{\mathcal{S}^*}+1}$ and using the above estimates, we obtain for any $k, n \ge 1$ satisfying $k \ge 2n+1$,
\begin{align*}
\prob_{(0,y)}\left(\tau_2(2\beta^{-1/2}) \ge k\beta^{-1}\right) &= \prob_{(0,y)}\left(\mathcal{S}^*_{2N_{\mathcal{S}^*}+1} \ge k\beta^{-1}\right)\\
 &\le \mathbb{P}_{(0,y)}(N_{\mathcal{S}^*} \ge n) + \sum_{i=1}^{2n+1}\prob_{(0,y)}\left(\mathcal{S}^*_{i+1} - \mathcal{S}^*_i \ge k\beta^{-1}/(2n+1)\right)\\
 &\le (1-p)^n + (2n+1)\left((1-q)^{k/(2n+1)}\right) + C'\e^{-((k/(2n+1)) - 1)\beta^{-1/2}}.
\end{align*}
Choosing any $k \ge 9$ and taking $n=(\sqrt{k}-1)/2$ in the above,
\begin{equation}\label{qonemain2}
\prob_{(0,y)}\left(\tau_2(2\beta^{-1/2}) \ge k\beta^{-1}\right) \le C\e^{-C'\sqrt{k}}.
\end{equation}
Using \eqref{qonemain1} and \eqref{qonemain2} in \eqref{qoneexc7}, we have positive constants $C, C', C"$ such that for any $y \in (0, \beta^{-1/2}), x \ge 2\beta^{1/4}$ and $n \ge 1$,
\begin{align*}
&\mathbb{E}_{(0,y)}\Big(\int_0^{\tau_2(2\beta^{-1/2})}\mathbf{1}_{[Q_1(s) \le -x]}ds\Big)\\
 &\le C\frac{n((x-\beta)^2 \wedge (x-\beta)\log(2+(x-\beta)^2))}{\beta^{5/4}(\e^{(x-\beta)^2/2} - 1)} + \sum_{k=n+1}^{\infty}k\beta^{-1}C\e^{-C'\sqrt{k}}\\
& \le C\frac{n((x-\beta)^2 \wedge (x-\beta)\log(2+(x-\beta)^2))}{\beta^{5/4}(\e^{(x-\beta)^2/2} - 1)} + C\e^{-C"\sqrt{n}} \le Cn\beta^{-5/4}\e^{-(x-\beta)^2/4} + C\e^{-C"\sqrt{n}}.
\end{align*}
Choosing $n = (1+(x-\beta)^4)/(8C")^2$ in the above, we obtain a positive constants $C$ such that for any $y \in (0, \beta^{-1/2}), x \ge 2\beta^{1/4}$,
\begin{equation}
\mathbb{E}_{(0,y)}\left(\int_0^{\tau_2(2\beta^{-1/2})}\mathbf{1}_{[Q_1(s) \le -x]}ds\right) \le C\beta^{-5/4}\e^{-(x-\beta)^2/8}
\end{equation}
completing the proof of Lemma~\ref{qoneexc-i}.
\end{proof}

\begin{proof}[Proof of Lemma~\ref{qoneexc-ii}]
Let $(Q_1(0), Q_2(0))=(0,4M_0\beta^{-1})$. Define the following stopping times: $\Delta^*_0 = 0$ and for $k \ge 0$,
\begin{align*}
\Delta^*_{2k+1} &= \inf\{t \ge \Delta^*_{2k}: Q_1(t) = -\beta \text{ or } Q_2(t) = 2M_0 \beta^{-1}\},\\
\Delta^*_{2k+2} &= \inf\{t \ge \Delta^*_{2k+1}: Q_1(t) = 0 \text{ or } Q_2(t) = 2M_0 \beta^{-1}\}.
\end{align*}
Define $\mathcal{N}^*_t = \inf\{k \ge 1 : \Delta^*_{k} \ge t \text{ or } Q_2(\Delta^*_k) = 2M_0\beta^{-1}\}$. Observe that for any $x \ge 2\beta$,
\begin{multline}\label{p21}
\sup_{y \ge 2M_0\beta^{-1}}\mathbb{E}_{(-\beta, y)}\left(\int_{0}^{\tau_1(0) \wedge \tau_2(2M_0\beta^{-1})}\mathbf{1}_{[Q_1(s) \le -x]}ds\right)\\
\le  \sup_{y \ge 2M_0\beta^{-1}}\mathbb{P}_{(-\beta,y)}(\tau_1(-x) < \tau_1(0)\wedge \tau_2(2M_0\beta^{-1})) \sup_{y \ge 2M_0\beta^{-1}}\mathbb{E}_{(-x,y)}\left(\tau_1(0) \wedge \tau_2(2M_0\beta^{-1})\right).
\end{multline}
On the time interval $[0, \tau_1(0) \wedge \tau_2(2M_0\beta^{-1})]$, $Q_1$ is stochastically bounded below by the process $t \mapsto \sqrt{2}W(t) + (2M_0 \beta^{-1} - \beta)t$. Using this and scale function arguments we obtain for $\beta \le \beta_0$ for sufficiently small $\beta_0 \in(0,1)$ and $x \ge 2\beta$,
\begin{equation*}
 \sup_{y \ge 2M_0\beta^{-1}}\mathbb{P}_{(-\beta,y)}(\tau_1(-x) < \tau_1(0)\wedge \tau_2(2M_0\beta^{-1})) \le \frac{\e^{(2M_0\beta^{-1} - \beta)\beta} - 1}{\e^{(2M_0\beta^{-1} - \beta)x} - 1} \le C\e^{-C'\beta^{-1}x},
\end{equation*}
where $C,C'$ are positive constants not depending on $\beta, x$. Denoting the hitting time of $0$ by process $t \mapsto -x + \sqrt{2}W(t) + (2M_0 \beta^{-1} - \beta)t$ by $\tau^x$,
$$
\sup_{y \ge 2M_0\beta^{-1}}\mathbb{E}_{(-x,y)}\left(\tau_1(0) \wedge \tau_2(2M_0\beta^{-1})\right) \le \mathbb{E}(\tau^x) \le C"x\beta,
$$
where $C"$ does not depend on $x,\beta$. Using these estimates in \eqref{p21}, we obtain
\begin{equation}\label{p22}
\sup_{y \ge 2M_0\beta^{-1}}\mathbb{E}_{(-\beta, y)}\left(\int_{0}^{\tau_1(0) \wedge \tau_2(2M_0\beta^{-1})}\mathbf{1}_{[Q_1(s) \le -x]}ds\right) \le Cx\beta \e^{-C'\beta^{-1}x}.
\end{equation}
Using a similar argument as that used to derive \eqref{qoneexc6} stochastically bounding $\beta^{-2}(\Delta^*_{2k+1} - \Delta^*_{2k})$ from below by sub-exponential random variables and using Chernoff's inequality, we obtain $n_0>0$ not depending on $\beta,t$ such that for any $n \ge n_0t\beta^{-2}$,
\begin{equation}\label{p23}
\mathbb{P}_{(0,4M_0\beta^{-1})}(\mathcal{N}^*_t \ge 2n+1) \le C'\e^{-C"n}.
\end{equation}
Using \eqref{p23}, and part (ii) of Lemma \ref{lem:q2regeneration} (recalling $c_1'=M_0$ and taking $y=3M_0\beta^{-1}$), we obtain $k_0>0$ such that for all $k \ge k_0$,
\begin{align*}
&\mathbb{P}_{(0,4M_0\beta^{-1})}(\mathcal{N}^*_{\tau_2(2M_0\beta^{-1})} \ge 2(n_0k\beta^{-4})+1)\\
& \le \mathbb{P}_{(0,4M_0\beta^{-1})}(\tau_2(2M_0\beta^{-1}) \ge k\beta^{-2}) + \mathbb{P}_{(0,4M_0\beta^{-1})}(\mathcal{N}^*_{k\beta^{-2}} \ge 2n_0(k\beta^{-2})\beta^{-2}+1)\\
 &\le c'_3\left(\exp(-c'_2\beta^{-2/5}(k\beta^{-2})^{1/5}) + \exp(-c'_2\beta^2 (k\beta^{-2})) + \beta^{-2}\exp(-c'_2 (k\beta^{-2}))\right) + C'\e^{-C"n_0k\beta^{-4}}
\end{align*}
From the above estimate, it follows by summing both sides over $k \ge k_0$ that
\begin{equation}\label{p24}
\mathbb{E}_{(0,4M_0\beta^{-1})}\left(\mathcal{N}^*_{\tau_2(2M_0\beta^{-1})}\right) \le C\beta^{-4}.
\end{equation}
Using \eqref{p22}, \eqref{p24} and the strong Markov property at stopping times $\Delta^*_{2k+1}$ in the upper bound
\begin{multline*}
\mathbb{E}_{(0,4M_0\beta^{-1})}\Big(\int_{0}^{\tau_2(2M_0\beta^{-1})}\mathbf{1}_{[Q_1(s) \le -x]}ds\Big)\\
\le \sum_{k=0}^{\infty}\mathbb{E}_{(0,4M_0\beta^{-1})}\Big(\Big(\int_{ \Delta^*_{2k+1}}^{\Delta^*_{2k+2}}\mathbf{1}_{[Q_1(s) \le -x]}ds\Big)\mathbf{1}_{[\mathcal{N}^*_{\tau_2(2M_0\beta^{-1})} \ge 2k+1]}\Big),
\end{multline*}
we obtain
$$
\mathbb{E}_{(0,4M_0\beta^{-1})}\Big(\int_{0}^{\tau_2(2M_0\beta^{-1})}\mathbf{1}_{[Q_1(s) \le -x]}ds\Big) \le Cx\beta^{-3}\e^{-C'\beta^{-1}x}
$$
for all $x \ge 2\beta$, which proves Lemma~\ref{qoneexc-ii}.
\end{proof}

\begin{proof}[Proof of Lemma~\ref{qoneexc-iii}]
We proceed as in Lemma~\ref{qoneexc-ii} and define the stopping times $\Delta^{**}_0 = 0$ and for $k \ge 0$,
\begin{align*}
\Delta^{**}_{2k+1} &= \inf\{t \ge \Delta^{**}_{2k}: Q_1(t) = -\sqrt{\beta} \text{ or } Q_2(t) = 4 M_0 \beta^{-1} \text{ or } Q_2(t)=\beta^{-1/2}\},\\
\Delta^{**}_{2k+2} &= \inf\{t \ge \Delta^{**}_{2k+1}: Q_1(t) = 0 \text{ or } Q_2(t) =  4M_0\beta^{-1} \text{ or } Q_2(t)=\beta^{-1/2}\}.
\end{align*}
Define $\mathcal{N}^{**}_t = \inf\{k \ge 1 : \Delta^{**}_{k} \ge t \text{ or } Q_2(\Delta^{**}_k) = 4 M_0\beta^{-1} \text{ or } Q_2(t)=\beta^{-1/2}\}$.
By the same argument used in Lemma~\ref{qoneexc-ii} by bounding $Q_1$ from below by a Brownian motion with drift $\beta^{-1/2} - \beta$ for $t \le \tau_1(0) \wedge \tau_2(\beta^{-1/2})$, we can conclude for any $x \ge 2\beta^{1/2}$,
\begin{equation}\label{31}
\sup_{\beta^{-1/2} \le y \le 4M_0\beta^{-1}}\mathbb{E}_{(-\beta^{1/2}, y)}\left(\int_{0}^{\tau_1(0) \wedge \tau_2(\beta^{-1/2})}\mathbf{1}_{[Q_1(s) \le -x]}ds\right) \le Cx\beta^{1/2} \e^{-C'\beta^{-1/2}x}.
\end{equation}
Moreover, using the same approach as the one used to derive \eqref{qonemain2}, for any $k \ge 9$,
\begin{equation}\label{4use}
\sup_{\beta^{-1/2} \le y \le 4M_0\beta^{-1}}\prob_{(0,y)}\left(\tau_2(4M_0\beta^{-1}) \ge k \beta^{-2}\right) \le C\e^{-C'\sqrt{k}}.
\end{equation}
Using a similar argument as that used to derive \eqref{qoneexc6} stochastically bounding $\beta^{-1}(\Delta^{**}_{2k+1} - \Delta^{**}_{2k})$ from below by sub-exponential random variables and using Chernoff's inequality, we obtain $n'_0>0$ not depending on $\beta,t$ such that for any $n \ge n'_0t\beta^{-1}$,
\begin{equation}\label{31.5}
\sup_{\beta^{-1/2} \le y \le 4M_0\beta^{-1}}\mathbb{P}_{(0,y)}(\mathcal{N}^{**}_t \ge 2n+1) \le C'\e^{-C"n}.
\end{equation}
Using \eqref{4use} and \eqref{31.5} and the calculation leading to \eqref{p24}, we obtain
\begin{equation}\label{32}
\sup_{\beta^{-1/2} \le y \le 4M_0\beta^{-1}}\mathbb{E}_{(0,y)}\left(\mathcal{N}^*_{\tau_2(\beta^{-1/2}) \wedge \tau_2(4M_0\beta^{-1})}\right) \le C\beta^{-3}.
\end{equation}
Using \eqref{31}, \eqref{32} and the strong Markov property at stopping times $\Delta^{**}_{2k+1}$ in the upper bound
\begin{multline*}
\sup_{\beta^{-1/2} \le y \le 4M_0\beta^{-1}}\mathbb{E}_{(0,y)}\Big(\int_{0}^{\tau_2(\beta^{-1/2}) \wedge \tau_2(4M_0\beta^{-1})}\mathbf{1}_{[Q_1(s) \le -x]}ds\Big)\\
\le \sum_{k=0}^{\infty}\mathbb{E}_{(0,4M_0\beta^{-1})}\Big(\Big(\int_{ \Delta^{**}_{2k+1}}^{\Delta^{**}_{2k+2}}\mathbf{1}_{[Q_1(s) \le -x]}ds\Big)\mathbf{1}_{[\mathcal{N}^*_{\tau_2(\beta^{-1/2}) \wedge \tau_2(4M_0\beta^{-1})} \ge 2k+1]}\Big),
\end{multline*}
we obtain
$$
\sup_{\beta^{-1/2} \le y \le 4M_0\beta^{-1}}\mathbb{E}_{(0,y)}\Big(\int_{0}^{\tau_2(\beta^{-1/2}) \wedge \tau_2(4M_0\beta^{-1})}\mathbf{1}_{[Q_1(s) \le -x]}ds\Big) \le Cx\beta^{-5/2}\e^{-C'\beta^{-1/2}x}
$$
for all $x \ge 2\beta^{1/2}$, which completes the proof of Lemma~\ref{qoneexc-iii}.
\end{proof}

\begin{proof}[Proof of Lemma~\ref{qoneexc-iv}]
Write
$$
I_{\beta}(x) = \int_{\tau_2(\beta^{-1/2})}^{\sigma\left(\tau_2(\beta^{-1/2})\right)}\mathbf{1}_{[Q_1(s) \le -x]}ds.
$$
Observe that for any $x>0$, $y \in [\beta^{-1/2}, 4M_0\beta^{-1}]$,
\begin{multline}\label{41}
\mathbb{E}_{(0,y)}\left(I_{\beta}(x)\mathbf{1}_{[\tau_2(\beta^{-1/2}) < \tau_2(4M_0\beta^{-1})]}\right)
\le \mathbb{E}_{(0,y)}\left(I_{\beta}(x)\mathbf{1}_{[Q_1(\tau_2(\beta^{-1/2})) \ge -x/2]}\right)\\
+ \mathbb{E}_{(0,y)}\left(I_{\beta}(x)\mathbf{1}_{[Q_1(\tau_2(\beta^{-1/2})) < -x/2, \ \tau_2(\beta^{-1/2}) < \tau_2(4M_0\beta^{-1}) ]}\right).
\end{multline}
To estimate the first term above, apply the strong Markov property at $\tau_2(\beta^{-1/2})$ and recall that $Q_1 + \beta$ is bounded below by an Ornstein-Uhlenbeck process $\hat{X}$ for $t \le \tau_1(0)$. Using this observation and proceeding as in the proof of \eqref{qoneexc5}, we obtain for any $x \ge 4\beta$,
\begin{eq}\label{42}
&\mathbb{E}_{(0,y)}\left(I_{\beta}(x)\mathbf{1}_{[Q_1(\tau_2(\beta^{-1/2})) \ge -x/2]}\right)\\ 
&\hspace{2cm}\le \sup_{z \in [-x/2,0]}\mathbb{E}_{(z,\beta^{-1/2})}\left(\int_{0}^{\tau_1(0)}\mathbf{1}_{[Q_1(s) \le -x]}ds\right)\\
&\hspace{2cm}\le  \sup_{z \in [-x/2,0]}\hat{\mathbb{E}}_{z + \beta}\left(\int_{0}^{\hat{\tau}(\beta)}\mathbf{1}_{[\hat{X}(s) \le -x+\beta]}ds\right)\\
&\hspace{2cm}\le \sup_{z \in [-x/2,0]}\hat{\mathbb{E}}_{z + \beta}\left(\mathbf{1}(\hat{\tau}(-x + \beta) < \hat{\tau}(\beta))(\hat{\tau}(\beta) - \hat{\tau}(-x + \beta))\right)\\
&\hspace{2cm}= \sup_{z \in [-x/2,0]}\hat{\mathbb{P}}_{z + \beta}\left(\hat{\tau}(-x + \beta) < \hat{\tau}(\beta))\right)\hat{\mathbb{E}}_{-x + \beta}(\hat{\tau}(\beta))
\le C'\e^{-(x-2\beta)^2/4}.
\end{eq}
In the above, we used
\begin{align*}
\sup_{z \in [-x/2,0]}\hat{\mathbb{P}}_{z + \beta}\left(\hat{\tau}(-x + \beta) < \hat{\tau}(\beta))\right) &\le \frac{\int_{-\beta}^{\frac{x}{2} - \beta}\e^{v^2/2}dv}{\int_{-\beta}^{x - \beta}\e^{v^2/2}dv} \le \frac{x(x-\beta)\e^{(\frac{x}{2} - \beta)^2/2}}{2(\e^{(x - \beta)^2/2} - 1)}\\
& \le \frac{2(x-\beta)^2\e^{-3(x - 2\beta)^2/8}}{3(1-\e^{-(x - \beta)^2/2})} \le C\e^{-3(x - 2\beta)^2/8}.
\end{align*}
To estimate $\hat{\mathbb{E}}_{-x + \beta}(\hat{\tau}(\beta))$, we decompose the path of $\hat{X}$ on $[0, \hat{\tau}(\beta)]$ into excursions: the first one from $-x+\beta$ to $0$, and then from $0$ to $\pm\beta$ and then $\pm \beta$ to $0$ until the first time $\beta$ is hit. Using this decomposition and standard estimates on Ornstein-Uhlenbeck processes, we obtain
$$
\hat{\mathbb{E}}_{-x + \beta}(\hat{\tau}(\beta)) \le C[(x-\beta) \wedge \log(2+ (x-\beta)^2) + \beta^2].
$$
The calculations are analogous to the ones used repeatedly in the article and we omit the details.

To estimate the second term in the right hand side of \eqref{41}, note that on the event $\{\tau_2(\beta^{-1/2}) < \tau_2(4M_0\beta^{-1})\}$, $\sigma(\tau_2(\beta^{-1/2})) < \tau_2(4M_0\beta^{-1})$. Thus, on this event, $I_{\beta}(x) \le \tau_2(4M_0\beta^{-1})$. Using this observation and the Cauchy-Schwarz inequality,
\begin{eq}\label{444}
&\mathbb{E}_{(0,y)}\left(I_{\beta}(x)\mathbf{1}_{[Q_1(\tau_2(\beta^{-1/2})) < -x/2, \ \tau_2(\beta^{-1/2}) < \tau_2(4M_0\beta^{-1}) ]}\right)\\
&\le \mathbb{E}_{(0,y)}\left(\tau_2(4M_0\beta^{-1})\mathbf{1}_{[Q_1(\tau_2(\beta^{-1/2})) < -x/2, \ \tau_2(\beta^{-1/2}) < \tau_2(4M_0\beta^{-1}) ]}\right)\\
&\le \left(\mathbb{E}_{(0,y)}\left(\tau_2(4M_0\beta^{-1})\right)^2\right)^{1/2}\left(\mathbb{P}_{(0,y)}\left(Q_1(\tau_2(\beta^{-1/2})) < -x/2, \ \tau_2(\beta^{-1/2}) < \tau_2(4M_0\beta^{-1})\right)\right)^{1/2}
\end{eq}
From \eqref{4use},
\begin{equation}\label{43}
\sup_{\beta^{-1/2} \le y \le 4M_0\beta^{-1}}\mathbb{E}_{(0,y)}\left(\tau_2(4M_0\beta^{-1})\right)^2 \le C\beta^{-4}.
\end{equation}
Using the stopping times $\Delta^{**}_k$ defined in the proof of Lemma~\ref{qoneexc-iii}, for any $y \in [\beta^{-1/2}, 4M_0\beta^{-1}]$ and any $x \ge 4\sqrt{\beta}$,
\begin{eq}\label{44}
&\mathbb{P}_{(0,y)}\left(Q_1(\tau_2(\beta^{-1/2})) < -x/2, \ \tau_2(\beta^{-1/2}) < \tau_2(4M_0\beta^{-1})\right)\\
& \le \mathbb{E}_{(0,y)}\Big(\sum_{k=0}^{\infty}\mathbf{1}_{[\Delta^{**}_{2k+1} < \tau_1(-x/2) < \Delta^{**}_{2k+2}, \ \tau_2(\beta^{-1/2}) < \tau_2(4M_0\beta^{-1})]}\mathbf{1}_{[\mathcal{N}^*_{\tau_2(\beta^{-1/2}) \wedge \tau_2(4M_0\beta^{-1})} \ge 2k+2]}\Big)\\
& \le \sup_{\beta^{-1/2} \le z \le 4M_0\beta^{-1}}\mathbb{P}_{(-\beta^{1/2}, z)}\left(\tau_1(-x/2) < \tau_1(0) \wedge \tau_2(\beta^{-1/2})\right)\mathbb{E}_{(0,y)}\left(\mathcal{N}^*_{\tau_2(\beta^{-1/2}) \wedge \tau_2(4M_0\beta^{-1})}\right)\\
& \le C\beta^{-3}\e^{-C'\beta^{-1/2}x},
\end{eq}
where for the bound on the last line, we used \eqref{32} and the fact that starting from $(Q_1(0), Q_2(0)) = (-\beta^{1/2},z)$, $\beta^{-1/2} \le z \le 4M_0\beta^{-1}$, $Q_1$ is bounded from below by a Brownian motion with drift $\beta^{-1/2} - \beta$ for $t \le \tau_1(0) \wedge \tau_2(\beta^{-1/2})$. Also, $C,C'$ appearing in the above bound do not depend on $y$.
Using \eqref{43} and \eqref{44} in \eqref{444},
\begin{equation}\label{45}
\sup_{\beta^{-1/2} \le y \le 4M_0\beta^{-1}}\mathbb{E}_{(0,y)}\left(I_{\beta}(x)\mathbf{1}_{[Q_1(\tau_2(\beta^{-1/2})) < -x/2, \ \tau_2(\beta^{-1/2}) < \tau_2(4M_0\beta^{-1}) ]}\right) \le C\beta^{-7/2}\e^{-C'\beta^{-1/2}x}.
\end{equation}
Using \eqref{42} and \eqref{45} in \eqref{41} completes the proof of Lemma~\ref{qoneexc-iv}.
\end{proof}

\begin{proof}[Proof of Lemma~\ref{qoneexc-v}]
We will proceed similarly as in Lemma~\ref{qoneexc-iv}. Write
$$
J_{\beta}(x) = \int_{\tau_2(2M_0\beta^{-1})}^{\sigma\left(\tau_2(2M_0\beta^{-1})\right)}\mathbf{1}_{[Q_1(s) \le -x]}ds.
$$
\begin{multline}\label{51}
\mathbb{E}_{(0,4M_0\beta^{-1})}\left(J_{\beta}(x)\right)
\le \mathbb{E}_{(0,4M_0\beta^{-1})}\left(J_{\beta}(x)\mathbf{1}_{[Q_1(\tau_2(2M_0\beta^{-1})) \ge -x/2]}\right)\\
+ \mathbb{E}_{(0,4M_0\beta^{-1})}\left(J_{\beta}(x)\mathbf{1}_{[Q_1(\tau_2(2M_0\beta^{-1})) < -x/2]}\right).
\end{multline}
The first term is estimated as in \eqref{42} yielding for any $x \ge 4\beta$,
\begin{equation}\label{52}
\mathbb{E}_{(0,y)}\left(J_{\beta}(x)\mathbf{1}_{[Q_1(\tau_2(2M_0\beta^{-1})) \ge -x/2]}\right) \le C'\e^{-(x-2\beta)^2/4}.
\end{equation}
To estimate the second term, recall $\alpha_1 = \tau_2(2M_0\beta^{-1})$. Note that by Cauchy-Schwarz inequality and strong Markov property,
\begin{eq}\label{53}
&\mathbb{E}_{(0,4M_0\beta^{-1})}\left(J_{\beta}(x)\mathbf{1}_{[Q_1(\tau_2(2M_0\beta^{-1})) < -x/2]}\right)\\
 &\le \left(\mathbb{E}_{(0,4M_0\beta^{-1})}(\sigma(\alpha_1) - \alpha_1)^2\right)^{1/2}\left(\mathbb{P}_{(0,4M_0\beta^{-1})}\left(Q_1(\tau_2(2M_0\beta^{-1})) < -x/2\right)\right)^{1/2}\\
 &=\left(\mathbb{E}_{(0,4M_0\beta^{-1})}\left(\mathbb{E}_{(Q_1(\alpha_1),2M_0\beta^{-1})}(\tau_1(0))^{2}\right)\right)^{1/2}\left(\mathbb{P}_{(0,4M_0\beta^{-1})}\left(Q_1(\tau_2(2M_0\beta^{-1})) < -x/2\right)\right)^{1/2}
\end{eq}
To estimate the first term in the product above, we again bound $Q_1 + \beta$ from below by an Ornstein-Uhlenbeck process $\hat{X}$ for $t \le \tau_1(0)$. For any $u > 0$, decompose the path of $\hat{X}$ starting from $-u$ on $[0, \hat{\tau}(\beta)]$ into excursions: the first one from $-x+\beta$ to $0$, and then from $0$ to $\pm\beta$ and from $\pm \beta$ to $0$ until the first time $\beta$ is hit. The number of excursions is distributed as $1 + \operatorname{Geometric}(1/2)$. Using this in a standard calculation to obtain the second moment (we omit the details), we obtain for sufficiently small $\beta$,
$$
\hat{\mathbb{E}}_{-u}\left(\hat{\tau}(\beta)\right)^2 \le C(u \wedge (\log(2+u^2))^2 + \beta).
$$
Using this and the estimate for $\mathbb{E}_{(0,4M_0\beta^{-1})}(-Q_1(\alpha_1))$ obtained in \eqref{three5}, we deduce for sufficiently small $\beta$,
\begin{multline}\label{54}
\mathbb{E}_{(0,4M_0\beta^{-1})}\left(\mathbb{E}_{(Q_1(\alpha_1),2M_0\beta^{-1})}(\tau_1(0))^{2})\right)\\
 \le C\mathbb{E}_{(0,4M_0\beta^{-1})}\left(|Q_1(\alpha_1) + \beta| \wedge (\log(2 + |Q_1(\alpha_1) + \beta|^2))^2 + \beta\right) \le C'\beta^{-4}.
\end{multline}
To estimate the second term in the product \eqref{52}, we proceed similarly as in \eqref{44}, but now using the stopping times $\Delta^{*}_k$ defined in the proof of Lemma~\ref{qoneexc-ii}. For any $x \ge 4\beta$,
\begin{eq}\label{55}
&\mathbb{P}_{(0,4M_0\beta^{-1})}\left(Q_1(\tau_2(2M_0\beta^{-1})) < -x/2\right)\\
& \le \mathbb{E}_{(0,4M_0\beta^{-1})}\Big(\sum_{k=0}^{\infty}\mathbf{1}_{[\Delta^{*}_{2k+1} < \tau_1(-x/2) < \Delta^{*}_{2k+2}]}\mathbf{1}_{[\mathcal{N}^*_{\tau_2(2M_0\beta^{-1})} \ge 2k+2]}\Big)\\
 &\le \sup_{z \ge 2M_0\beta^{-1}}\mathbb{P}_{(-\beta, z)}\left(\tau_1(-x/2) < \tau_1(0) \wedge \tau_2(2M_0\beta^{-1})\right)\mathbb{E}_{(0,y)}\left(\mathcal{N}^*_{\tau_2(2M_0\beta^{-1})}\right)
 \le C\beta^{-4}\e^{-C'\beta^{-1}x},
\end{eq}
where for the bound on the last line, we used \eqref{p24} and the fact that starting from $(Q_1(0), Q_2(0)) = (-\beta,z)$, $z \ge 2M_0\beta^{-1}$, $Q_1$ is bounded from below by a Brownian motion with drift $2M_0\beta^{-1} - \beta$ for $t \le \tau_1(0) \wedge \tau_2(2M_0\beta^{-1})$.

Using \eqref{54} and \eqref{55} in \eqref{53},
\begin{equation}\label{56}
\mathbb{E}_{(0,4M_0\beta^{-1})}\left(J_{\beta}(x)\mathbf{1}_{[Q_1(\tau_2(2M_0\beta^{-1})) < -x/2]}\right) \le C\beta^{-4}\e^{-C'\beta^{-1}x}.
\end{equation}
Using \eqref{52} and \eqref{56} in \eqref{51} completes the proof of (v).
\end{proof}

\section{Proof of Lemma \ref{smallbetafluc}}
\label{app:lem5.6}

We begin with the following estimate, which is a consequence of Lemma \ref{Q2gebeta2}.
\begin{lemma}\label{smallint}
There exist positive constants $C_S, C'_S, C''_S$ such that for any $\beta \in (0,\e^{-1})$,
\begin{equation*}
\prob_{(0,2C_S\beta^{-1} \log \beta^{-1})}\left(\tau_2(y) < \tau_2(C_S\beta^{-1} \log \beta^{-1})\right) \le C'_S \e^{-C''_S \beta y}
\end{equation*}
for all $y \ge 4C_S\beta^{-1} \log \beta^{-1}$.
\end{lemma}
\begin{proof}
Take $C_S = R^- + 1, C'_S = C^*_1$ and $C''_S = C^*_2$ (where $R^-, C^*_1, C^*_2$ are the constants appearing in Lemma \ref{Q2gebeta2}). $\beta < \e^{-1}$ ensures $\beta^{-1} \log \beta^{-1} > \beta^{-1}$. Write $z= (y-\beta)/2$. Then $z + \beta =(y+\beta)/2 \ge 2C_S\beta^{-1} \log \beta^{-1}$ for all $y \ge 4C_S\beta^{-1} \log \beta^{-1}$. Therefore, we can derive the lemma from Lemma \ref{Q2gebeta2} by noting that $y_0(\beta)$ defined in the lemma satisfies $y_0(\beta) + \beta \le C_S\beta^{-1} \log \beta^{-1}$ and observing that for all $y \ge 4C_S\beta^{-1} \log \beta^{-1}$,
\begin{align*}
&\prob_{(0,2C_S\beta^{-1} \log \beta^{-1})}\left(\tau_2(y) < \tau_2(C_S\beta^{-1} \log \beta^{-1})\right)\\ 
&\hspace{2cm}= \prob_{(0,2C_S\beta^{-1} \log \beta^{-1})}\left(\tau_2(2z + \beta) < \tau_2(C_S\beta^{-1} \log \beta^{-1})\right)\\
& \hspace{2cm}\le \prob_{(0,2C_S\beta^{-1} \log \beta^{-1})}\left(\tau_2(2z + \beta) < \tau_2(y_0(\beta) + \beta)\right)\\
&\hspace{2cm}\le  \prob_{(0,z+\beta)}\left(\tau_2(2z + \beta) < \tau_2(y_0(\beta) + \beta)\right) \le C'_S \e^{-C''_S \beta y},
\end{align*}
where the second to last inequality follows from the strong Markov property and the last one follows from Lemma \ref{Q2gebeta2}.
\end{proof}
The following lemma gives an estimate analogous to that in Lemma \ref{smallint} in the region $y \in [8M_0 \beta^{-1}, 4C_S\beta^{-1} \log \beta^{-1}]$, where $M_0 > 0 $ does not depend on $\beta$.
\begin{lemma}\label{smallsmall}
Recall the constant $C_S$ defined in Lemma \ref{smallint}. There exist positive constants $M_0, C_S^1, C_S^2, \beta^{S'}_0$ such that for all $\beta \le \beta^{S'}_0$ and for all $y \in [8M_0 \beta^{-1}, 4C_S\beta^{-1} \log \beta^{-1}]$,
$$
\prob_{(0,y/2)}\left(\tau_2(y) < \tau_2(2M_0 \beta^{-1})\right) \le C_S^1 \e^{-C_S^2 \beta y}.
$$ 
\end{lemma}
\begin{proof}
We recall from Lemma~\ref{lem:q1integral} that for $\beta \le 1$, there exist positive constants $c'_1, c'_2, c^*_3$ not depending on $\beta$ such that for any $y > c'_1\beta^{-1}$,
$$\prob_{(0, y)}\Big(\int_0^t(-Q_1(s))\dif s>\frac{\beta t}{2},\ \inf_{s\leq t}Q_2(s) \geq c'_1\beta^{-1} + \beta\Big)\leq \exp(-c'_2t^{1/5}\beta^{-2/5})$$
for $t \ge c^*_3\beta^{2}$. Take $M_0 = c_1'$. By the explicit choice of constants made in Lemma~\ref{lem:q1integral}, $c'_1, c'_2$ are the same constants as the ones appearing in Lemma \ref{lem:q2regeneration}. Write $\mathcal{I}_t = \int_0^t(-Q_1(s))\dif s$ and $t_k = k \beta^2$, $k = 0,1,2,\dots$. Take any $k \ge (\beta y)^5$. Note that if $\frac{k}{k+1} \ge \frac{2}{3}$, then the event $\{\mathcal{I}_{t} > 3\beta t/4 \text{ for some } t \in [t_k, t_{k+1}]\}$ implies
$$
\mathcal{I}_{t_{k+1}} > 3\beta t_k/4 \ge \beta t_{k+1}/2.
$$
Moreover, there exists $\beta_1>0$ such that for all $\beta \le \beta_1$, the event $\{\tau_2(2M_0 \beta^{-1}) > t \text{ for some } t \in [t_k, t_{k+1}]\}$ implies
$$
Q_2(t) \ge 2M_0 \beta^{-1} \e^{-\beta^2} > M_0\beta^{-1} + \beta
$$
for all $t \in [\tau_2(2M_0 \beta^{-1}), t_{k+1}]$ which, in turn, implies $\tau_2(M_0 \beta^{-1} + \beta) > t_{k+1}$. Therefore, there exists $k_1 \ge 1$ and $\beta_1>0$ such that for all $y \in [8M_0 \beta^{-1}, 4C_S\beta^{-1} \log \beta^{-1}]$ and all $\beta \le \beta_1$,
\begin{eq}\label{driftunif}
&\prob_{(0,y/2)}\left(\mathcal{I}_t > 3\beta t/4 \text{ and } \tau_2(2M_0 \beta^{-1}) > t \text{ for some } t  \ge k_1(\beta y)^5\beta^2 \right)\\
&\le \sum_{k=\lfloor k_1(\beta y)^5\rfloor}^{\infty}\prob_{(0,y/2)}\left(\mathcal{I}_t > 3\beta t/4 \text{ and } \tau_2(2M_0 \beta^{-1}) > t \text{ for some } t  \in [t_k, t_{k+1}] \right)\\
&\le \sum_{k=\lfloor k_1(\beta y)^5\rfloor}^{\infty}\prob_{(0,y/2)}\left(\mathcal{I}_{t_{k+1}} > \beta t_{k+1}/2 \text{ and } \tau_2(M_0 \beta^{-1} + \beta) \ge t_{k+1}\right)\\
&\le \sum_{k=\lfloor k_1(\beta y)^5\rfloor}^{\infty}\e^{-c_2'k^{1/5}} \le \e^{-c_2''\beta y}
\end{eq}
for some positive constant $c_2''$ that does not depend on $\beta$. For any $y \in [8M_0 \beta^{-1}, 4C_S\beta^{-1} \log \beta^{-1}]$,
\begin{align*}
&\prob_{(0,y/2)}\left(\tau_2(y) < \tau_2(2M_0 \beta^{-1})\right) \\
&\le \prob_{(0,y/2)}\left(\tau_2(y) \le k_1(\beta y)^5\beta^2\right)\\
&+ \prob_{(0,y/2)}\left(\mathcal{I}_t > 3\beta t/4 \text{ and } \tau_2(2M_0 \beta^{-1}) > t \text{ for some } t  \ge k_1(\beta y)^5\beta^2 \right)\\
&+ \prob_{(0,y/2)}\Big(k_1(\beta y)^5\beta^2 < \tau_2(y) < \tau_2(2M_0 \beta^{-1}), \ \mathcal{I}_t \le 3\beta t/4 \text{ for all } t \in [k_1(\beta y)^5\beta^2, \tau_2(2M_0 \beta^{-1})]\Big).
\end{align*}
Note that with starting configuration $(Q_1(0), Q_2(0)) = (0,y/2)$, $\tau_2(y)$ in the above expression also corresponds to the hitting time of the level $y$ by the sum $S(t) = Q_1(t) + Q_2(t)= S(0) + \sqrt{2}W(t) - \beta t + \mathcal{I}_t = \frac{y}{2} + \sqrt{2}W(t) - \beta t + \mathcal{I}_t$. Further, note that as $y \le 4C_S\beta^{-1} \log \beta^{-1}$, $k_1(\beta y)^5\beta^2 \le k_1 (4C_S)^5\beta^2 \left(\log \beta^{-1}\right)^5$. Thus,
\begin{eq}\label{fund}
&\prob_{(0,y/2)}\left(\tau_2(y) < \tau_2(2M_0 \beta^{-1})\right)\\
& \le \prob_{(0,y/2)}\Big(\sup_{t \le k_1 (4C_S)^5\beta^2 \left(\log \beta^{-1}\right)^5}S(t) > y\Big)\\
&\hspace{3cm}+ \prob_{(0,y/2)}\left(\mathcal{I}_t > 3\beta t/4 \text{ and } \tau_2(2M_0 \beta^{-1}) > t \text{ for some } t  \ge k_1(\beta y)^5\beta^2 \right)\\
&\hspace{3cm}+ \prob_{(0,y/2)}\left(\sup_{t < \infty} (\sqrt{2}W(t) - \beta t/4) > y/2\right)\\
&\le \prob_{(0,y/2)}\Big(\sup_{t \le k_1 (4C_S)^5\beta^2 \left(\log \beta^{-1}\right)^5}S(t) > y\Big) + \e^{-c_2''\beta y} + \e^{-\beta y/8},
\end{eq}
where the last line is a consequence of \eqref{driftunif} and the fact that $\sup_{t < \infty} (\sqrt{2}W(t) - \beta t/4)$ is exponentially distributed with mean $4\beta^{-1}$. To complete the proof, we need to estimate the first probability on the right hand side above. To do this, first note that
$$
Q_1(t) \ge \sqrt{2}W(t) - \beta t - \sup_{s \le t} (\sqrt{2}W(s) - \beta s).
$$
Therefore, there exists $\beta_2>0$ such that for any $\beta \le \beta_2$ and any $y \in [8M_0 \beta^{-1}, 4C_S\beta^{-1} \log \beta^{-1}]$,
\begin{align*}
&\prob_{(0,y/2)}\Big(\inf_{t \le k_1 (4C_S)^5\beta^2 \left(\log \beta^{-1}\right)^5} Q_1(t) < -\sqrt{\beta}\Big)\\
& \le \prob_{(0,y/2)}\Big(\inf_{t \le k_1 (4C_S)^5\beta^2 \left(\log \beta^{-1}\right)^5} (\sqrt{2}W(t) - \beta t) < -\sqrt{\beta}/2\Big)\\
& \hspace{2cm}+ \prob_{(0,y/2)}\Big(\sup_{t \le k_1 (4C_S)^5\beta^2 \left(\log \beta^{-1}\right)^5} (\sqrt{2}W(t) - \beta t) > \sqrt{\beta}/2\Big)\\
& \le \prob_{(0,y/2)}\Big(\inf_{t \le k_1 (4C_S)^5\beta^2 \left(\log \beta^{-1}\right)^5} (\sqrt{2}W(t)) < -\sqrt{\beta}/4\Big)\\
& \hspace{2cm}+ \prob_{(0,y/2)}\Big(\sup_{t \le k_1 (4C_S)^5\beta^2 \left(\log \beta^{-1}\right)^5} (\sqrt{2}W(t) ) > \sqrt{\beta}/4\Big)
  \le \e^{-\frac{1}{32 k_1 (4C_S)^5\beta (\log \beta^{-1})^5}}.
\end{align*}
If $\inf_{t \le k_1 (4C_S)^5\beta^2 \left(\log \beta^{-1}\right)^5} Q_1(t) \ge -\sqrt{\beta}$, then for any $t \le  k_1 (4C_S)^5\beta^2 \left(\log \beta^{-1}\right)^5$,
$$
S(t) \le S(0) + \sqrt{2}W(t) - \beta t + k_1 (4C_S)^5\beta^{5/2} \left(\log \beta^{-1}\right)^5.
$$
Therefore, there exists $\beta_3 >0$ such that for all $\beta \le \beta_3$ and any $y \in [8M_0 \beta^{-1}, 4C_S\beta^{-1} \log \beta^{-1}]$,
\begin{eq}\label{firstprob}
&\prob_{(0,y/2)}\Big(\sup_{t \le k_1 (4C_S)^5\beta^2 \left(\log \beta^{-1}\right)^5}S(t) > y\Big)\\
&\le \prob_{(0,y/2)}\Big(\inf_{t \le k_1 (4C_S)^5\beta^2 \left(\log \beta^{-1}\right)^5} Q_1(t) < -\sqrt{\beta}\Big) + \prob\Big(\sup_{t < \infty}\left(\sqrt{2}W(t) - \beta t\right) > y/4\Big)\\
&\le \e^{-\frac{1}{32 k_1 (4C_S)^5\beta (\log \beta^{-1})^5}} + \e^{-\beta y/ 4} \le 2 \e^{-\beta y/4},
\end{eq}
where in the last inequality, we used the information that $y \le 4C_S\beta^{-1} \log \beta^{-1}$. Using \eqref{firstprob} in \eqref{fund}, the proof of the lemma is completed by choosing $\beta^{S'}_0 = \min\{\beta_1, \beta_2, \beta_3\}$.
\end{proof}
Now we ``patch up" the estimates obtained in Lemma \ref{smallint} and Lemma \ref{smallsmall} to prove Lemma~\ref{smallbetafluc}.
\begin{proof}[Proof of Lemma~\ref{smallbetafluc}]
Choose $M_0$ as in Lemma \ref{smallsmall}. If $y \in [8M_0 \beta^{-1}, 4C_S\beta^{-1} \log \beta^{-1}]$, then the bound is furnished by Lemma \ref{smallsmall}. Therefore, it suffices to consider $y > 4C_S\beta^{-1} \log \beta^{-1}$. Define the following stopping times: $S_0 = 0$ and for $k \ge 0$,
\begin{align*}
S_{2k+1} &= \inf\{t \ge S_{2k}: Q_2(t) = 2M_0\beta^{-1} \text{ or } Q_2(t) = 2C_S\beta^{-1} \log \beta^{-1}\},\\
S_{2k+2} &= \inf\{t \ge S_{2k+1}: Q_2(t) = 2M_0\beta^{-1} \text{ or } Q_2(t) = C_S\beta^{-1} \log \beta^{-1}\}.
\end{align*}
Let $N_S = \inf\{ k \ge 0: Q_2(S_{2k+1}) = 2M_0\beta^{-1}$\}.
Note that by Lemma \ref{smallsmall}, there exists $\beta_1^*>0$ such that for all $\beta \le \beta_1^*$,
\begin{multline*}
\sup_{x \ge 0} \prob_{(-x, C_S\beta^{-1} \log \beta^{-1})}\left(\tau_2(2C_S\beta^{-1} \log \beta^{-1}) < \tau_2(2M_0\beta^{-1})\right)\\
 \le \prob_{(0, C_S\beta^{-1} \log \beta^{-1})}\left(\tau_2(2C_S\beta^{-1} \log \beta^{-1}) < \tau_2(2M_0\beta^{-1})\right) \le C_S^1\e^{-2C_S^2 C_S \log \beta^{-1}} < \frac{1}{2},
\end{multline*}
where the first inequality above follows from the strong Markov property applied at the stopping time $\inf\{t > 0: Q_2(t) = C_S\beta^{-1} \log \beta^{-1}\}$. This immediately gives us
\begin{equation}\label{ns}
\mathbb{E}_{(0,4M_0\beta^{-1})}\left(N_S\right) = \sum_{k=0}^{\infty} \prob_{(0,4M_0\beta^{-1})}(N_S \ge k) \le 1 + \sum_{k=1}^{\infty}2^{-k} = 2.
\end{equation}
For $y > 4C_S\beta^{-1} \log \beta^{-1}$, by applying the strong Markov property at $S_{2k+1}$ for $k \ge 0$ and using Lemma \ref{smallint}, we obtain $\beta_2^*>0$ such that for all $\beta \le \beta_2^*$,
\begin{align*}
&\prob_{(0,4M_0\beta^{-1})}\left(\tau_2(y) < \tau_2(2M_0\beta^{-1})\right)\\
&\hspace{2cm} \le \sum_{k=0}^{\infty}\prob_{(0,4M_0\beta^{-1})}\left(\sup_{t \in [S_{2k+1}, S_{2k+2}]} Q_2(t) > y, N_S > k\right)\\
&\hspace{2cm} \le \sum_{k=0}^{\infty}\mathbb{E}_{(0,4M_0\beta^{-1})}\mathbf{1}_{[N_S >k]}\prob_{(0,2C_S\beta^{-1} \log \beta^{-1})}\left(\tau_2(y) < \tau_2(C_S\beta^{-1} \log \beta^{-1})\right)\\
&\hspace{2cm} \le \mathbb{E}_{(0,4M_0\beta^{-1})}\left(N_S\right)C'_S \e^{-C''_S \beta y} \le 2C'_S \e^{-C''_S \beta y}.
\end{align*}
This completes the proof of the lemma by choosing $\beta_0 = \min\{\beta_1^*, \beta_2^*\}$.
\end{proof}

{\small
\bibliographystyle{plain}
\bibliography{bibl}
}

\end{document}